\documentclass[11pt,a4paper]{article}
\usepackage{amsmath, amssymb, graphics, setspace}
\usepackage{empheq}
\usepackage{amssymb,amsbsy,amsmath,amsfonts,amssymb,amscd,amsthm, mathrsfs, bbm, bigints}
\usepackage{amssymb}
\usepackage{color}
\usepackage{hyperref}

\newcommand{\1}{\mathbbm{1}}
\newcommand{\CC}{\mathbb{C}}
\newcommand{\NN}{\mathbb{N}}
\newcommand{\RR}{\mathbb{R}}

\newtheorem{theo}{Theorem}
\newtheorem{prop}[theo]{Proposition}
\newtheorem{lem}[theo]{Lemma}
\newtheorem{cor}[theo]{Corollary}
\newtheorem{rem}[theo]{Remark}

\newcommand{\beqn}{\begin{equation}}
\newcommand{\eeqn}{\end{equation}}
\newcommand{\bear}{\begin{eqnarray}}
\newcommand{\eear}{\end{eqnarray}}
\newcommand{\bean}{\begin{eqnarray*}}
\newcommand{\eean}{\end{eqnarray*}}

\begin{document}
\title{Local  classical solutions of a kinetic equation for three waves interactions in presence of  a Dirac measure at the origin. }
\maketitle
\begin{center}
{\large M. Escobedo}\\
{\small Departamento de Matem\'aticas,} \\
{\small Universidad del
Pa{\'\i}s Vasco,} \\
{\small Apartado 644, E--48080 Bilbao, Spain.}\\
{\small E-mail~: {\tt miguel.escobedo@ehu.es}}
\end{center}
\noindent
{\bf Abstract}: 
The existence of local, classical solutions is proved,  for  a system of two coupled  equations that describe, in the framework of the wave turbulence theory,  the fluctuations around an equilibrium,  of a system of  nonlinear waves  satisfying the 3-d cubic Schr\"odinger equation, weakly  interacting  in presence of a condensate. The function that describes  the density  of waves behaves like a singular Rayleigh Jeans equilibria near the origin, and   induces a strictly increasing behavior in time of the function describing the condensate's density.\\



\section{Introduction and main result.}
\setcounter{equation}{0}
\setcounter{theo}{0}

In the wave turbulence description  of a  nonlinear system of  waves  satisfying the cubic Schr\"odinger equation (in dimensions larger than one), and weakly interacting  in presence of a condensate, one may want to consider the  
fluctuations of the solutions to the corresponding wave kinetic equations;  these fluctuations being caused by an initial  perturbation of an equilibrium. Then, a nonlinear system of two equations arises, describing the density of waves and the condensate's density (\cite{DNPZ}, \cite{N}). (A  presentation of the wave turbulence theory for  systems of interacting waves  may also be found in \cite{Zbook}. A rigorous  deduction of the kinetic wave equation for waves obeying the cubic Schr\"odinger equation in dimension three may be found in \cite{DH}. See also \cite{EV3, AV1} for some results on the kinetic equation itself).

In the case of dimension three (the only one discussed here), if the four waves or particles  interactions are disregarded (i.e. only interactions involving the condensate are kept) and under suitable assumptions on the physical system (in particular on its temperature; cf. \cite{DNPZ, Eckern, Kirkpatrick}),  the resulting  system of equations reads as follows, 

\begin{align}
&\frac {\partial F} {\partial \tau }(\tau , X)=n(\tau )\widetilde Q(F(\tau ), F(\tau ))(X)\label{S1E00}\\
&n'(\tau )=-n(\tau )\int _0^\infty \widetilde Q(F(\tau ), F(\tau ))(X)\, \sqrt XdX\label{S1E021}\\
&\widetilde Q(F, F)(X)=X^{-1/2}\int _0^X\Big(F(X-Y)F(Y)-F(X)(F(X-Y)+F(Y))\Big)dY+ \nonumber\\
&\hskip 2.5cm +2X^{-1/2}\int _X^\infty\Big((F(X)+F(Y-X))F(Y)-F(X)F(Y-X)\Big)dY.\label{S1E01}
\end{align}
The function  $n(\tau )$ describes  the condensate's density at time $\tau >0$. The variable $X$ is a scaled version of $\omega $,  the approximated frequency of the waves, where the presence of the  condensate's density $n(\tau )$ is neglected. The function  $(\tau , X)\to \sqrt X F(\tau , X)$ describes  the density of waves or particles at time $\tau $, with energy $X$. System (\ref{S1E00})--(\ref{S1E01}) also appears in the context of boson gases (cf.  Appendix \ref{deduction}), where the variable $X$  represent the energy of the particles.

The change of time variable
\begin{align}
t  =\int _0^\tau  n(\sigma  )d\sigma,\,\,\,F(\tau , X)=f(t, X)
\label{S2NewTime}
\end{align}
simplifies somewhat the system since the equation in (\ref{S1E00}) reduces to,
\begin{align}
\frac {\partial f} {\partial t}(t, X)= \widetilde Q(f(t), f(t))(X).\label{lzm75}
\end{align}
The existence of  global weak solutions $(F, n)$ to  (\ref{S1E00}), (\ref{S1E021}) has been proved in \cite{AV1} where $\sqrt X\, f(t, X)$ is a bounded measure on $(0, \infty)$ for all $t>0$. 

However, one basic aspect of the non equilibrium solutions of the system (\ref{S1E00}),(\ref{S1E021}) is  the behavior of the condensate's density, described by the function $n$, and that  crucially depends  on the behavior of $F(\tau, X  )$ as $X\to 0$ (cf.  for example Proposition 2 in \cite{S} and  Theorem 1.7 in  \cite{CE}). Indeed, if a  regular behavior of $F(\tau )$ as $X \to 0$ is assumed, and Fubini's Theorem may be applied in the right hand side of (\ref{S1E021}), 
\begin{align} 
 \frac {dn_c(\tau )} {d\tau  }&=-n_c(\tau  )\int_0^\infty \widetilde Q(F(\tau ), F(\tau ))( X  )\,  \sqrt X  \, dX =0 \label{S1EFNG}
\end{align}
due to the symmetry  properties of the operator $\widetilde Q$, and $n$ would be constant in time.  On the other hand, it is  well  known that the three waves kinetic equations  (\ref{S1E00}) or (\ref{lzm75})  have a one parameter family of stationary solutions $F_C(X)=CX^{-1}$, for $C\in \RR$,  that are singular at the origin and to which the previous argument would not apply. 
For the weak solutions obtained in \cite{AV1} such  information on the pointwise behavior near the origin is not available.

It is then natural to wonder if  non stationary solutions $(F, n)$ to (\ref{S1E00}), (\ref{S1E021}) exist, with $F$ behaving   like $\lambda (\tau )X^{-1}$ as $X\to 0$ for some function $\lambda (\tau )$, and the function $n$ would have a non trivial time dependent behavior. The purpose of this work is to prove the existence of such solutions for a particular  set of initial perturbations  $F _{ \text{in} }$ of an equilibria $F_C$,  that may be taken without loss of generality for $C=1$.

In order to describe our main result  consider, for $\theta\in \RR, \rho \in \RR$, the  spaces $X _{ \theta, \rho  }$ of continuous functions $g\in C(0, \infty)$ such that
\begin{align}
&\exists C>0; |g(X)|\le C\left(X^{-\theta}\1 _{ 0<X<1 }+X^{-\theta-\rho }\1 _{ X>1 }\right)\,\,\forall X>0,\nonumber\\
&\exists C>0; \,\, |g(X)|\le CX^{-\theta}(1+X)^{-\rho }\,\,\forall X>0,\nonumber \\
&\qquad \Longleftrightarrow \sup _{ X>0 } |g(X)|X^\theta (1+X)^\rho <\infty \label{xtr1.0} 
\end{align}
and define as usual $||g|| _{ \theta, \rho  }$ to be the infimum of all the constants $C\ge0$ satisfying (\ref{xtr1.0}).   

As initial data, we shall consider any  function $\varphi \in C_c^2(0,\infty)$  such that for $R>0$ large, $\text{supp}(\varphi )\subset (0, R)$, $\varphi (X)=1$ for all $X\in (0, R/2)$ and $\varphi '\le 0$. The following choice simplifies somewhat some of the calculations below (cf. Remark \ref{SER1})
\begin{align}
&\varphi (X)=\phi \left(\frac {X} {R} \right),\,\,R>1\label{defvfiR}\\
&\phi \in C^2_c(0, \infty),\,\,\text{supp}(\phi )\subset (0, 1),\, \label{deffi1}\\
&\phi (X)=1,\,\,\forall X\in (0, 1/2),\,\,\phi '\le 0.\label{deffi2}
\end{align}
Let us also denote   $L^1 _{ \text{loc} }([0, \infty); XdX)=\{f; \int _0^d|f(X)|XdX<\infty,\,\forall d\in (0, \infty)\}$.
\begin{theo}
\label{Mth2.0} 
For all  $r\in [0, \frac {1} {2}), q\in [0, 3)$, 
there exists $R>0$, $T>0$, and two functions  $(F, \lambda) $ satisfying  $F\in C((0, T\times (0, \infty)$, $\lambda \in C([0, T))$,
\begin{align}
&\sup_{\substack {0\le t\le T \\ X>0}}(X^{-r}+X^{q})\left|F(\tau , X)-\frac {\lambda (\tau )} {X}\right|<\infty; \, F_\tau \in L^\infty((0,  T); L^1 _{ \text{loc} }([0, \infty),XdX)),\label{Mth2.0E1BT}\\
&F(0, X)=\frac {\varphi (X)} {X},\,\forall X>0,\,\,\lambda (0)=1;\,\,\,\lambda (\tau )>0,\,\,\forall \tau \in (0,  T),\,\,\label{Mth2.0E1B}  \\
&\lim _{ X \to 0 }X\left(F(\tau, X )-\frac {\lambda (\tau )} {X}\right)=0,\,\,\forall \tau \in (0,  T)\label{Mth2.0E1}
\end{align}
and such that $F$,  together with the function $n$ defined as
\begin{align}
n(\tau )=n(0)\exp\left(\frac {\pi ^2} {3}\int _0^\tau \lambda ^2(s )ds\right),\,\tau \in (0,  T)\label{SE1Ennn}
\end{align}
satisfy  equation (\ref{S1E00}) in $L _{ \text{loc} }^\infty([0, T); L^1 _{ \text{loc} }([0, \infty); XdX)$ and equation 
(\ref{S1E021}) pointwise, for all $\tau \in (0,  T)$. They also satisfy,
\begin{align}
&\frac {d} {d\tau }\left(n(\tau )+ \int _0^\infty F(\tau , X) X^{1/2} dX\right)=0,\,\,\forall \tau \in (0,  T)\label{Mth2.0E2}\\
&\frac {d} {d\tau } \int _0^\infty F(\tau , X) X^{3/2} dX=0,\,\,\forall \tau \in (0,  T). \label{Mth2.0E3}
\end{align}
\end{theo}
Properties (\ref{Mth2.0E2}) and (\ref{Mth2.0E3}) correspond respectively to the expected conservation of the total number of waves (or particles), and of energy in the physical system. The property (\ref{SE1Ennn}) shows how does the function $n$ increases for small values of $\tau $. It is important in view of the comment about identity (\ref{S1EFNG}) above.  It confirms the prediction in \cite{S} based on formal arguments., and reflects the flux of waves towards zero frequency at $\tau >0$,  due to the behavior $\lambda (\tau )X^{-1} $ as $X\to 0$ of $F(\tau, X )$.

The   strategy of the proof of  Theorem (\ref{Mth2.0})   is similar to that already used in [10], [11], i.e. to  look
for solutions as sum of a leading order term (that yields the expected power law
behavior for small $X$) plus some higher order corrective terms (cf.  (\ref{S1EinnD2})). The leading order
term contains a free function $\lambda(\tau)$ that depends on the initial
data. This function, together with the higher order corrections are obtained by means of a
fixed point argument. The asymptotic behavior of the solutions for small $X$ that is required to
close the fixed point argument is obtained by means of a detailed analysis of the linearized
operator near the power law. The function $\lambda$ is determined by means of an integral
equation, in order to recover the asymptotic behavior (\ref{S1EinnD2}).  A new regularizing effect of the equation
is  proved and, used to control the $X$ modulus of continuity of the solution and show that the equation is satisfied in a classical sense.
The arguments used have also  many points in common with those in semigroup theory since, despite 
 the involved structure of the linearized operator, the fundamental solution of the semigroup $S$ that it generates
 may be computed.

\subsection{Functional setting}
The pair  $(F, n)$ in Theorem \ref{Mth2.0}  is obtained from a solution $f$ of (\ref{lzm75}) of  the form,
\begin{align}
\label{S1EinnD2}
f (t, X)=\frac {\lambda (t)\varphi (X)+g(t, X)} {X}
\end{align}
where the functions $g$ and  $\lambda $ must be determined such that $\lambda\in C^1(0, \infty)$, $\lambda (0)=1$, and  $g(0)=0$, from the equation
(\ref{lzm75}). In the variables $g, \lambda $, equation (\ref{lzm75}) reads
\begin{align}
\label{S1E7}
&\lambda '(t)\varphi +\frac {\partial g} {\partial t} = \lambda ^2(t) Q_N(\varphi )+\lambda (t) \mathscr L_\varphi (g)+ Q_N(g)
\end{align}
(the detailed derivation is given in the Appendix),  where the operator $Q_N$ is defined as,
\begin{align}
Q_N(h)(X)=X\widetilde Q_N(h, h)(X)
\label{S1InnD6}
\end{align}
and
\begin{align}
&\mathscr L_\varphi (g)=\mathscr L(g)+T_1(g, \varphi )+T_2(g, \varphi ) \label{S1InnD7}\\
&\mathscr L(g)=X^{-1/2}\int _{0}^\infty (g(t, Y)-g(t, X))\left(\frac {1} {|X-Y|}-\frac {1} {|X+Y|}\right)dY\label{S1InnD8}\\
&T_1(g, \varphi )=X^{-1/2} \int _{0}^\infty\! (g(t, Y)-g(t, X))\left(\frac {\varphi  (|X-Y|)-1} {|X-Y|}-\frac {\varphi  (|X+Y|)-1} {|X+Y|}\right)dY
\label{S1InnD75}\\
&T_2(g, \varphi )=X^{-1/2}\int _{0}^\infty \frac {\text{sign}(X-Y)g(t, Y)} { Y}\left(\varphi  (|X-Y|)-\varphi  (t, X) \right)+\nonumber\\
&\hskip 6cm +\frac {(\varphi  (t, Y)-\varphi  (X))g(t, |X-Y|)} { |X-Y|}dY.\label{S1InnD76}
\end{align}
The  operator $\mathscr L$ was already studied in \cite{m, m3} although with slightly different variables. 

The existence of a local classical  solution $f$ to (\ref{lzm75})  is proved,  through a fixed point argument, in the following result. 
\begin{theo}
\label{Mth1}
For all $r\in (0, 1/2)$ and $q\in (1, 3/2)$, there exists  $R>0$ large enough,  $T^{**}>0$ and two  functions $g, \lambda $ satisfying
\begin{align}
&g\in C\left([0,  T^{**}); X _{ -r, r+q } \right),\,\,\,\left|\frac {\partial  g } {\partial  t} \right|+\left| \mathscr L( g )\right|\in L^\infty _{ \text{loc} }((0, \infty)\times (0, T^*));\,\,\,\,\lambda \in C[0, T^{**}), \label{Mth1ELD0}\\
&g(0)=0;\,\,\,\,\lambda  (0)=1,  \sup _{ 0\le t <  T^{**} }|\lambda  (t )-1|<\frac {1} {4}, \label{Mth1ELD1} \\
&\forall d>0,\,\,\exists C>0;\,\,\, \int _0^d\left| \mathscr L(g(t))\right|dX \le C\sup _{ 0\le s\le t }|| g(s)|| _{-r, r+q  },\, \forall t\in (0, T^*) \label{Mth1ELD}
\end{align}
and such that the functions $\lambda $ and $f$, defined as in (\ref{S1EinnD2}), satisfy properties (\ref{Mth2.0E1BT})--(\ref{Mth2.0E1}),
and also equation (\ref{lzm75}) in $L _{ \text{loc} }^\infty([0, T^{**}); L^1 _{ \text{loc} }([0, \infty);XdX)$.
\end{theo}
Theorem \ref{Mth2.0} follows from Theorem \ref{Mth1} inverting  the change of variables (\ref{S2NewTime}) and defining the function $n$. 

By Property (\ref{Mth1ELD}) of $g$,  if the function $G$ is defined as $G(\tau , X)=F(\tau , X)-\frac{\lambda (\tau )\varphi (X) }{X}$ for $\tau \in (0, T), X>0$, then
\begin{align}
\forall d>0,\,\,\exists C>0;\,\,\, \int _0^d\left| \mathscr L(G(\tau ))\right|dX \le C\sup _{ 0\le s\le t }|| g(s)|| _{-r, r+q},  \,\,\, \forall t\in (0, T^*)
\end{align}

The property (\ref{Mth1ELD}) is a consequence of the regularizing effect of the operator $\mathscr L$ proved in Lemma \ref{blabla}: under suitable conditions on $\theta>0, \rho >0$, if $u_t-\mathscr L(u(t))=g(t)$, $u(0)=0$ and $g\in L^\infty((0, T); X _{ \theta, \rho  })$,
\begin{align}
\forall d>0,\,\exists C>0;\,\,\int_0^d \int _0^{2X} \left|\frac {u (t, X)} { X}-\frac {u (t, Y)} {Y} \right|dYdX\le C\sup _{ 0\le s\le t }|| g(s)|| _{\theta, \rho  }\label{S1RP0}
\end{align} 
Estimate (\ref{S1RP0}) is also used in the proof of (\ref{SE1Ennn}) in order to prove that the perturbation $G(\tau, X )$ of $X^{-1}$ modifies but does not cancel the flux of waves towards the origin.

Regularizing effects in kinetic equations with singular collision kernels are now well known and have been object of an extensive literature (cf. for example \cite{Al, Al2, Gr, Chen} for the Boltzmann equation, \cite{EV4} for a coagulation equation, and more recently \cite{G} for the so called kinetic MMT equation).

\begin{rem}[\upshape\bfseries{On the initial data}] 
\label{SER1} The idea behind the choice of the initial data $F(0)$ is to have $\widetilde Q(F(0), F(0))$ small. 
Very  probably the proof of  Theorem \ref{Mth1} still works for all  data where $\varphi \in C_c^2([0, \infty))$ such that, for $R>0$ large enough, $\varphi (x)=1$ on $[0, R/2)$, $\varphi (x)=0$  on $(R, \infty)$, $|\varphi '(x)|\le CR^{-1}$ and $ |\varphi ''(x)|\le CR^{-2}$ for some constant $C>0$. Conditions (\ref{defvfiR})--(\ref{deffi2}) simplify somewhat the proof of some estimates in Section  \ref{Q0} but do not modify essentially what our arguments are able to prove. No initial data really more general  (like rapidly decaying at infinity but not compactly supported, or not constantly equal to $X^{-1}$on a large interval containing the origin) have been investigated.
\end{rem}
\begin{rem}[\upshape\bfseries{On regularity}] Although only local in time, the functions  $(F, n)$ and $f$ satisfy respectively  (\ref{S1E00})--(\ref{S1E01}) and   \ref{lzm75} pointiwise and in suitable  $L ^\infty(L _{ \text{loc} }^1)$ space.
In the previous  results  (cf. \cite{AV1} and \cite{SB}) the  solutions of  \ref{lzm75} are global but weak, measure valued (for measures as initial data). A key point in the proof of Theorem \ref{Mth1} (and then Theorem \ref{Mth2.0}), is to obtain  solutions $u$ of $ u_t-\mathscr L(u)=g, u(0)=0$ satisfying (\ref{Mth1ELD}). This follows from the regularizing effect of $\mathscr L$  proved in Lemma \ref{blabla}, using detailed information on the Mellin transform of the fundamental solution of  $\partial_t-\mathscr L$.  We do not know if the regularity of the solutions of  Theorem \ref{Mth2.0}  or Theorem \ref{Mth1} may be improved,  even considering more regular initial data.
\end{rem}
\begin{rem} For the system where, in the equations (\ref{S1E00}) and (\ref{S1E021}), $\widetilde Q$ is replaced by the operator $\widetilde Q_q$ defined in (\ref{PR2}) of Appendix \ref{deduction}, weak measure valued solutions are obtained in \cite{CE}.  We believe that results similar to Theorem \ref{Mth1} and Theorem \ref{Mth2.0} should still be  true in that case. However the linearization of $\widetilde Q_p$ around the corresponding equilibrium $(e^{X}-1)^{-1}$, is more involved than $\mathscr L$(cf. \cite{Ba} for its regularizing effects). The strategy could be perhaps  to consider the term $p_2$, in (\ref{S10.1Ep2}) of Subsection  \ref{deduction}, as a small perturbation of $\widetilde Q$. 
\end{rem}

{\bf Plan of the paper.} Some properties of the operator $\mathscr L$, old and new, are presented in Section \ref{LandS}. Then, sections \ref{SecT} and  \ref{Q0} are devoted to estimate respectively  the operator $T=T_1+T_2$ and the term $Q_N(\varphi )$ for $\varphi $ defined  in (\ref{defvfiR}). The nonlinear operator $Q_N$ and the linear operator $S_\varphi $ acting on the spaces $X _{ \theta, \rho  }$ are respectively  estimated in Section \ref{QN} and  Section \ref{SFi0}. Theorem  \ref{Mth1} is  proved in Section \ref{SMth1} after some preliminary results shown in Section \ref{Sg1g2}. The proof of Theorem \ref{Mth2.0} is given in Section \ref{SMth2.0}. The Appendix contains a detailed deduction of the equation (\ref{S1E7}) and a technical detail on the values of some integrals.
\section{About the operator $\mathscr L$, and operator $S$.}
\label{LandS}
\setcounter{equation}{0}
\setcounter{theo}{0}
This Section begins with  two  Propositions about the operator $\mathscr L$, that were proved in \cite{m3}.
\begin{align}
&\frac {\partial v  (t, X  )} {\partial t}=\mathscr L(v  (t))(X )\label{E2.17hm}\\
&\frac {\partial u  (t, X  )} {\partial t}=\mathscr L(u  (t))(X )+g(t, X).\label{E2.17Nhm}
\end{align}
It may be useful to have in mind the following simple remark when reading the statements:  given $\theta \ge 0$, $\rho \ge 0$ and $\theta+\rho <3/2$  then $X _{ \theta, \rho  }\subset X _{ \theta' , \rho '}$ with continuous injection such that  $||g|| _{ \theta', \rho ' }\le ||g|| _{ \theta, \rho  }$ if
\begin{align}
\label{thetaP}
\theta' >\theta,\, \rho '\in [0, \rho +\theta-\theta'),\,\,0\le \theta'+\rho '<3.
\end{align}
The following result is proved as Corollary 2.7 in \cite{m3}.   Its proof requires the use of several properties of the function $B$  defined in Subsection \ref{FB}, studied in some detail in  \cite{m}. For a function $g$ of $s>0$ and  $X>0$, and $\beta \in (0, 2)$,  we denote, for $t>0, s\in [0, t)$:
\begin{align}
 L(t-s, g(s))=\frac {3 B(1)}{i \pi ^3}\int  _{ \mathscr Re(r )=\beta  }\frac {\Gamma(r )} {B(r )}(t-s)^{-r }
\left(\int _0^{t-s} g(s, \zeta ^2 )\zeta ^{-1+ r }d\zeta  \right)dr. \label{S2ELLg}
\end{align}

\begin{prop}
\label{S7cor1}
For all $v _0\in X _{\theta,  \rho  }$ with $\theta \ge 0$, $\rho \ge 0$, $\theta+\rho <3/2$, and $\theta',\, \rho '$ satisfying (\ref{thetaP}),  there exists a function $v\in  L^\infty((0, \infty); X _{ \theta, \rho  })\cap C((0, \infty); X _{ \theta', \rho ' }))$,   also denoted
\begin{align}
v(t, X)=S(t)v_0(X)
\end{align}
such that 
 \begin{align*}
&\qquad\left|\frac {\partial v } {\partial t} \right|+\left| \mathscr L(v )\right|\in L^\infty _{ \text{loc} }((0, \infty)\times (0, \infty))\\
&\qquad v\,\,\text{satisfies (\ref{E2.17hm})  pointwise for}\,\,t>0, X>0,\\
&\qquad\forall X>0,\,\,\lim _{t\to 0 }v (t, X)=v _0(X),
\end{align*} 
There also exists positive constants $C>0$  such that
\begin{align*}
&(ii)\quad ||v (t)|| _{  \theta, \rho  }\le C ||v|| _{\theta, \rho  }\,\,\forall t>0,\\
&(iii)\quad|v (t, X)|\le C ||v_0|| _{  \theta, \rho  }  t^{-2\theta}(1+t)^{-2\rho } \left(1+\frac {\sqrt X} {t}+\sqrt X \right)\,0<X<t^2,\\
&(iv)\quad |v (t, X)|\le  C||v _0|| _{ \theta, \rho  }t^{-2\theta} (1+t)^{-2\rho }\left(\frac {t} {\sqrt X}\right)^3,\,0<t^2<X.
\end{align*}
Moreover,
\begin{align*}
&(v)\quad v (t, X)=L (t; v _0)+R_2(t, v, X),\,\forall X\in (0, \min(1,t^2))\\
&|R_2(t, v, X)|\le  ||v_0|| _{  \theta, \rho  }t^{-2\theta}(1+t)^{-2\rho } \left(\frac {\sqrt X} {t}+\sqrt X \right),\,0<X<\min (1, t^2)\\
&(vi)\quad \forall T>0, \forall p>-1,\,\,\exists C>0;\,\,|L(t, v _0)|\le C||v _0|| _{ p, \rho  }t^{-2p}\\
&(vii)\quad \text{If for some}\,\, \varepsilon >0:\,\,v _0(X)=\alpha X^{-\theta}(1+\mathcal O(X)^\varepsilon ),\,X\to 0:\\
&\qquad \quad  L(t; v _0)=\frac {12B(1)\Gamma (2\theta)} {\pi^2 B(2\theta)}t^{-2\theta}+\mathcal O(t)^{-2\theta+\varepsilon },\,t\to 0.
\end{align*}
\end{prop}
\noindent
The next result is proved as Theorem 1.1 in \cite{m3}
\begin{prop}
\label{S7cor2}
Suppose that  $g \in L^\infty  ((0, T);  X _{\theta, \rho })$ for some $T>0$,  $\theta\ge 0$, $\rho >0$ such that $\theta+\rho \in (0, 3/2)$,  and  consider the function $u $ defined as
\begin{align}
&u  (t, X)=\int _0^t ( S(t-s)g(s))(X)ds.\label{S7cor2e1}
\end{align}
Then, for all $t\in (0, T)$ and $\theta'$, $\rho '$ satisfying (\ref{thetaP}),
\begin{align*}
&(i)\quad u \in L^\infty((0, T); X _{ \theta, \rho  })\cap C((0, T); X _{ \theta', \rho ' })\\
&\qquad\left|\frac {\partial  u } {\partial t} \right|+\left| \mathscr L(u )\right|\in L^\infty _{ \text{loc} }((0, \infty)\times (0, \infty))\\
&\hskip 1cm ||u (t)|| _{ \theta, \rho  } \le Ct\,\sup _{ 0\le s\le t  }||g(s)|| _{ X _{ \theta, \rho  }},\,\forall t>0,\\
&(ii)\,\, \text{the function}\,\,u \,\text{satisfies (\ref {E2.17Nhm})  for all}\,\,t\in (0, T), X>0,\,\text{and}\, w(0, X)=0.
\end{align*}
\begin{align*}
&(iii)\quad \forall t\in (0, T),\,s\in (0, t),\, \exists L(t-s, g(s))\in \RR,\, \text{such that for}\, p\in (-1, 1/2), \\
&\qquad \int _0^{t-X^{1/2}} L(t-s; g (s)) ds\le  C\sup _{ 0\le s\le t }||g(s)|| _{ p, \rho  }t^{1-2p},\, \forall X\in (0, t^2),\nonumber\\
& \qquad \int _{t-X^{1/2}}^t L(t-s; g (s)) ds\le C\sup _{ 0\le s\le t }||g(s)|| _{ \theta, \rho  }X^{\frac {1} {2}-\theta},\, \forall X\in (0, t^2)\\
&(iv)\quad \text{For},\,t>0,\,\, X\in (0, \min(1,t^2))\\
&\hskip 1cm  u (t, X)=\int _0^{t } L(t-s; g (s))ds+R_3(t, g, X),\nonumber\\
&\hskip 1cm |R_3(t, g, X)|\le  C\sup _{ 0\le s\le t } ||g (s)|| _{ \theta, \rho  }X^{1/2}  \left( X^{- \theta }+t^{1-2\theta}\right)\\
&(v) \quad  \quad |u (t, X)|\le C\sup _{ 0\le s\le t } ||g (s)|| _{ \theta, \rho  }
X^{-3/2}t^{4-2\theta}(1+t)^{-2\rho },\,\forall X>t^2, t\in (0,  T). \nonumber
\end{align*}
\end{prop}
\noindent
The proof of Theorem \ref{Mth2.0} requires some newregularizing property of  equation (\ref{E2.17Nhm}). 
Regularizing  effects were shown  in \cite{m} for equation (\ref{E2.17hm}), and in \cite{Ba} for the linearized Nordheim equation (cf. \cite{No} and Subsection \ref{deduction} in the Appendix)
around an equilibrium, as well as for equation (\ref{E2.17Nhm}) already in \cite{m3}. However, the following estimate of the modulus of continuity of the function $u(t)$ defined in  (\ref{S7cor2e1}) is needed here, and proved below.
\begin{lem}
\label{blabla}
For $\theta\in (0, 1/2)$, $\rho \ge 0$ such that $\theta+\rho <3/2$, there exists a constant $C>0$ and a function   $\Omega  _{2\theta }: (0, \infty)\times (0, \infty)\to \RR$  (defined in (\ref{E935})), such that for  all  $d>0$,
\begin{align}
\int _0^d\int_0^{2X'} \Omega_{2\theta} (X^{1/2}, X'^{1/2})\frac {dXdX'} {|X-X'|}<\infty, \label{l2e1E3}
\end{align}
and for all  $g\in L^\infty ((0, T)); X_{\theta,\rho })$, the function $u $ defined in (\ref{S7cor2e1}) satisfies
\begin{align}
\left|\frac {1} { X}u (t, X)-\frac {1} {X'}u (t, X')\right|\le C\sup _{ 0\le s\le t }|| g(s)|| _{\theta, \rho  }\Omega_{2\theta} (X^{1/2}, X'^{1/2})  \label{l2e1E2}
\end{align}
for all $t, X$ and $X'$ such that  $0<\min(X, X')$ and $\max(X, X')\le t<T.$
\end{lem}
The proof of Lemma \ref{blabla} requires some detailed description of the operator $S(t)$. The operator $\mathscr L$ was studied in \cite{m} under the slightly different variables,
\begin{align*}
&U(t, x)=u  (t, X),\,\,X=x^2,\\
&L(U)(x)=\mathscr L(u )(X)\\
&L(U)(x)=\int _0^\infty (U(y)-U(x))\left(\frac {1} {|x^2-y^2|}-\frac {1} {x^2+y^2} \right)\frac {y} {x}dy.
\end{align*}
It was proved in Theorem 1.2 of \cite{m} that the equation 
\begin{align}
\label{S3E1.L1}
\frac {\partial U} {\partial t}=L(U)
\end{align}
has a fundamental solution,  denoted $\Lambda \in C((0,\infty); L^1(0, \infty))$, that satisfies the equation  in $\mathscr D'((0, \infty)\times (0, \infty))$ and also pointwise  for almost every $t>0$, $x>0$, and that $\Lambda(t)\rightharpoonup \delta (x-1)$ in $\mathscr D'(0, \infty)$ as $t\to 0$. It is then proved in Theorem 1.5 of \cite{m}  that for all initial data $U_0\in L^1$  there exists a weak solution of (\ref{S3E1.L1}), denoted $S(t)U_0$, 
\begin{equation}
\label{S2ESG}
\mathscr S(t)U_0(x)=\int _0^\infty \Lambda\left(\frac {t} {y}, \frac {x} {y} \right)U_0(y)\frac {dy} {y},\,\,\forall t>0,\,\,\forall x>0,
\end{equation}
such that $\mathscr S(\cdot )U_0\in C((0, \infty); L^1(0, \infty))$, $\mathscr S(t)U_0\in C((0, \infty))$ for  all $t>0$ and  (\ref{S3E1.L1}) is satisfied in $\mathscr D'((0, \infty)\times (0, \infty))$ and $\mathscr S(t)U_0\rightharpoonup$ in $\mathscr D'(0, \infty)$ as $t\to 0$. It was also proved  that if $U_0\in L^1(0, \infty)\cap L^\infty _{ loc }(0, \infty)$ then  $L(U)\in L^\infty ((0, \infty)\times (0, \infty))$, $U_t\in L^\infty ((0, \infty)\times (0, \infty))$ and  (\ref{S3E1.L1}) is satisfied pointwise, for $t>0$ and $x>0$.  

For all $g _0\in X _{\theta,  \rho  }$ with $\theta \ge 0$, $\rho \ge 0$ and $\theta+\rho <3/2$ the function $h$ defined as $h_0(x)=g_0(x^2)$ belongs to $X _{ 2\theta, 2\rho  }$ and  the function $v (t, X)=( S(t)g_0 )(X) $ is given as
\begin{align*}
&( S(t)g_0 )(X) = (\mathscr S(t)h_0)(X^{1/2}).
\end{align*}
The proof of Lemma \ref{blabla} starts then with  the estimate of 
\begin{align*}
\frac {1} {x}\int _0^t (\mathscr S(t-s)h(s))(x)ds-\frac {1} {x'}\int _0^t (\mathscr S(t-s)h(s))(x')ds
\end{align*}
for all $h\in L^\infty (0, T; X_{\theta,\rho })$, from where the estimate (\ref{l2e1E2}) immediately follows for all $g\in L^\infty (0, T; X_{\frac {\theta} {2},\frac {\rho } {2} })$ with $0<\theta+\rho <3$, and then for $g\in L^\infty (0, T; X_{ \theta,\rho  })$ with $0<\theta+\rho <3/2$

\begin{proof}[\upshape\bfseries{Proof of Lemma \ref{blabla}}]
Consider first a function  $h\in L^\infty (0, T; X_{\theta,\rho })$ for $\theta\ge 0$ and $\theta+\rho <3$, and the operator $\mathscr S(t)$ defined in (\ref{S2ESG})  acting on it.
By the definition of the operator $\mathscr S$, for all  $0<x<x'<t$,
\begin{align*}
&\int _0^t (\mathscr S(t-s)h(s))(x)ds=\int _0^{t-x}\mathscr S(t-s)h(s)(x)ds+\int  _{ t-x }^t\mathscr S(t-s)h(s)(s)ds\\
&=\int _0^{t-x} \int _0^\infty \Lambda\left(\frac {t-s} {y}, \frac {x} {y} \right)h(s, y)\frac {dy} {y}ds+
\int _{t-x}^t \int _0^\infty \Lambda\left(\frac {t-s} {y}, \frac {x} {y} \right)h(s, y)\frac {dy} {y}ds\\
&=\int _0^{t-x} \int _0^{t-s} \Lambda\left(\frac {t-s} {y}, \frac {x} {y} \right)h(s, y)\frac {dy} {y}ds+\int _0^{t-x} \int _{t-s}^\infty \Lambda\left(\frac {t-s} {y}, \frac {x} {y} \right)h(s, y)\frac {dy} {y}ds+\\
&+\int _{t-x}^t \int _0^{t-s} \Lambda\left(\frac {t-s} {y}, \frac {x} {y} \right)h(s, y)\frac {dy} {y}ds+\int _{t-x}^t \int _{t-s}^\infty \Lambda\left(\frac {t-s} {y}, \frac {x} {y} \right)h(s, y)\frac {dy} {y}ds\\
&=J_1+J_2+J_3+J_4.
\end{align*}
where the four integrals depend on whether $x\ge t-s$, $0<x<t-s$ and $y>t-s$ or $0<y<t-s$. Let us consider $J_4$ first, where $x>t-s$, $y>t-s$. By Proposition 3.5 of  \cite{m}, when $y>t-s$ and $x>t-s$,
\begin{align*}
J_4(t, x)= \int _{t-x}^t \int _{t-s}^\infty \Lambda\left(\frac {t-s} {y}, \frac {x} {y} \right)h(s, y)\frac {dy} {y}ds
\end{align*}
\begin{align*}
&\Lambda (t, x) = \mu(t)\sum _{ k=1 }^\infty \Bigg(
\left(\frac {x} {t}\right)^{-k-\beta '_1} \Bigg)A _{ k, 4 } 
+ \sum  _{ n=1 }^\infty \Bigg(\left(\frac {x} {t}\right)^{-4n-1}\Bigg)B _{ n, 4 }\nu  _{ n, 4 }(t)  \\
&A _{ k, 4 }(t)=(-1)^k\frac { B(\beta '_1 +k)} {k!};  \,\,\, \mu (t)=\frac {1} {2i\pi }\int  _{ {\mathscr Re} (z ) =\beta' _1 }\frac {t^{-z }} {B(z  )}dz\\
&B _{ n, 4 }= (4n+1)^2\\
&\nu _{ n, 4 }(t)= \frac {\tilde r_{4n+1}} {2i\pi }\int  _{ {\mathscr Re} (z ) =\beta ' _1 }\frac {\Gamma (z -4n-1)} {B(z  )}t^{-z }dz,\,\,\,\beta _1'\in (0, 2). 
\end{align*}
The function $J_4$ may then be written as follows,
\begin{align*}
J_4(t, x)&=\sum _{ k=1 }^\infty J _{ 4, 1; k }(t, x)+\sum _{ n=1 }^\infty J _{ 4, 2; n }(t, x)\\
J _{ 4, 1; k }(t, x)&=A _{ k, 4 } \int _{t-x}^t \left(\frac {x} {t-s}\right)^{-k-\beta '_1}G_4(t-s, s) ds\\
J _{ 4, 2; n }(t, x)&=B _{ n, 4 }\int _{t-x}^t  \left(\frac {x} {t-s}\right)^{-4n-1}G _{ n, 4 }(t-s, s)ds
\end{align*}
where
\begin{align*}
G_4(t-s, s)=\int _{t-s}^\infty  \mu\left( \frac {t-s} {y}\right) h(s, y)\frac {dy} {y}\\
G _{ n, 4 }(t-s, s)=\int _{t-s}^\infty  \nu _n\left(\frac {t-s} {y}\right)  h(s, y)\frac {dy} {y}.
\end{align*}
Therefore,
\begin{align*}
&\frac {J _{ 4, 1; k }(t, x)} {x^2}-\frac {J _{ 4, 1; k }(t, x')} {x'^2}=
A _{ k, 4 } \left(x^{-k-\beta '_1-2}-x'^{-k-\beta '_1-2} \right)\times \\
&\times  \int _{t-x}^t (t-s)^{k+\beta '_1}G_4(t-s, s) ds+A _{ k, 4 } x'^{k+\beta '_1-2}\int _{t-x'}^{t-x}(t-s )^{k+\beta '_1}G_4(t-s, s) ds
\end{align*}
and 
\begin{align*}
|G(t-s, s)|\le C\sup _{ 0\le s\le t } ||h(s)|| _{ \theta, \rho  }(t-s)^6\int _{t-s}^\infty  y^{-7-\theta}  dy=C\sup _{ 0\le s\le t } ||h(s)|| _{ \theta, \rho  }(t-s)^{-\theta},
\end{align*}
since $
\mu(\sigma )=\rho _1\sigma ^6+\mathcal O (\sigma )^7$ for $\sigma \in (0, 1)$. Then,
\begin{align*}
\left|\frac {J _{ 4, 1; k }(t, x)} {x^2}-\frac {J _{ 4, 1; k }(t, x'^2)} {x}\right|\le C\sup _{ 0\le s\le t } 
||h(s)|| _{ \theta, \rho  }A _{ k, 4 } \left|x^{-k-\beta '_1-2}-x'^{-k-\beta '_1-2} \right|\times \\
\times x^{k+\beta '_1+1-\theta}+C\sup _{ 0\le s\le t } ||h(s)|| _{ \theta, \rho  }A _{ k, 4 } x'^{-k-\beta '_1-2}\left|x^{k+\beta '_1+1-\theta}-x'^{k+\beta '_1+1-\theta}\right|.
\end{align*}
A similar argument using
\begin{align*}
\nu _{n, 4  }(\sigma )=\gamma  _{ n, 4 }\sigma ^2(1+\mathcal O(\sigma )),\,\,\sigma \in (0, 1)\\
|G _{ n, 4 }(t-s, s)|\le C||h(s)|| _{ \theta, \rho  }\gamma _{ n, 4 }(t-s)^\theta
\end{align*}
shows,
\begin{align*}
\left|\frac {J _{ 4, 2; n }(t, x)} {x^2}-\frac {J _{ 4, 2; n }(t, x')} {x'^2}\right|\le C\sup _{ 0\le s\le t } ||h(s)|| _{ \theta, \rho  }
B _{ n, 4 }\gamma  _{ n, 4 }\left|x^{-4n-3}-x'^{-4n-3} \right|\times \\
\times x^{4n+2-\theta}
+C\sup _{ 0\le s\le t } ||h(s)|| _{ \theta, \rho  }B _{ n, 4 }\gamma _{ n, 4 } x'^{-4n-3}\left|x^{4n+2-\theta}-x'^{4n+2-\theta}\right|.
\end{align*}
It follows,
\begin{align*}
\left|\frac {J_4(t, x)} {x}-\frac {J_4(t, x')} {x'}\right| \le C\sup _{ 0\le s\le t } ||h(s)|| _{ \theta, \rho  }\sum _{ k=1 }^\infty A _{ k, 4 }
\Bigg(  \left|x^{-k-\beta '_1-2}-x'^{-k-\beta '_1-2} \right|\times \\
\times x^{k+\beta '_1+1-\theta}
+ x'^{-k-\beta '_1-2}\left|x^{k+\beta '_1+1-\theta}-x'^{k+\beta '_1+1-\theta}\right|\Bigg)+\\
+C \sup _{ 0\le s\le t }||h(s)|| _{ \theta, \rho  }\sum _{ n=1 }^\infty B _{ n, 4 }\gamma _{ n, 4 }\Bigg(
 \left|x^{-4n-3}-x'^{-4n-3} \right|x^{4n+2-\theta}+\\
+x'^{-4n-3}\left|x^{4n+2-\theta}-x'^{4n+2-\theta}\right|\Bigg)
\end{align*}

Since $y>t-s$ and $0<s<t-x$  in $J_2(t, x)$, by Proposition 3.5 of  \cite{m}, there exists two sequences $\{A _{ k, 2 }\} _{ k\ge 0 }\subset \RR$ and $\sigma^* _k\in (-4k-2, -4k-1)$ such that
\begin{align*}
\Lambda\left(\frac {t-s} {y}, \frac {x} {y}\right)=\sum _{ k=0 }^\infty A _{ k, 2 }\left(\frac {x} {t-s} \right)^{-\sigma^* _k}\mu  _{ k, 2 }\left(\frac {t-s} {y}\right)\\
\mu  _{ k, 2 }(t)=\int  _{ \Re e(z )=\beta _1' }\Gamma(z -\sigma^* _k)t ^{-z }dz .
\end{align*}
and the function  $J_2(t, x)$ may be written,
\begin{align*}
&J_2(t, x)=\sum _{k=0 }^n J _{ k, 2 } (t, x)\\
&J _{ k, 2 }(t, x)=A _{ k, 2 }\int _0^{t-x}\left(\frac {x} {t-s} \right)^{-\sigma^* _k}G _{ 2, k }(t-s, s)ds\\
&G _{ 2, k }(t-s, s)=\int  _{ t-s }^\infty
\mu  _{ k, 2 }\left(\frac {t-s} {y}\right)h(s, y)\frac {dy} {y}
\end{align*}
Then,
\begin{align*}
&\frac {J _{ k, 2 }(t, x)} {x^2}-\frac {J _{k, 2  }(t, x')} {x'^2}= A _{ k, 2 }x^{-2}\int _{ 0 }^{t-x} \left(\frac {x} {t-s} \right)^{-\sigma^* _k}
G _{ k, 2 }(t-s, s)ds-\\
&\hskip 5cm -A _{ k, 2 }x'^{-2}\int _{ 0 }^{t-x'} \left(\frac {x'} {t-s} \right)^{-\sigma _k}
G _{ k, 2 }(t-s, s)ds
\end{align*}
and since,
\begin{align*}
&|\mu  _{ k, 2 }(s )|\le \gamma  _{ k, 2 }s^{6},\,\,s \in (0, 1)\\
&|G _{ k, 2 }(t-s, s)|\le C||h(s)|| _{ \theta, \rho  }\gamma _{ k, 2 }(t-s)^{-\theta}
\end{align*}
it follows,
\begin{align*}
\left|\frac {J _{ k, 2 }(t, x)} {x^2}-\frac {J _{k, 2  }(t, x')} {x'^2}\right|\le CA _{ k, 2 }\int _0^{t-x}
\left| \frac {1} {x^2}\left(\frac {x} {t-s} \right)^{-\sigma ^*_k}-\frac {1} {x'^2}\left(\frac {x'} {t-s} \right)^{-\sigma ^*_k}\right|\times \\
\times  |G _{ 2, k }(t-s, s)|ds+
CA _{ k, 2}\frac {1} {x'^2}\int  _{ t-x }^{t-x'} \left(\frac {x'} {t-s} \right)^{-\sigma ^*_k} |G _{ 2, k }(t-s, s)|ds\\
\le CA _{ k, 2 }\gamma  _{ k, 2 }\sup _{ 0\le s\le t }||h(s)|| _{ \theta, \rho  } 
\left|x^{-\sigma ^*_k-2}-x'^{-\sigma ^*_k-2}\right|\int _0^{t-x} (t-s)^{\sigma ^*_k-\theta } ds+\\
+CA _{ k, 2}\gamma  _{ k, 2 }\sup _{ 0\le s\le t }||h(s)|| _{ \theta, \rho  } x'^{-\sigma ^*_k-2}\int  _{ t-x }^{t-x'}  (t-s)^{\sigma ^*_k-\theta} ds.
\end{align*}
Therefore,
\begin{align*}
&\left|\frac {J _{ k, 2 }(t, x)} {x^2}-\frac {J _{k, 2  }(t, x')} {x'^2}\right|\le 
 CA _{ k, 2 }\gamma  _{ k, 2 }\sup _{ 0\le s\le t }||h(s)|| _{ \theta, \rho  } 
\left|x^{-\sigma ^*_k-2}-x'^{-\sigma ^*_k-2}\right|\times \\
&\times x^{\sigma ^*_k-\theta +1}+CA _{ k, 2}\gamma  _{ k, 2 }\sup _{ 0\le s\le t }||h(s)|| _{ \theta, \rho  }\, x'^{-\sigma ^*_k-2}
\left|x^{\sigma ^*_k-\theta+1}-x'^{\sigma ^*_k-\theta +1}\right|
\end{align*}
and,
\begin{align*}
&\left|\frac {J _{ 2 }(t, x)} {x^2}-\frac {J _{2  }(t, x')} {x'^2}\right|\le
C\sup _{ 0\le s\le t }||h(s)|| _{ \theta, \rho  }\sum _{ k=0 }^\infty A _{ k, 2 }\gamma  _{ k, 2 }\times \\
&\times \Bigg(
\left|x^{-\sigma ^*_k-2}-x'^{-\sigma ^*_k-2}\right|x^{\sigma ^*_k-\theta +1}
+x'^{-\sigma ^*_k-2}\left|x^{\sigma ^*_k-\theta+1}-x'^{\sigma ^*_k-\theta +1}\right|
\Bigg).
\end{align*}
Since $y\in (0, t-s)$ and $s<t-x$ in $J_1$, Proposition 3.1 and Proposition 3.2 in \cite{m} yield,
\begin{align*}
\Lambda\left(\frac {t-s} {y}, \frac {x} {y}\right)=\sum _{ k=0 }^\infty A _{ k, 1 }\left(\frac {x} {t-s} \right)^{-\sigma ^*_k}\mu  _{ k, 1 }\left( \frac {t-s} {y}\right)\\
\mu  _{ k, 1 }\left(t\right)=\int  _{ \Re e(z  )= \beta }\frac {\Gamma (z -\sigma ^*_k)t^{-z }} {B(z  )}dz 
\end{align*}
where the sequence $\{\sigma ^*_k\} _{ k\in \NN }$ is as before. The function  $J_1(t, x)$ may be written,
\begin{align*}
&J_1(t, x)=\sum _{k=0 }^n J _{ k, 1 } (t, x)\\
&J _{ k, 1 }(t, x)=A _{ k, 1 }\int _0^{t-x}\left(\frac {x} {t-s} \right)^{-\sigma ^*_k}G _{ k, 1 }(t-s, s)ds\\
&G _{ k, 1 }(t-s, s)=\int  _0^{ t-s }
\mu  _{ k, 1 }\left(\frac {t-s} {y}\right)h(s, y)\frac {dy} {y}
\end{align*}
and,
\begin{align*}
&\frac {J _{ k, 1 }(t, x)} {x^2}-\frac {J _{k, 1  }(t, x')} {x'^2}= A _{ k, 1 }x^{-2}\int _{ 0 }^{t-x} \left(\frac {x} {t-s} \right)^{-\sigma ^*_k}
G _{ k, 1 }(t-s, s)ds-\\
&\hskip 5cm -A _{ k, 1 }x'^{-2}\int _{ 0 }^{t-x'} \left(\frac {x'} {t-s} \right)^{-\sigma ^*_k}
G _{ k, 1 }(t-s, s)ds.
\end{align*}
Since,
\begin{align*}
&|\mu  _{ k, 1 }(s)|\le \gamma  _{ k, 1 }s^{-3},\,\,s>1\\
&|G _{ k, 1 }(t-s, s)|\le C||h(s)|| _{ \theta, \rho  }\gamma _{ k, 1 }(t-s)^{-\theta}
\end{align*}
we obtain arguing as for $J_2$,
\begin{align*}
&\left|\frac {J _{ k, 1 }(t, x)} {x^2}-\frac {J _{k, 1  }(t, x')} {x'^2}\right|\le 
 CA _{ k, 1 }\gamma  _{ k, 1 }\sup _{ 0\le s\le t }||h(s)|| _{ \theta, \rho  } 
\left|x^{-\sigma ^*_k-2}-x'^{-\sigma ^*_k-2}\right|\times \\
&\times x^{\sigma ^*_k-\theta +1}+CA _{ k, 1}\gamma  _{ k, 1 }\sup _{ 0\le s\le t }||h(s)|| _{ \theta, \rho  }\, x'^{-\sigma ^*_k-2}
\left|x^{\sigma ^*_k-\theta+1}-x'^{\sigma ^*_k-\theta +1}\right|
\end{align*}
and
\begin{align*}
&\left|\frac {J _{ 1 }(t, x)} {x^2}-\frac {J _{1  }(t, x')} {x'^2}\right|\le
C\sup _{ 0\le s\le t }||h(s)|| _{ \theta, \rho  }\sum _{ k=0 }^\infty A _{ k, 1 }\gamma  _{ k, 1 }\times \\
&\times \Bigg(
\left|x^{-\sigma ^*_k-2}-x'^{-\sigma ^*_k-2}\right|x^{\sigma ^*_k-\theta +1}
+x'^{-\sigma ^*_k-2}\left|x^{\sigma ^*_k-\theta+1}-x'^{\sigma ^*_k-\theta +1}\right|
\Bigg).
\end{align*}
The function $\Lambda$ in  $J_3$, where $y\in (0, t-s)$ and $s>t-x$ and then $(t-s)/y>1$ and $x>t-s$, is slightly more involved due to the distributions of poles and zeros of the function $\Lambda$ in  (3.8) of \cite{m}. 

However, there exists three sequences of real numbers $\{A _{ k, 3 }\} _{ k=1 }^\infty$, $\{B _{ n, 3 }\} _{ n=1 }^\infty$, $\{C _{ k, n}\} _{ n=0 }^\infty$, and a sequence $\{\sigma _{ k }\} _{ k=1 }^\infty$, $\sigma _ k\in (4k+4, 4k+5)$ such that,
\begin{align*}
\Lambda\left(\frac {t-s} {y}, \frac {x} {y} \right)=\sum _{ n=0 }^\infty B _{ n, 3 }\left( \frac {x} {t-s}\right)^{-4n-1} \sum _{ k=1 }^\infty A _{ k, 3 }\left( \frac {t-s} {y}\right)^{-\sigma _k}+\\
+\sum _{ n=0 }^\infty  \sum _{ k=1 }^\infty A _{ k, 3 }C _{ k, n }\left( \frac {t-s} {y}\right)^{-\sigma _k}\left( \frac {x} {t-s}\right)^{-\sigma _k-n}.
\end{align*}
The function $J_3(t, x)$ may then be written,
\begin{align*}
J_3(t, x)=\sum _{ n=0 }^\infty J _{ 3, n }(t, x)+\sum _{ n=0 }^\infty \sum _{ k=1 }^\infty J _{ 3, n, k }(t, x)
\end{align*}
where,
\begin{align*}
J _{ 3, n }(t, x)=B _{ n, 3 }\int  _{ t-x }^t \left( \frac {x} {t-s}\right)^{-4n-1} G _{ n, 3 }(t-s, s)ds\\
G _{ n, 3 }(t-s, s)=\sum _{ k=1 }^\infty A _{ k, 3 }\int  _{ 0}^{t-s}\left( \frac {t-s} {y}\right)^{-\sigma _k}h(y)\frac {dy} {y}
\end{align*}
\begin{align*}
J _{ 3, n, k }(t, x)=A _{ k, 3 }C _{ k, n }\int  _{ t-x }^t\left( \frac {x} {t-s}\right)^{-\sigma _k-n}\widetilde G _{k, 3 }(t-s, s)ds\\
\widetilde G _{k, 3 }(t-s, s)=\int_0^{t-s}\left( \frac {t-s} {y}\right)^{-\sigma _k}h(y)\frac {dy} {y}.
\end{align*}
It follows,
\begin{align*}
\left|\frac {J _{ 3, n }(t, x)} {x^2}-\frac {J _{ 3, n }(t, x'^2)} {x'^2}\right|\le |B _{ n, 3 }|\int  _{ t-x }^t
\left|x^{-2}\left( \frac {x} {t-s}\right)^{-4n-1} -x'^{-2}\left( \frac {x'} {t-s}\right)^{-4n-1}  \right|\times \\
\times  |G _{ n, 3 }(t-s, s)|ds+|B _{ n, 3 }| x'^{-2}\int  _{ t-x }^{t-x'}\left( \frac {x'} {t-s}\right)^{-4n-1} |G _{ n, 3 }(t-s, s)|ds.
\end{align*}
Following the same arguments as above, use of
\begin{align*}
|G _{ n, 3 }(t-s, s)|\le C||h(s)|| _{ \theta, \rho  } \sum _{ k=1 }^\infty |A _{ k, 3 }|
\int  _{ 0}^{t-s}\left( \frac {t-s} {y}\right)^{-\sigma _k}y^{-1-\theta}dy\\
\le C||h(s)|| _{ \theta, \rho  }(t-s)^{-\theta} \sum _{ k=1 }^\infty\frac { |A _{ k, 3 }|} {(\sigma  _k-\theta)}
\end{align*}
yields the following estimate,

\begin{align*}
&\left|\frac {J _{ 3, n }(t, x)} {x^2}-\frac {J _{ 3, n }(t, x')} {x'^2}\right|\le C|B _{ n, 3 }|\sup _{ 0\le s\le t  }||h(s)|| _{ \theta, \rho  }
 \Bigg(\left|x^{-4n-3} -x'^{-4n-3}\right|\times 
 \\
&\hskip 2.4cm \times\int  _{ t-x }^t  (t-s)^{4n+1-\theta}ds+x'^{-4n-3} \int  _{ t-x }^{t-x'} (t-s )^{4n+1-\theta}ds\Bigg)\\
&\le \frac {C|B _{ n, 3 }|} {(4n+2-\theta)}\sup _{ 0\le s\le t  }||h(s)|| _{ \theta, \rho  }
 \Bigg(\left|x^{-4n-3} -x'^{-4n-3}\right| x^{4n+2-\theta}+\\
 &\hskip 6cm +x'^{-4n-3}\left|x^{4n+2-\theta}-x'^{4n+2-\theta} \right|\Bigg).
\end{align*}
By similar arguments,
\begin{align*}
|\widetilde G _{k, 3 }(t-s, s)|\le C||h(s)|| _{ \theta, \rho  }(t-s)^{-\sigma _k}\int _0^{t-s} y^{\sigma _k-1-\theta}dy\\
=
C||h(s)|| _{ \theta, \rho  }(\sigma _k-\theta)^{-1}(t-s)^{-\theta}
\end{align*}
and then
\begin{align*}
&\left|\frac {J _{ 3, n, k }(t, x)} {x^2}-\frac {J _{ 3, n, k }(t, x')} {x'^2}\right|\le C|A _{ k, 3 }||C _{ k, n }|
\sup _{ 0\le s\le t }||h(s)|| _{ \theta, \rho  }(\sigma _k-\theta)^{-1}\times \\
&\hskip 4cm \times \left|x^{-\sigma _k-n-2}-x'^{-\sigma _k-n-2} \right|\int  _{ t-x }^t
 (t-s)^{\sigma _k+n-\theta}ds+\\
&\hskip 0.5cm+C|A _{ k, 3 }||C _{ k, n }|\sup _{ 0\le s\le t }||h(s)|| _{ \theta, \rho  }(\sigma _k-\theta)^{-1}x'^{-\sigma _k-n-2}\int  _{ t-x }^{t-x'} (t-s )^{\sigma _k+n-\theta}ds\le 
\end{align*}

\begin{align*}
&\le C|A _{ k, 3 }||C _{ k, n }|\sup _{ 0\le s\le t }||h(s)|| _{ \theta, \rho  }(\sigma _k-\theta)^{-1}(\sigma _k+n+1-\theta)^{-1}\times \\
&\times \Bigg(\left|x^{-\sigma _k-n-2}-x'^{-\sigma _k-n-2} \right| x^{\sigma _k+n+1-\theta}+
x'^{-\sigma _k-n-2}\left|x^{\sigma _k+n+1-\theta}-x'^{\sigma _k+n+1-\theta} \right| \Bigg)
\end{align*}
from where,
\begin{align*}
&\left|\frac {J _{ 3 }(t, x)} {x^2}-\frac {J _{ 3 }(t, x')} {x'^2}\right|\le C\sum _{ n=0 }^\infty 
 \frac {|B _{ n, 3 }|} {(4n+2-\theta)}\sup _{ 0\le s\le t  }||h(s)|| _{ \theta, \rho  }\times \\
 &\times \Bigg(\left|x^{-4n-3} -x'^{-4n-3}\right| x^{4n+2-\theta}
 +x'^{-4n-3}\left|x^{4n+2-\theta}-x'^{4n+2-\theta} \right|\Bigg)+\\
 &+C\sum _{ n=0 }^\infty\sum _{ k=1 }^\infty  |A _{ k, 3 }||C _{ k, n }|\sup _{ 0\le s\le t }||h(s)|| _{ \theta, \rho  }(\sigma _k-\theta)^{-1}(\sigma _k+n+1-\theta)^{-1}\times \\
&\times \Bigg(\left|x^{-\sigma _k-n-2}\!\!-x'^{-\sigma _k-n-2} \right| x^{\sigma _k+n+1-\theta}+
x'^{-\sigma _k-n-2}\left|x^{\sigma _k+n+1-\theta}-x'^{\sigma _k+n+1-\theta} \right|\Bigg).
\end{align*}
We have then, for $0<x<t<1$ and $0<x'<t<1$

\begin{align*}
&\left|\frac {1} {x}\int _0^t (\mathscr S(t-s)h(s))(x)ds-\frac {1} {x'}\int _0^t (\mathscr S(t-s)h(s))(x')ds\right|\le \sup _{ 0\le s\le t }||h(s)|| _{ \theta, \rho  }\times \\ 
&\hskip 10.6cm \times \widetilde \Omega (t, x, x'),\\
&\widetilde \Omega_\theta (x, x')=\\
&= \Bigg[\sum _{ j=1 }^3\sum _{ k=0 }^\infty D_k\Bigg(
\left|x^{-\alpha  _{ k, j }} -x'^{-\alpha  _{ k, j }}\right|x^{\alpha  _{ k, j }-1-\theta}+
\left|x^{\alpha  _{ k, j }-1-\theta} -x'^{\alpha  _{ k, j }-1-\theta}\right|x'^{-\alpha  _{ k, j }}
\Bigg)+\\
&+\sum _{ k=0 }^\infty\sum _{ n=0 }^\infty D _{ k, n }
\Bigg(
\left|x^{-\gamma   _{ k, n }} -x'^{-\gamma   _{ k, n }}\right|x^{\gamma   _{ k, n }-1-\theta}+
\left|x^{\gamma   _{ k, n }-1-\theta} -x'^{\gamma   _{ k, n }-1-\theta}\right|x'^{-\gamma   _{ k, n }}
\Bigg)\Bigg]
\end{align*}
for some real constants $D_k$ and $D _{ k, n }$ and with
\begin{align*}
&\alpha  _{ k, 1 }=k+\beta '_1+2;\,\,\,\,
\alpha  _{ k, 2 }=4k+3,\\
&\alpha  _{ k, 3}=-\sigma _k^*+1+\theta;\,\,\,\,\,\gamma   _{ k, n }=\sigma _k+n+2.
\end{align*}
The change of variables $x^2=X$, $g(t, X)=h(t, x)$ yields $g\in L^\infty((0, T); X _{ \frac {\theta} {2}, \frac {\rho } {2} }$ and for $t\in (0, T)$,
$u (t, X)=\int _0^t\mathscr (S(t-s)h(s))(x)ds$ from where, for $0<X, X'<t^2<1$,

\begin{align}
\left|\frac {1} { X}u (t, X)-\frac {1} {X'}u (t, X')\right|& \le \sup _{ 0\le s\le t }||h(s)|| _{ \theta, \rho  } \widetilde\Omega_\theta (X^{1/2}, X'^{1/2})\nonumber \\
&=\sup _{ 0\le s\le t }|| g(s)|| _{ \frac {\theta} {2},\frac {\rho} {2}  }\widetilde \Omega_\theta (X^{1/2}, X'^{1/2}).\label{E933}
\end{align}
Since the left hand side of (\ref{E933}) is symmetric with respect to $X$ and $X'$ it follows
\begin{align}
& \left|\frac {1} { X}u (t, X)-\frac {1} {X'}u (t, X')\right|\le \sup _{ 0\le s\le t }|| g(s)|| _{ \frac {\theta} {2},\frac {\rho} {2}  }\Omega_\theta (X^{1/2}, X'^{1/2}),\nonumber \\
&\Omega_\theta (X^{1/2}, X'^{1/2})= \min\left(\widetilde\Omega_\theta (X^{1/2}, X'^{1/2}); 
\widetilde\Omega_\theta (X'^{1/2}, X^{1/2})\right), \label{E935}
\end{align}
which is just (\ref{l2e1E2}) written for $\theta/2$, $\rho /2$  instead of $\theta$ and $\rho $. Property (\ref{l2e1E3}) follows now from Proposition \ref{P52} in the Appendix.
\end{proof}
\begin{prop}
\label{blablabcor}
Under the same hypothesis on $\theta, \rho $ as in Lemma \ref{blabla}, for all $d>0$ there exists a constant $C>0$ such that, for all $g\in L^\infty ((0, T)); X_{\theta,\rho })$ the function $u$ defined in (\ref{S7cor2e1}) satisfies
\begin{align*}
\int _0^d\left| \mathscr L(u)\right|dYdX\le C\sup _{ 0\le s\le t }|| g(s)|| _{\theta, \rho  } \left(1+  \int _0^d\int_0^{rX} \Omega_{2\theta} (X^{1/2}, Y^{1/2})\frac {dYdX} {|X-Y|}\right).
\end{align*}
\end{prop}
\begin{proof}
Given $g\in X _{\theta, \rho  }$,  denote,
\begin{align*}
\ell_u (X, Y)=(u(t, Y)-u(t, X))\left(\frac {1} {|X-Y|}-\frac {1} {|X+Y|}\right).
 \end{align*}
Fix now $r >1$, $M>2rd$  and write,
\begin{align*}
\int _0^d\left| \mathscr L(u)\right|dYdX=I_1+I_2+I_3,\,\,\,
I_1=\int _0^d\frac {1} {\sqrt X}\int _0^{r X}|\ell_u (X, Y)|dYdX\\
I_2=\int _0^d\frac {1} {\sqrt X}\int _{r X}^M|\ell_u (X, Y)|dYdX
I_3=\int _0^d\frac {1} {\sqrt X}\int _M^{\infty}|\ell_u (X, Y)|dYdX.
\end{align*}
Let us bound $I_1$ first. Since,
\begin{align*}
\frac {u(t, X)} {X}-\frac {u(t, Y)} {Y}= \frac {(u(t, X)-u(t, Y))} {X}-\frac {u(t, Y)(X-Y)} {XY},
\end{align*}
\begin{align*}
 \frac {|u(t, Y)-u(t, X)|} {XY|X-Y|} \le \left|\frac {u(t, X)} {X}-\frac {u(t, Y)} {Y}\right|\frac {1} {|X-Y|}+\frac {|u(t, Y)|} {XY}
 \end{align*}
 By Lemma \ref{blabla}
 \begin{align*}
 \int _0^d&\int _0^{r X}\frac {|u(t, Y)-u(t, X)|} {|X-Y|}\frac{dYdX}{XY}\le C\sup _{ 0\le s\le t }|| g(s)|| _{\theta, \rho  }\times \\
& \times \int _0^d\int_0^{r X} \Omega_{2\theta} (X^{1/2}, Y^{1/2})\frac {dYdX} {|X-Y|}+ \int _0^d\int _0^{r X} \frac {|u(t, Y)|} {XY}dYdX.\\
&\mathscr L(g)=X^{-1/2}\int _{0}^\infty (g(t, Y)-g(t, X))\left(\frac {1} {|X-Y|}-\frac {1} {|X+Y|}\right)dY
\end{align*}
And since,
\begin{align*}
 \int _0^d\int _0^{rX} \frac {|u(t, Y)|} {XY}dYdX\le \sup _{ 0\le s\le t }|| g(s)|| _{\theta, \rho  }\int _0^d\frac {1} {X}\int _0^{rX}X Y^{\theta-1}dY\\
 =\frac {(rd)^\theta} {\theta(1+\theta)}\sup _{ 0\le s\le t }|| g(s)|| _{\theta, \rho  }
\end{align*}
\begin{align*}
\int _0^d\int _0^{rX}\frac {|u(t, Y)-u(t, X)|} {|X-Y|}&\frac{dYdX}{XY}\le C\sup _{ 0\le s\le t }|| g(s)|| _{\theta, \rho  }\times \\
& \times \left(1+  \int _0^d\int_0^{rX} \Omega_{2\theta} (X^{1/2}, Y^{1/2})\frac {dYdX} {|X-Y|}\right)
\end{align*}
and  it easily  follows,
\begin{align*}
I_1\le C\sup _{ 0\le s\le t }|| g(s)|| _{\theta, \rho  } \left(1+  \int _0^d\int_0^{rX} \Omega_{2\theta} (X^{1/2}, Y^{1/2})\frac {dYdX} {|X-Y|}\right).
\end{align*}
On the other hand,
\begin{align*}
I_2&\le \frac {2} {r-1}\int _0^d |u(t, X)|\int  _{ rX }^M\frac {dY} {Y}dX+ \frac {2} {r-1} \int _0^d\int  _{ rX }^\infty \frac {|u(Y)|dYdX} {X+Y}\\
&\le  C\sup _{ 0\le s\le t }|| g(s)|| _{\theta, \rho  }\left( \int _0^dX^{-\theta}\log (rX)dX+\int _0^d X^{-\theta} \int _r^\infty \frac {dZ} {Z^\theta (1+Z)}\right).
\end{align*}
In $I_3$, since $\frac {1} {|X-Y|}-\frac {1} {|X+Y|}=\frac {2X} {(Y-X)(X+Y)}\le $
\begin{align*}
&\int _0^d \int _M^\infty |u(t, X)|\left(\frac {1} {|X-Y|}-\frac {1} {|X+Y|}\right)dYdX\le \\
&\le \frac {2r} {r-1} \sup _{ 0\le s\le t }|| g(s)|| _{\theta, \rho  }
\int _0^d X^{-\theta} \int _M^\infty \frac {dY} {Y^2}dX=C\sup _{ 0\le s\le t }|| g(s)|| _{\theta, \rho  }\\
&\int _0^d \int _M^\infty |u(t, Y)|\left(\frac {1} {|X-Y|}-\frac {1} {|X+Y|}\right)dYdX\\
&\le 2(M-d)^{-1})\sup _{ 0\le s\le t }|| g(s)|| _{\theta, \rho  } \int _0^d \int _M^\infty Y^{-\theta-1}(1+Y)^{-\rho }dYdX
\end{align*}
and Proposition follows.
\end{proof} 
 
\section{Estimates of $T=T_1+T_2$.}
\label{SecT}
\setcounter{equation}{0}
\setcounter{theo}{0}

In order to solve (\ref{S1E7})  it is necessary to  solve first 
 \begin{align}
\label{S1E8B}
&\frac {\partial h(t) } {\partial t} = \mathscr L_\varphi (h(t) )\equiv \mathscr L(h(t))+T(h(t))\\
&h(0)=h_0\in X _{ \theta, \rho  }
\end{align}
where $\varphi $ is defined in (\ref{defvfiR}),  and then to estimate the operators $T=T_1+T_2$ acting on the spaces 
$X _{ \theta, \rho  }$.
We suppose in this Section that for some $r>-2, q>0$ such that $q+r>0$,
\begin{align*}
&|g (X)|\le C_gX^{r},\,\,0<X<1\\
&|g (X)|\le C_gX^{-q},\,\,X>1
\end{align*}
and wish to estimate $T_1(g , \varphi )$ and $T_2(g , \varphi )$, defined in (\ref{S1InnD75}) and (\ref{S1InnD76}). Denote, as before
\begin{align*}
||g|| _{ -r, q+r }=\sup _{ 0<X<1 }X^{-r}|g(X)|+\sup _{ X>1 }X^{q}|g(X)|\\
\sim  \sup _{ X>0 }X^{-r}\left(1+X\right)^{q+r}|g(X)|.
\end{align*}
In order for our calculations to be  easier  to follow,  let us  in this Section  consider more generally any  function  $\varphi(X) \equiv \phi (X/R)$ where $\phi \in C^2(0, \infty)\cap C([0, \infty)$, $\phi (X)=1$ on $[0, \alpha ]$ for some $\alpha >0$, $\text{supp}(\phi )=[0, \beta]$ for some $\beta >\alpha $. In Definition (\ref{defvfiR}), $\alpha =1/2$ and $\beta =1$, so that the estimates of this Section apply.
\subsection{The operator  $T_1$.}

If we denote $\chi  \equiv \varphi -1$ then, $\chi =0$ on $(0, \alpha R)$,  $\chi =-1$  on $(\beta R, \infty)$,  and
\begin{align*}
T_1(g, \varphi )=X^{-1/2} \int _{0}^\infty (g(t, Y)-g(t, X))\left(\frac {\chi   (|X-Y|)} {|X-Y|}-\frac {\chi   (|X+Y|)} {|X+Y|}\right)dY.
\end{align*}
\begin{prop} 
\label{T1}
There exists a constant $C>0$ such that,
\begin{align*}
|T_1(g, \varphi )(X)|\le C||g|| _{ -r, q+r }\left(X^{\frac {1} {2}-r^-}\1 _{ 0<X<1 }+X^{-q-\frac {1} {2}}(\log X)\,\1 _{ X>1 }\right).
\end{align*}
\end{prop}
\subsubsection{In the region where $X$ small.}
For $X<\alpha R/2$ the integral in $T_1$ may be decomposed as follows:
\begin{align}
\label{S3E10}
\int _0^\infty[\cdots]dy&=\int _0^{\alpha R-X}[\cdots]dY+\int _{\alpha R-X}^ {\alpha R+X}[\cdots]dY+\nonumber\\
&+\int_{\alpha R+X}^{\beta R-X}[\cdots]dY+\int _{\beta R-X}^{\beta R+X}[\cdots]dY+\int  _{ \beta R+X }^\infty[\cdots]dY. 
\end{align}
The first integral in the right hand side of (\ref{S3E10}) is identically zero. In the second integral, $Y\in (\alpha R-X, X+\alpha R)$. Then $\chi (|X-Y|)=0$. Since $-r-q<0$, if  $r>0$,
\begin{align*}
|g(Y)|\le ||g|| _{ -r, q+r }Y^r(1+Y)^{-r-q}\le ||g|| _{ -r, q+r }\left(X+\alpha R\right)^r\left(1+\alpha R-X\right)^{-r-q}\\
\le ||g|| _{ -r, q+r }\left(X+\alpha R\right)^{r}\left( 1+\frac {\alpha R} {2}\right)^{-r-q}\le ||g|| _{ -r, q+r }C(R).
\end{align*}
Then,
\begin{align*}
&X^{-1/2}\int _{\alpha R-X}^ {\alpha R+X} |g(t, Y)-g(t, X)|\left(\frac {\chi   (|X-Y|)} {|X-Y|}-\frac {\chi   (|X+Y|)} {|X+Y|}\right)dY\le \\
&\le X^{-1/2}\int _{\alpha R-X}^ {\alpha  R+X} (|g(t, Y)|+|g(t, X)|) \frac {dY} {|X+Y|}\\
&\le ||g|| _{ -r, q+r }(C(R)+X^{r})X^{-1/2}\int _{\alpha R-X}^ {\alpha  R+X}  \frac {dY} {|X+Y|}\\
&=||g|| _{ -r, q+r }(C(R)+X^{r})X^{-1/2}\log\left(1+\frac {2 X} {\alpha R} \right)\le C(R) ||g|| _{ -r, q+r } X^{1/2}.
\end{align*}
In the third integral  in the right hand side of (\ref{S3E10}), $Y\in \left(X+ \alpha R, \beta R-X \right)$. If we denote $\Psi (z)=\chi (z)/z$ then $\Psi \in C^\infty (0, \infty)$, $\Psi '(z)=\chi '(z)/z-\chi (z)/z^2$ and there exists a constant $C>0$ such that
\begin{align*}
|\psi '(z)|\le C,\,\,\forall  z>0.
\end{align*}
It follows, for  $Y\in \left(X+ \alpha R, R-\beta X \right)$
\begin{align*}
|\Psi (|Y-X|)-\Psi (X+Y)|\le C X
\end{align*}
and,
\begin{align*}
&X^{-1/2}\int _{X+\alpha R}^ { \beta R-X} |g(t, Y)-g(t, X)|\left(\frac {\chi   (|X-Y|)} {|X-Y|}-\frac {\chi   (|X+Y|)} {|X+Y|}\right)dY\le \\
&\le C X^{1/2}\int _{X+ \alpha R}^ {\beta R-X} (|g(t, Y)|+|g(t, X)|)  dY\le 
C X^{1/2}\int _{X+\alpha R}^ {\beta R-X} (|g(t, Y)|+||g|| _{ -r, q+r }X^r)  dY.
\end{align*}
For  $Y\in \left(X+\alpha R, \beta R-X \right)$,
\begin{align*}
|g(Y)|\le ||g|| _{ -r, q+r }Y^r(1+Y)^{-r-q}\le C||g|| _{ -r, q+r }(R-X)^r(X+R )^{-r-q}\le C(R)||g|| _{ -r, q+r }
\end{align*}
\begin{align*}
&X^{-1/2}\int _{X+\alpha  R}^ { \beta R-X} |g(t, Y)-g(t, X)|\left(\frac {\chi   (|X-Y|)} {|X-Y|}-\frac {\chi   (|X+Y|)} {|X+Y|}\right)dY\le \\
&\le 
C(R)||g|| _{ -r, q+r } X^{1/2}(1+X^r)\int _{X+\alpha  R}^ {\beta R-X}  dY\le C(R)||g|| _{ -r, q+r } X^{1/2}.
\end{align*}
In the fourth integral  in the right hand side of (\ref{S3E10}), $Y\in (\beta R-X, \beta R+X)$, the same argument as for the third integral gives the same estimate,
\begin{align*}
&X^{-1/2}\int _{\beta R-X }^ { \beta R+X} |g(t, Y)-g(t, X)|\left(\frac {\chi   (|X-Y|)} {|X-Y|}-\frac {\chi   (|X+Y|)} {|X+Y|}\right)dY\le C(R)||g|| _{ -r, q+r } X^{1/2}.
\end{align*}
For the fifth  integral  in the right hand side of (\ref{S3E10}), $Y>\beta R+X$, then, $Y-X>\beta R, Y+X>\beta R$, $\chi (Y-X)=\chi (Y+X)=1$ and
\begin{align*}
|g(Y)|\le ||g|| _{ -r, q+r }Y^r(1+Y)^{-r-q}\le ||g|| _{ -r, q+r }Y^{-q}\le ||g|| _{ -r, q+r }(X+R)^{-q}.
\end{align*}
Therefore,
\begin{align*}
&X^{-1/2}\int _{\beta R+X }^ {\infty} |g(t, Y)-g(t, X)|\left(\frac {\chi   (|X-Y|)} {|X-Y|}-\frac {\chi   (|X+Y|)} {|X+Y|}\right)dY\le \\
&\le ||g|| _{ -r, q+r }(R^{-q}+X^{r})X^{-1/2}\int  _{ \beta R+X }^\infty \frac {2XdY} {(Y-X)(X+Y)}\\
&= ||g|| _{ -r, q+r }(R^{-q}+X^{r})X^{1/2}\log\left(\frac {\beta R+X} {R-X} \right)\le   ||g|| _{ -r, q+r }(R^{-q}+X^{r})X^{3/2}.
\end{align*}
We deduce,
 \begin{align*}
&|T_1(g, \varphi )|\le  C(R)||g|| _{ -r, q+r }X^{1/2},\,\,\forall X\in (0, \alpha R/2).
\end{align*}
The same arguments show that if $r<0$,
 \begin{align*}
&|T_1(g, \varphi )|\le  C(R)||g|| _{ -r, q+r }X^{\frac {1} {2}+r},\,\,\forall X\in (0, \alpha R/2),
\end{align*}
and then for $r\in \RR$,
\begin{align*}
&|T_1(g, \varphi )|\le  C(R)||g|| _{ -r, q+r }X^{\frac {1} {2}-r^-},\,\,\forall X\in (0, \alpha R/2),
\end{align*}

\subsubsection{In the region where $X$ large.}
For $X>> R$ the integral in $T_1$ may be decomposed as follows:
\begin{align}
\label{S3E11}
\int _0^\infty[\cdots]dy&=\int _0^{X-\beta R}[\cdots]dY+\int _{X-\beta R}^ {X- \alpha R}[\cdots]dY+\nonumber\\
&+\int_{X-\alpha R}^{X+\alpha R}[\cdots]dY+\int _{X+\alpha R}^{X+\beta R}[\cdots]dY+\int  _{ \beta R+X }^\infty[\cdots]dY. 
\end{align}

In the first integral at the right hand side of  (\ref{S3E11}), $Y\in (0, X-\beta R)$, then $X-Y>\beta R$, $X+Y>\beta R$ and then $\chi (Y-X)=\chi (X+Y)=1$. Then,
\begin{align*}
X^{-1/2}\int _0^{X-\beta R} |g(t, X)|\left(\frac {\chi   (|X-Y|)} {|X-Y|}-\frac {\chi   (|X+Y|)} {|X+Y|}\right)dY\le\\
\le 2||g|| _{ -r, q+r }X^{-q-\frac {1} {2}}\int _0^{X-\beta R}\frac {YdY} {(X-Y)(X+Y)}\le C||g|| _{ -r, q+r }X^{-q-\frac {1} {2}}\log X.
\end{align*}
Since also,
\begin{align*}
|g(Y)|\le ||g|| _{ -r, q+r }Y^r(1+Y)^{-r-q}
\end{align*}
we may argue as follows,
\begin{align*}
\int _0^{X-R}[\cdots]dY=\int _0^{1}[\cdots]dY+\int _1^{X-R}[\cdots]dY.
\end{align*}
Now, for $Y\in (0, 1)$, $X>>1$ and $r>-2$,

\begin{align*}
&\int _0^{1}  |g(t, Y)|\left(\frac {\chi   (|X-Y|)} {|X-Y|}-\frac {\chi   (|X+Y|)} {|X+Y|}\right)dY=\\
&=\int _0^{1} |g(t, Y)|\left(\frac {1} {|X-Y|}-\frac {1} {|X+Y|}\right)dY
\le ||g|| _{ -r, q+r }\int _0^1\frac {Y^{1+r}(1+Y)^{-q-r}dY} {(X-Y)(X+Y)}\\
&\le C||g|| _{ -r, q+r }X^{-2}\int _0^1Y^{1+r}(1+Y)^{-q-r}dY\le C||g|| _{ -r, q+r }X^{-2}.
\end{align*}

On the other hand, for $X>>R$,
\begin{align*}
&\int _1^{X-\beta R}  |g(t, Y)|\left(\frac {\chi   (|X-Y|)} {|X-Y|}-\frac {\chi   (|X+Y|)} {|X+Y|}\right)dY\le\\
&\le ||g|| _{ -r, q+r }\int _1^{X-\beta R} \frac {Y^{1-q}dY} {(X-Y)(X+Y)}
\le C||g|| _{ -r, q+r }\left(X^{-2} +X^{-q}\log X\right).
\end{align*}
It follows for the first term at the right hand side of  (\ref{S3E11}),
\begin{align*}
X^{-1/2} \int _{0}^{X-\beta R} |g(t, Y)-g(t, X)|\left(\frac {\chi   (|X-Y|)} {|X-Y|}-\frac {\chi   (|X+Y|)} {|X+Y|}\right)dY\le \\
\le C||g|| _{ -r, q+r }X^{-1/2}\left(X^{-2} +X^{-q}\log X\right).
\end{align*}
In the second term at the right hand side of  (\ref{S3E11}) $Y\in (X-\beta R, X-\alpha R)$. Then, for $X>>R$: $Y>X-\beta R>\beta R$ and then, $X-Y\in (\alpha R, \beta R)$ from where $|g(Y)|\le ||g|| _{ -r, q+r }Y^r(1+Y)^{-r-q}\le C||g|| _{ -r, q+r }X^{-q}$. It follows,
\begin{align*}
&X^{-1/2} \int _{X-\beta R}^{X- \alpha R} |g(t, Y)-g(t, X)|\left(\frac {\chi   (|X-Y|)} {|X-Y|}-\frac {\chi   (|X+Y|)} {|X+Y|}\right)dY\le\\
&\le C||g|| _{ -r, q+r }X^{-q-1/2}\int _{X-\beta R}^{X-\alpha R} \left(\frac {1} {X-Y}+\frac { 1} {X+Y}\right)dY\\
&=C||g|| _{ -r, q+r }X^{-q-1/2}\log\left(\frac {4X-R} {2X-R}\right)\le C||g|| _{ -r, q+r }X^{-q-1/2}.
\end{align*}
In the third integral  at the right hand side of  (\ref{S3E11}) $Y\in (X-\alpha R, X+ \alpha R)$ from where $\chi (|X-Y|)=0$. Since $X-\alpha R>X/2$ and  $|g(Y)|\le ||g|| _{ -r, q+r }Y^r(1+Y)^{-q-r}\le ||g|| _{ -r, q+r }Y^{-q}\le C(R)||g|| _{ -r, q+r }X^{-q}$,
\begin{align*}
&X^{-1/2} \int _{X-\alpha R}^{X+\alpha R} |g(t, Y)-g(t, X)|\left(\frac {\chi   (|X-Y|)} {|X-Y|}-\frac {\chi   (|X+Y|)} {|X+Y|}\right)dY\le\\
&\le C(R)||g|| _{ -r, q+r }X^{-q-\frac {1} {2}} \int _{X-\alpha R}^{X+\alpha R}\frac {dY} {X+Y}\le C(R)||g|| _{ -r, q+r }X^{-q-\frac {3} {2}}.
\end{align*}
Consider now the fourth  integral  at the right hand side of  (\ref{S3E11}), where $Y\in (X+\alpha R, X+\beta R)$ and then $|g(y)|\le C||g|| _{ -r, q+r }X^{-q}$,
\begin{align*}
&X^{-1/2} \int _{X+ \alpha R}^{X+\beta R} |g(t, Y)-g(t, X)|\left(\frac {\chi   (|X-Y|)} {|X-Y|}-\frac {\chi   (|X+Y|)} {|X+Y|}\right)dY\le\\
&\le C(R)||g|| _{ -r, q+r }X^{-q-\frac {1} {2}} \int _{X+ \alpha R}^{X+\beta R}\left(\frac {1} {Y-X}+\frac {1} {Y+X}\right)dY\le C(R)||g|| _{ -r, q+r }X^{-q-\frac {1} {2}}.
\end{align*}
In the fifth  integral  at the right hand side of  (\ref{S3E11}), where $Y> X+\beta R$ from where $|g(y)|\le C||g|| _{ -r, q+r }X^{-q}$ and $\chi (Y-X)=\chi (Y+X)=1$ from where it follows,
\begin{align*}
&X^{-1/2} \int _{X+\beta R}^\infty |g(t, Y)-g(t, X)|\left(\frac {\chi   (|X-Y|)} {|X-Y|}-\frac {\chi   (|X+Y|)} {|X+Y|}\right)dY\le\\
&\le C(R)||g|| _{ -r, q+r }X^{-q-\frac {1} {2}} \int _{X+\beta R}^\infty \left(\frac {1} {Y-X}-\frac {1} {Y+X}\right)dY\\
&=C(R)||g|| _{ -r, q+r }X^{-q-\frac {1} {2}}\log\left( 1+\frac {2X} {\beta  R}\right).
\end{align*}
We have then proved that
\begin{align*}
|T_1(g, \varphi )(X)|\le C||g|| _{ -r, q+r }\left(X^{\frac {1} {2}-r^-}\1 _{ 0<X<1 }+X^{-q-\frac {1} {2}}(\log X)\,\1 _{ X>1 }\right).
\end{align*}

\subsection{The operator $T_2$.}
Acting on $X _{ -r, q+r }$, the operator $T_2$ a may be bounded as follows.
\begin{prop} 
\label{T2}
There exists a constant $C>0$ such that,
\begin{align*}
|T_2(g, \varphi )(X)|\le C||g|| _{ -r, q+r }\left(X^{\frac {1} {2}}\1 _{ 0<X<1 }+X^{-q-\frac {3} {2}}(\log X)\,\1 _{ X>1 }\right).
\end{align*}
\end{prop}

\subsubsection{In the region where $X$ small.}
For $Y\in (0, X)$, and $X<<1$ it follows that $Y<X<\alpha R$ and then $\varphi (X-Y)=\varphi (X)=\varphi (Y)=1$. Therefore,
\begin{align*}
&T_2(g, \varphi )(X)= -X^{-1/2}\int _{X}^\infty\Bigg (\frac {g(t, Y)} { Y}\left(\varphi  (Y-X)-\varphi  (t, X) \right)+\\
&\hskip 6cm +\frac {(\varphi  (t, Y)-\varphi  (X))g(t, Y-X)} { Y-X}\Bigg)dY.
\end{align*}
This integral is now split as follows,
\begin{align*}
T_2(g, \varphi )(X)=X^{-1/2}\int _X^{X+\beta R}[\cdots]dY+X^{-1/2}\int _{X+\beta R}^\infty [\cdots]dY=J_1+J_2
\end{align*}
We notice in the second, $\varphi (Y)=\varphi (Y-X)=0$, $\varphi (X)=1$ and then for all $r$,
\begin{align*}
|J_2|\le X^{-1/2}\left|\int _{X}^\infty \frac {g(t, Y)} { Y} -\frac {g(t, Y-X)} { Y-X}dY\right|\\
\le X^{-1/2}\int _{\beta R}^{X+\beta R}\frac {|g(Z)|dZ} {Z}\le C||g|| _{ -r, q+r }X^{-1/2}\int _{\beta R}^{X+\beta R}Z^{r-1}dZ\le  C||g|| _{ -r, q+r }X^{\frac {1} {2}}.
\end{align*}
The term $J_1$ may be written as
\begin{align*}
&J_1= X^{-1/2}\int _{X}^{X+\beta R} \Bigg(\frac {g(t, Y)} { Y}\left(1-\varphi  (Y-X)\right)+\frac {(\varphi  (Y)-1)g(t, Y-X)} { Y-X}\Bigg)dY\\
&=X^{-1/2}\int _{X}^{X+\beta R} \frac {g(t, Y)} { Y}\left(1-\varphi  (Y-X)\right)dY+\\
&\hskip 5cm +X^{-1/2}\int _{X}^{X+\beta R} \frac {(\varphi  (Y)-1)g(t, Y-X)} { Y-X}dY\\
&=X^{-1/2}\int _{X}^{X+\beta R} \frac {g(t, Y)} { Y}\left(1-\varphi  (Y-X)\right)dY+\\
&\hskip 5.5cm+X^{-1/2}\int _0^{\beta R} \frac {(\varphi  (X+Y)-1)g(t, Y)} { Y}dY.
\end{align*}
If  $Y\in (X, X+ \alpha R)$  one has $\varphi (Y-x)=1$ and  if $X+Y>\alpha R$, then $\varphi (X+Y)=1$. Therefore $J_1$  may still written as follows,
\begin{align*}
&J_1=X^{-1/2}\int _{\alpha R+X}^{X+\beta R} \frac {g(t, Y)} { Y}\left(1-\varphi  (Y-X)\right)dY+\\
&\hskip 4cm+X^{-1/2}\int  _{\alpha R-X }^{\beta R} \frac {(\varphi  (X+Y)-1)g(t, Y)} { Y}dY\\
&=X^{-1/2}\int _{ \beta R}^{X+ \beta R} \frac {g(t, Y)} { Y}\left(1-\varphi  (Y-X)\right)dY+\\
&\hskip 4cm+X^{-1/2}\int  _{ \alpha R-X }^{\alpha X+X} \frac {(\varphi  (X+Y)-1)g(t, Y)} { Y}dY+\\
&\hskip 5.5cm+\int  _{\alpha R+X }^{\beta R}(\varphi (X+Y)-\varphi (Y-X))\frac {g(t, Y)\, dY} {Y}
\end{align*}
from where it  follows 
\begin{align*}
|J_1(t, X)|\le C(R)||g|| _{ -r, q+r }X^{\frac {1} {2}}.
\end{align*}

\subsubsection{In the region where $X$ large.}
Let us write,
\begin{align}
\label{S3E12}
T_2(g, \varphi )(X)=X^{-1/2}\int _0 ^{ X-\beta R }[\cdots]dY+X^{-1/2}\int _{ X-\beta R }^X[\cdots]dY+\nonumber \\
+X^{-1/2}\int _X^{ \infty}[\cdots]dY
\end{align}
In the first term, $Y\in (0, X-\beta R)$ then $X-Y>\beta R$. It follows that  $\varphi (X)=\varphi (X-Y)=0$ and
\begin{align*}
&X^{-1/2}\int _0 ^{ X-\beta R }[\cdots]dY=X^{-1/2}\int _{0}^{X-\beta R}\frac {\varphi  (t, Y)g(t, X-Y)} { |X-Y|}dY
\end{align*}
If $Y>\beta R$ then $\varphi (Y)=0$ and then
\begin{align*}
&X^{-1/2}\left|\int _0 ^{ X-\beta R }[\cdots]dY\right|\le X^{-1/2}\int _{0}^{\beta R}\frac { |g(t, X-Y)|} { |X-Y|}dY\\
&=X^{-1/2 }\int _{X-\beta R}^{X}\frac {|g(t, Z)|dZ} {Z}\le C||g|| _{ -r, q+r }X^{-\frac {1} {2}-q}\int _{X-\beta R}^X \frac {dZ} {Z}\le  C||g|| _{ -r, q+r }X^{-\frac {3} {2}-q}
\end{align*}
In the second term at the right hand side of (\ref{S3E12})  term, $Y\in (X-\beta R, X)$, then $\varphi (Y)=\varphi (X)=0$ and
\begin{align*}
&X^{-1/2}\left|\int _{ X-\beta R }^X[\cdots]dY\right|\le X^{-1/2}\int _{X-\beta R}^{X}\frac { |g(t,Y)|\varphi (X-Y)} {Y}dY\\
&\le C ||g|| _{ -r, q+r }X^{-\frac {1} {2}-q}\int _{X-\beta R}^{X}\frac {dZ} {Z}\le  C||g|| _{ -r, q+r }X^{-\frac {3} {2}-q}.
\end{align*}
In the last term at the right hand side of (\ref{S3E12})  term, $Y>X$, then $\varphi (Y)=\varphi (X)=0$ and
\begin{align*}
&X^{-1/2}\left|\int _{ X }^\infty[\cdots]dY\right|\le X^{-1/2}\int _{X}^{\infty}\frac { |g(t,Y)|\varphi (Y-X)} {Y}dY
\end{align*}
If $Y>X+\beta R$ then $\varphi (Y-X)=0$ from where,
\begin{align*}
&X^{-1/2}\left|\int _{ X }^\infty[\cdots]dY\right|\le X^{-1/2}\int _{X}^{X+\beta R}\frac { |g(t,Y)|} {Y}dY\le C||g|| _{ -r, q+r }X^{-\frac {1} {2}-q-1}.
\end{align*}
We have then proved that
\begin{align*}
|T_2(g, \varphi )(X)|\le C||g|| _{ -r, q+r }\left(X^{\frac {1} {2}}\1 _{ 0<X<1 }+X^{-q-\frac {3} {2}}\,\1 _{ X>1 }\right).
\end{align*}

\begin{cor}
\label{T}
There exists a constant $C>0$ such that
\begin{align*}
|T(g, \varphi )(X)|\le C||g|| _{ -r, q+r }\left(X^{\frac {1} {2}-r^-}\1 _{ 0<X<1 }+X^{-q-\frac {1} {2}}\,\1 _{ X>1 }\right).
\end{align*}
\end{cor}

\begin{cor} 
\label{TT}
Suppose that $\alpha $ and $\beta $ are such that
\begin{align}
\alpha +\frac {1} {2}-r^-\ge 0,\,\,\text{and}\,\,q+\frac {1} {2}-\beta -\alpha \ge 0.
\end{align}
for $r$ and $q$ as before.  Then, there exists a constant $C>0$ such that the map
 \begin{align*}
& T:\,\,X _{ -r, r+q }\longrightarrow X _{ \alpha , \beta  }\\
& \hskip 1.1cm g   \xrightarrow{\hspace*{1cm}}  T(g, \varphi )
 \end{align*}
is linear and continuous, such that
\begin{align*}
||T(g, \varphi )|| _{ \alpha , \beta   }\le C||g|| _{-r, q+r  }.
\end{align*}
\end{cor}
\begin{proof}
By Corollary \ref{T}, for $r>-2$ and $q+r>0$,
\begin{align*}
|T(g, \varphi )(X)|\le C||g|| _{ -r, q+r }\left(X^{\frac {1} {2}-r^-}\1 _{ 0<X<1 }+X^{-q-\frac {1} {2}}\,\1 _{ X>1 }\right). 
\end{align*}
If $\alpha +\frac {1} {2}-r^-\ge 0$ then, for all $X\in (0, 1)$,
\begin{align*}
X^\alpha |T(g, \varphi )(X)|\le C||g|| _{ -r, q+r }X^{\frac {1} {2}-r^-+\alpha }\le C.
\end{align*}
If $q+\frac {1} {2}-\beta -\alpha \ge 0$,
\begin{align*}
X^{\alpha+\beta } |T(g, \varphi )(X)|\le C||g|| _{ -r, q+r }X^{-q-\frac {1} {2}+\alpha+\beta  }\le C
\end{align*}
and then, $T(g, \varphi )\in X _{ \alpha , \beta  }$.
\end{proof}
\begin{rem}
\label{lzm453}
It follows in particular, for all $g\in X _{ \theta, \rho  }$ with $\theta\ge 0$ and $\rho >0$, $T(g, \varphi )\in X _{ \theta-\frac {1} {2}, \rho +1 }$ and $||T(g, \varphi )|| _{ \theta-\frac {1} {2} , \rho +1}\le C||g|| _{ \theta, \rho  }$.  Since $\theta >\theta-\frac {1} {2}$ and $\theta+\rho <(\theta-\frac {1} {2})+\rho +1>\theta+\rho $, it then also follows that $T(g, \varphi )\in X _{ \theta, \rho }$ .
\end{rem}
\section{Estimate of $Q_N(\varphi )$.}
\label{Q0}
\setcounter{equation}{0}
\setcounter{theo}{0}
For the same reason as in Section \ref{SecT}, consider any function $\varphi(X) \equiv \phi (X/R)$ where $\phi \in C^2(0, \infty)\cap C([0, \infty)$, $\phi (X)=1$ on $[0, \alpha ]$ for some $\alpha >0$, $\text{supp}(\phi )=[0, \beta]$ for some $\beta >\alpha $. In Definition (\ref{defvfiR}), $\alpha =1/2$ and $\beta =1$.
\begin{prop}
\label{QQ0}
There exists a positive constant $C>0$ such that, for all $R>1$,
\begin{align*}
|Q_N(\varphi )(X)|\le  CR^{-1}X^{1/2},\,\,\forall X\in  (0, \alpha R),\\
|Q_N(\varphi )(X)|\le  C X^{-1/2},\,\,\forall X\in  (\alpha R, 2\beta R)\\
Q_N(\varphi )(X)=0,\,\,\forall X>2\beta R,
\end{align*}
and then,
\begin{align*}
&Q_N(\varphi )\in X _{ \theta, \rho},\,\,\forall \theta\ge  -1/2,\,\,\forall \rho >0,\\
&||Q_N (\varphi )|| _{ \theta, \rho  }\le  CR^{-1}+CR^{-\frac {1} {2}+\theta+\rho },\,\,\theta\ge -1/2,\,\rho \ge 0,\\
&Q_N(\varphi )(X)=C(\varphi )X^{1/2}\xi (X)+q _{ 0, R }(X),\,\,X\to 0,\\
&C(\varphi )=\left(\int  _{\alpha R }^{\infty}\frac {\chi ^2(Y)dY} {Y^2} -\int  _{\alpha R }^{\infty}\frac {\chi (Y)dY} {Y^2}
-\int  _{ \alpha R }^{\beta R}\frac {\chi (Y)dY} {Y^2}-\frac {\beta -\alpha } {\alpha \beta R}\right)\\
&q _{ 0, R }(X)=\mathcal O(X)^{3/2},\,\,X\to 0.
\end{align*}
\end{prop}
\begin{proof} 
\subsection{In the region where $X\in (0, \alpha R)$.} For $X\in (0, \alpha R)$  and $Y\in (0, X)$, $\varphi (X-Y)=\varphi (Y)=\varphi (X)=1$. Then $Q_N(\varphi )(X)=0$. Using again that $Q(1)=0$,
\begin{align*}
Q_N(\varphi )=X^{1/2}\int _{X}^\infty\Bigg(\frac {1-\chi   ( |X-Y|)} {Y-X} \left(\frac {1-\chi    (Y)} {Y}-\frac {1- \chi    (X)} {X} \right)+\\+\frac {(1-\chi    (X))(1-\chi (Y))} {XY}\Bigg)dY
\end{align*}

\begin{align*}
=-X^{1/2}\int _{X}^\infty \frac {\chi   ( |X-Y|)} {Y-X} \left(\frac {1-\chi    (Y)} {Y}-\frac {1- \chi    (X)} {X} \right)dY+\\
+X^{1/2}\int _{X}^\infty \frac {1} {Y-X} \left(\frac {-\chi    (Y)} {Y}-\frac {- \chi    (X)} {X} \right)dY+\\
+X^{1/2}\int _{X}^\infty \frac {\chi  (X)\chi (Y)dY} {XY}-X^{1/2}\int _{X}^\infty \frac {\chi  (X)+\chi (Y)dY} {XY}
\end{align*}
and then, using that $\chi  (X)=0$ for $X<R/2$,
\begin{align*}
Q_N(\varphi )=X^{1/2}\int _{X}^\infty \frac {\chi   ( Y-X)} {Y-X}\frac {\chi    (Y)} {Y}dY
-X^{1/2}\int _{X}^\infty \frac {\chi   ( Y-X)} {Y-X} \left(\frac {1 } {Y}-\frac {1 } {X} \right)dY+\\
-X^{1/2}\int _{X}^\infty \frac {1} {Y-X} \frac {\chi    (Y)} {Y} dY
-X^{1/2}\int _{X}^\infty \frac {\chi (Y)dY} {XY}\Longrightarrow
\end{align*}

\begin{align*}
Q_N(\varphi )=X^{1/2}\int _{X}^\infty \frac {\chi   ( Y-X)} {Y-X}\frac {\chi    (Y)} {Y}dY+X^{1/2}\int _X^\infty\frac {\chi (Y-X)dY} {XY}-\\
-X^{1/2}\int _{X}^\infty \frac {1} {Y-X} \frac {\chi    (Y)} {Y} dY
-X^{1/2}\int _{X}^\infty \frac {\chi (Y)dY} {XY}
\end{align*}
That we re write,
\begin{align*}
&Q_N(\varphi )=X^{1/2}\int _{X}^\infty \frac {\chi   ( Y-X)} {Y-X}\frac {\chi    (Y)} {Y}dY+\\
&+X^{1/2}\int _X^\infty\frac {(\chi (Y-X)-\chi (Y))dY} {XY}
-X^{1/2}\int _{X}^\infty \frac {\chi    (Y)dY} {Y(Y-X)} =J_1+J_2-J_3
\end{align*}
Let us begin with the estimate of $J_3$. Since $\chi (Y)=0$ for $Y\in (0, \alpha R)$, and $X<\alpha R$,
\begin{align*}
J_3&= X^{1/2}\int _{\alpha R}^\infty \frac {\chi (Y)dY} {Y(Y-X)}\\
&= X^{1/2}\left(\int _{\alpha R}^\infty \frac {\chi (Y)dY} {Y^2}+X\int _{\alpha R}^\infty \frac {\chi (Y)dY} {Y^2(Y-X)}\right)
\end{align*}
then,
\begin{align*}
\left|J_3-X^{1/2}\int _{\alpha R}^\infty \frac {\chi (Y)dY} {Y^2}\right|
&\le  X^{3/2} \int _{\alpha R}^\infty \frac {dY} {Y^2(Y-X)}\\
&=-X^{3/2} \frac {X+\alpha R\log(1-\frac {X} {\alpha R})} {\alpha RX^2}
\end{align*}
and since $z+\log(1-z)=-z^2/2+\mathcal O(z)^3$,
\begin{align*}
\left|J_3-X^{1/2}\int _{\alpha R}^\infty \frac {\chi (Y)dY} {Y^2}\right|
&\le CX^{3/2}R^{-2}.
\end{align*}
In the term $J_2$, if $Y<\alpha R$ then $Y-X<\alpha R$ and  $\chi  (Y-X)=\chi (Y)=0$. And if $Y>\beta X+\beta R$ then $Y-X>\beta R$ and $\chi (Y-X)=\chi (Y)=1$. It follows,
\begin{align*}
J_2&=X^{1/2}\int _{\alpha R}^{X+\beta R}\frac {(\chi (Y-X)-\chi (Y))dY} {XY}.
\end{align*}
For all $Y\in (\alpha R, X+\beta R)$,
\begin{align*}
&\chi (Y)-\chi (Y-X)=X\int _0^1\chi '(Y-zX)dz\\
&\chi '(Y-zX)=\chi '(Y)-zX\chi ''(\xi ),\,\,\xi \in (Y-X, Y).
\end{align*}
Then, 
\begin{align*}
&\int _0^1\chi '(Y-zX)dz-\int _0^1\chi '(Y)dz=-X\int _0^1 z\chi ''(\xi )dz
\end{align*}
and
\begin{align*}
&\left|\chi (Y)-\chi (Y-X)-X\int _0^1\chi '(Y)dz\right|=X^2\left|\int _0^1\chi ''(\xi )dz \right|\le ||\chi ''|| _{ \infty }X^2\int _0^1 z dz.
\end{align*}
By the definition of $\varphi $ in (\ref{defvfiR}) and the definition of $\chi (X)=1-\varphi (X)$,
\begin{align*}
\chi ''(\xi )=-\frac {1} {R^2}\phi ''\left(\frac {X} {R} \right),\,\,\, ||\chi ''|| _{ \infty }\le CR^{-2}
\end{align*}
it follows,
\begin{align*}
&\left|\chi (Y)-\chi (Y-X)-X \chi '(Y)\right| \le CX^2R^{-2},\\
&\left|J_2+X^{1/2}\int  _{ \alpha R }^{X+\beta R}\frac {\chi' (Y)dY} {Y}\right| \le 
CX^{-1/2}X^2R^{-2}\int  _{ \alpha R }^{X+\beta R}\frac {dY} {Y}\le CX^{3/2} R^{-2}
\end{align*}
Since $\chi  (\alpha R)=0$, $\chi (X+\beta R)=1$, integration by parts yields,
\begin{align*}
\int  _{ \alpha R }^{X+\beta R}\frac {\chi' (Y)dY} {Y}=\int  _{ \alpha R }^{X+\beta R}\frac {\chi (Y)dY} {Y^2}+\frac {1} {X+\beta R}
\end{align*}
 then,
\begin{align*}
\left|\int  _{ \alpha R }^{X+\beta R}\frac {\chi' (Y)dY} {Y}\right|\le \frac {CX^{1/2}} {R}.
\end{align*}

In the term $J_1$, if $Y\in (X, X+\frac {R} {2})$ then $\chi (Y-X)=0$, then
\begin{align*}
J_1= X^{1/2}\int  _{ X+\alpha R }^\infty\frac {\chi (Y)\chi (Y-X)dY} {Y(Y-X)}
\end{align*}
since
\begin{align*}
\frac {\chi (Y)\chi (Y-X)} {Y(Y-X)}=\frac {\chi ^2(Y)} {Y^2}-\frac {X\chi (Y)\chi '(\xi )} {Y^2}+\frac {X\chi (Y)\chi (Y-X)} {Y^2(Y-X)}\\
\left|\frac {\chi (Y)\chi (Y-X)} {Y(Y-X)}-\frac {\chi ^2(Y)} {Y^2}\right|\le \frac {\chi (Y)\chi (Y-X)X} {Y^2(Y-X)}+\frac {X\chi '(\xi )} {Y^2}.
\end{align*}
for some $\xi \in (Y-X, X)$. Since $Y>X+R/2$,  then $\xi \ge Y-X\ge R/2$. And, since $\xi =1-\varphi $, by the definition of $\phi $ in (\ref{defvfiR}), $|\chi '(\xi )|\le CR^{-1}$.  Therefore,
\begin{align*}
\left|\frac {\chi (Y)\chi (Y-X)} {Y(Y-X)}-\frac {\chi ^2(Y)} {Y^2}\right|\le \frac {\chi (Y)\chi (Y-X)X} {Y^2(Y-X)}+\frac {X\chi '(\xi )} {Y^2}\le \frac {CX} {RY^2}
\end{align*}
and,
\begin{align*}
\left|J_1-X^{1/2}\int  _{ X+ \alpha R }^\infty\frac {\chi ^2(Y)dY} {Y^2}\right|\le
CX^{1/2} \int  _{ X+\alpha R}^\infty \frac {X} {RY^2}dY\le \frac {CX^{3/2}} {R}
\end{align*}
All this shows,
\begin{align*}
&\Bigg|Q_N(\varphi )(X)-X^{1/2} \Bigg(\int  _{\alpha R }^{\infty}\frac {\chi ^2(Y)dY} {Y^2} -
\int  _{ \alpha R }^{\infty}\frac {\chi (Y)dY} {Y^2}-\\
&\hskip 4cm -\int  _{\alpha R}^{\beta R}\frac {\chi (Y)dY} {Y^2}-\frac {1} { \beta R}\Bigg)\Bigg|\le  \frac {CX^{3/2}} {R}\\
&\left|\int  _{\alpha R }^{\infty}\frac {\chi ^2(Y)dY} {Y^2} -\int  _{\alpha R}^{\infty}\frac {\chi (Y)dY} {Y^2}
-\int  _{\alpha R }^{\beta R}\frac {\chi (Y)dY} {Y^2}-\frac {1 } {  \beta R}\right|\le \frac {C} {R}.
\end{align*}
Since $\chi (Y)=1$ for $Y\ge \beta $,
\begin{align*}
\int  _{\alpha R }^{\infty}\frac {\chi ^2(Y)dY} {Y^2} -\int  _{\alpha R }^{\infty}\frac {\chi (Y)dY} {Y^2}=\int  _{ \alpha R}^{\beta R}\frac {\chi ^2(Y)dY} {Y^2}-
\int  _{\alpha R }^{\beta R}\frac {\chi (Y)dY} {Y^2}
\end{align*}
and then,
\begin{align*}
&\int  _{\alpha R }^{\infty}\frac {\chi ^2(Y)dY} {Y^2} -\int  _{ \alpha R }^{\infty}\frac {\chi (Y)dY} {Y^2}
-\int  _{\alpha R }^{\beta R}\frac {\chi (Y)dY} {Y^2}-\frac {1 } { \beta R}=\\
&=\int  _{\alpha R}^{\beta R}\frac {\chi (Y)^2-2\chi (Y)-\frac {\alpha } {2\beta -\alpha }} {Y^2}dY
\end{align*}
By definition,
\begin{align*}
\chi (Y)^2-2\chi (Y)-\frac {\alpha } {\beta -\alpha }=(1-\varphi )^2-2(1-\varphi )-\frac {\alpha } {\beta-\alpha }\\
=1-2\varphi+\varphi ^2 -2+2\varphi -\frac {\alpha } {\beta -\alpha }=\varphi ^2-\frac {\beta } {\beta-\alpha }
\end{align*}
with
\begin{align*}
&C\equiv C(\varphi )=\int  _{ \alpha R }^{\beta R}\frac {\varphi ^2(Y)dY} {Y^2}-\frac {\beta } {\beta -\alpha }
\int  _{ \alpha R }^{ \beta R}\frac {dY} {Y^2}\\
&=\int  _{ \alpha R }^{ \beta R}\frac {\varphi ^2(Y)dY} {Y^2}-\frac {\beta } {R(\beta -\alpha )}\left(\frac {1} {\alpha }-\frac {1} {\beta }\right)
=\frac {1} {R}\left( \int  _{ \alpha  }^{ \beta }\frac {\phi  ^2(z)dz} {z^2}-\frac {1} {\alpha }\right)
\end{align*}
\subsection{In the region where $X$ large.}
If $X>2\beta R$ and $Y>X$ then $\varphi (Y)=\varphi (X)=0$ and
\begin{align*}
Q_N(\varphi )(X)=X^{1/2}\int _{X/2}^\infty\Bigg(\frac {\varphi  (|Y-X|)} {|X-Y|} \left(\frac {\varphi   (Y)} {Y}-\frac { \varphi   (X)} {X} \right)+\\
+\text{sign}(Y-X)\frac {\varphi   (X) \varphi  (Y)} {XY}\Bigg)dY=0.
\end{align*}
If $X\in (\beta R, 2\beta R)$ then $\varphi (X)=0$ and,
\begin{align*}
Q_N(\varphi )(X)=X^{1/2}\int _{X/2}^\infty \frac {\varphi  (|Y-X|)} {|X-Y|}\frac {\varphi   (Y)} {Y} 
dY,
\end{align*}
If $Y>\beta R$ then $\varphi (Y)=0$ and so,
\begin{align*}
Q_N(\varphi )(X)=X^{1/2}\int _{X/2}^{\beta R} \frac {\varphi  (|X-Y|)} {|X-Y|} 
\left(\frac {\varphi   (Y)} {Y} -\frac {\varphi (X)} {X}\right)dY,
\end{align*}
where the term $\frac {\varphi (X)} {X}$, although equal to zero, has been added in order to make the argument more transparent. The change of variables $X=RX'$, $Y=RY'$ gives,
\begin{align}
\label{lzm478}
Q_N(\varphi )(X)=R^{-1}X^{1/2}\int _{X'/2}^{\beta } \frac {\phi  (|X'-Y'|)} {|X'-Y'|} 
\left(\frac {\phi   (Y')} {Y'} -\frac {\phi (X')} {X'}\right)dY'.
\end{align}
The integral in the right hand side of (\ref{lzm478}) is estimated using the regularity of $\phi $ as follows,
\begin{align*}
&\int _{X'/2}^{\beta } \frac {\phi  (|X'-Y'|)} {|X'-Y'|} 
\left(\frac {\phi   (Y')} {Y'} -\frac {\phi (X')} {X'}\right)dY'=I_1+I_2\\
&I_1=\int _{X'/2}^{\beta } \frac {\phi  (|X'-Y'|)} {|X'-Y'|} 
\frac {(X'-Y')\phi (X')} { X'Y'}dY'\\
&I_2=\int _{X'/2}^{\beta } \frac {\phi  (|X'-Y'|)} {|X'-Y'|} 
\frac {X'(\phi (Y')-\phi (X'))} { X'Y'}dY'
\end{align*}
with,
\begin{align*}
|I_1|\le C ||\phi || ^2_{ \infty  }X'^{-2},\quad |I_2|\le C ||\phi || _{ \infty  } ||\phi' || _{ \infty  }X'^{-2}
\end{align*}
and then,
\begin{align*}
&|Q_N(\varphi )(X)|\le C R^{-1}X^{1/2} X'^{-2}
\left( ||\phi || ^2_{ \infty  }+||\phi || _{ \infty  } ||\phi' || _{ \infty  } \right)\\
&\le C \left( ||\phi || ^2_{ \infty  }+||\phi || _{ \infty  } ||\phi' || _{ \infty  } \right) (RX^{-1})X^{-1/2},
\quad \forall X\in (\beta R, 2\beta R).
\end{align*}

and, for $\theta\ge -1/2$, $\rho \ge 0$,
\begin{align*}
&||Q_N(\varphi )|| _{ \theta, \rho  }=\sup _{ X>0 }|Q_N(\varphi )(X)|X^{\theta}(1+X)^\rho \le
C\sup _{ 0<X<\alpha R }X^{1/2}R^{-1}(1+X)^\rho+\\
&\hskip 5cm  +\sup _{ \alpha R<X<2\beta R }X^{-1/2}X^\theta(1+X)^\rho \le CR^{-\frac {1} {2}+\theta +\rho },\\
 &C\equiv C(\varphi )=\frac {1} {R}\left( \int  _{ \alpha  }^{ \beta }\frac {\phi  ^2(z)dz} {z^2}-\frac {1} {\alpha }\right).
\end{align*}
\vskip -0.5cm 
\end{proof}

\section{Estimate of $Q_N(g )$.}
\label{QN}
\setcounter{equation}{0}
\setcounter{theo}{0}
Suppose that $g $ is such that for some $r>0, q>0$,
\begin{align*}
&|g(X)|\le CX^{r},\,\,0<X<1\\
&|g(X)|\le CX^{-q},\,\,X>1.
\end{align*}
and consider the four different terms of $N(g)$ as $X\to 0$ and $X\to \infty$.

(i) If $Y\in (0, X/2$ and $X<<1$  it follows by hypothesis $|g(Y)|\le CY^r$.  Since in that case  $1>>|X-Y|>X/2$ also, $|g(X-Y)|\le C|X-Y|^r$ and  
$|g(X-Y)|/(X-Y)\le  (X-Y)^{r-1}\le CX^{r-1}$ we deduce,
\begin{align*}
&\sqrt X \left| \int _0^{X/2} \frac {g(Y)} {Y}\left( \frac {g(X-Y)} {X-Y}-\frac {g(X)} {X} \right)dY\right|\le\\
&\le  C||g||^2 _{ -r, q+r }X^{-\frac {1} {2}+r} \int _0^{X/2}  Y^{r-1}dY\le CX^{-\frac {1} {2}+2r},\,\,X<<1.
\end{align*} 

(ii) For the same reason as in (i) just above, for $X<<1$,
\begin{align*}
\sqrt X \frac {|g(X)|} {X} \left| \int _0^{X/2} \frac {g(X-Y)} {X-Y}dY\right| \le  C||g||^2 _{ -r, q+r }X^{-\frac {1} {2}+r} \int _{X/2}^X  \frac {g(Y)} {Y}dY\\
\le  C||g||^2 _{ -r, q+r }X^{-\frac {1} {2}+2r}.
\end{align*} 

(iii) If $Y>X$,
\begin{align*}
&\sqrt X \left| \int _X^{\infty } \frac {g(Y-X)} {Y-X}\left( \frac {g(X)} {X} -\frac {g(Y)} {Y}\right)dY\right|
\le   C||g||^2 _{ -r, q+r }\sqrt X\left(I_1+I_2 \right)\\
&I_1=\sqrt X \int _X^{2X } \frac {|g(Y-X)|} {Y-X}\left| \frac {g(X)} {X} -\frac {g(Y)} {Y}\right|dY\\
&I_2=\sqrt X \int _{2X}^{\infty } \frac {|g(Y-X)|} {Y-X}\left| \frac {g(X)} {X} -\frac {g(Y)} {Y}\right|dY
\end{align*} 
For $Y\in (X, 2X)$, $|g(Y)|\le C||g|| _{ -r, q+r }Y^r\le C||g|| _{ -r, q+r }X^r$ and then,
\begin{align*}
I_1&=\sqrt X \int _X^{2X } \frac {|g(Y-X)|} {Y-X}\left| \frac {g(X)} {X} -\frac {g(Y)} {Y}\right|dY\\
&\le C||g|| _{ -r, q+r }X^{\frac {1} {2}+r-1}\int _X^{\infty } \frac {|g(Y-X)|} {Y-X} dY\\
&=C||g|| _{ -r, q+r }X^{\frac {1} {2}+r-1}\int _0^{\infty } \frac {|g(Y)|} {Y} dY\le C||g||^2 _{ -r, q+r }X^{\frac {1} {2}+r-1}.
\end{align*}
When $Y>2X$,
\begin{align*}
I_2\le \frac {|g(X)|} {\sqrt X}\int  _{ 2X }^\infty \frac {|g(Y-X)|} {Y-X}dY+\sqrt X\int  _{ 2X }^\infty \frac {|g(Y-X)|} {Y-X} \frac {|g(Y)|} {Y}dY=J_1+J_2,
\end{align*}
with
\begin{align*}
J_1= \frac {|g(X)|} {\sqrt X}\int  _{ X }^\infty \frac {|g(Y)|} {Y}dY\le \frac {|g(X)|} {\sqrt X} \int _X^\infty\frac {|g(Y)|} {Y}dY\le C||g||^2 _{ -r, q+r }X^{-\frac {1} {2}+r}.
\end{align*}

\begin{align*}
J_2\le \sqrt X\int  _{ 2X }^2 \frac {|g(Y-X)|} {Y-X} \frac {|g(Y)|} {Y}dY
+\sqrt X\int  _2^\infty  \frac {|g(Y-X)|} {Y-X} \frac {|g(Y)|} {Y}dY
\end{align*}
When $Y\in (2X, 2)$, the hypothesis on $g$ give $|g(Y-X)|\le C|Y-X)|^r$ and when $Y>2>2X$, then $Y-X>Y/2$ and $|g(Y-X)|/(Y-X)\le (Y-X)^{-1-q}<CY^{-1-q}$ Then,
\begin{align*}
J_2\le C||g||^2 _{ -r, q+r }\sqrt X\int  _{ 2X }^2 (Y-X)^{r-1} Y^{r-1}dY+C||g||^2 _{ -r, q+r }\sqrt X\int  _2^\infty Y^{-2-2q} dY\\
\le C||g||^2 _{ -r, q+r }\sqrt X\left(X^{\min (2r-1, 0)}+1 \right)\le C||g||^2 _{ -r, q+r }X^{-\frac {1} {2}+r},\,X<<1.
\end{align*}
It follows, $I_2\le J_1+J_2\le C||g||^2 _{ -r, q+r }X^{-\frac {1} {2}+r}$ and,
\begin{align*}
&\sqrt X \left| \int _X^{\infty } \frac {g(Y-X)} {Y-X}\left( \frac {g(X)} {X} -\frac {g(Y)} {Y}\right)dY\right|\le 
C||g||^2 _{ -r, q+r }X^{-\frac {1} {2}+r},\,\,X<<1.
\end{align*}

(iv) In the last term,
\begin{align*}
\sqrt X \left|\frac {g(X)} {X}\int _X^\infty \frac {g(Y)dY} {Y}\right|\le C ||g|| _{ -r, q+r }X^{-\frac {1} {2}+r}\int _0^\infty \frac {|g(Y)|dY} {Y}\le C||g||^2 _{ -r, q+r }X^{-\frac {1} {2}+r}.
\end{align*}
All this shows that $|Q_N(g)(X)|\le C||g||^2 _{ -r, q+r }X^{-\frac {1} {2}+r}$ for $X<<1$. \\

\noindent
Consider now the case where $X>>1$.

(i) If $Y \in (0, X/2)$  then $X-Y>X/2>>1$, and,
$$|g(Y-X)|\le ||g|| _{ -r, q+r }(Y-X)^{-q}\le C||g|| _{ -r, q+r }X^{-q}$$ 
and  then
\begin{align*}
\sqrt X \left| \int _0^{X/2} \frac {g(Y)} {Y}\left( \frac {g(X-Y)} {X-Y}-\frac {g(X)} {X} \right)dY\right|\le
C ||g|| _{ -r, q+r }X^{-\frac {1} {2}-q}   \int _0^{X/2} \frac {g(Y) dY} {Y} \\ \le C ||g|| _{ -r, q+r }X^{-\frac {1} {2}-q}   \int _0^{\infty} \frac {g(Y) dY} {Y}\le C ||g|| ^2_{ -r, q+r }X^{-\frac {1} {2}-q}.
\end{align*} 
\begin{align*}
(ii)\qquad &\sqrt X\left|\frac {g(X)} {X}\int_0^{X/2}\frac {g(X-Y)dY} {X-Y}\right|\le C ||g||_{ -r, q+r }X^{-\frac {1} {2}-q}\int  _{ X/2 }^X\frac {|g(Y)|dY} {Y}\\
&\le C ||g||^2_{ -r, q+r }X^{-\frac {1} {2}-2q}.\\
(iii) \qquad &\sqrt X \left| \int _X^{\infty } \frac {g(Y-X)} {Y-X}\left( \frac {g(X)} {X} -\frac {g(Y)} {Y}\right)dY\right|
\le   \sqrt X\left(I_3+I_4 \right),\\
&I_3= \frac {|g(X)|} {\sqrt X} \int _X^{\infty } \frac {|g(Y-X)|} {Y-X} dY=\frac {|g(X)|} {\sqrt X} \int _0^{\infty } \frac {|g(Y)|} {Y} dY\\
&\hskip 5.5cm \le C ||g||^2_{ -r, q+r }X^{-q-\frac {1} {2}} ,\\
&I_4=\sqrt X \int _{2X}^{\infty } \frac {|g(Y-X)|} {Y-X} \frac {|g(Y)|dY} {Y}=\sqrt X \int _{X}^{\infty } \frac {|g(Y)|} {Y} \frac {|g(Y+X)|dY} {Y+X}\\
&\quad \le  C ||g||_{ -r, q+r }X^{-\frac {1} {2}-q} \int _{X}^{\infty } \frac {|g(Y)|dY} {Y}
\le C ||g||^2_{ -r, q+r }X^{-q-\frac {1} {2}}
\end{align*} 

(iv) And in the fourth term,
\begin{align*}
\sqrt X \left|\frac {g(X)} {X}\int _X^\infty \frac {g(Y)dY} {Y}\right|\le  CX^{-\frac {1} {2}-q}\int _X^\infty \frac {|g(Y)|dY} {Y}\le C ||g||^2_{ -r, q+r }X^{-\frac {1} {2}-2q}.
\end{align*}
All this shows that $|N(g)(X)|\le CX^{-\frac {1} {2}-q}$ for $X>>1$. 

\begin{lem}
\label{Lemnonlin}
For all $g$ such that
\begin{align*}
&|g(X)|\le CX^{r},\,\,0<X<1\\
&|g(X)|\le CX^{-q},\,\,X>1,
\end{align*}
\begin{align*}
\hskip -1cm \text{it follows}\qquad &|Q_N(g )(X)|\le CX^{-\frac {1} {2}+r},\,\,0<X<1\\
&|Q_N(g ) (X)|\le CX^{-\frac {1} {2}-q},\,\,X>1.
\end{align*}
If $0<2r\le 1$, in the notation of the $X _{ \theta, \rho  }$ spaces: 
\begin{align*}
g(s) \in X _{-r, r+q}\Longrightarrow Q_N(g(s))\in X _{ \frac {1} {2}-r, r+q }:|| Q_N(g(s))|| _{ \frac {1} {2}-r, r+q }\le C ||g(s)||^2 _{-r, r+q}
\end{align*}
\end{lem}

\begin{prop}
\label{lrds01}
\begin{align*}
||Q_N(g)-Q_N(h)|| _{ \frac {1} {2}-r, r+q}\le C||g-h|| _{ -r, r+q }(||g|| _{ -r, r+q }+||h|| _{ -r, r+q }).
\end{align*}
\end{prop}
\begin{proof}
In all the four terms of $Q_N$, $N_j$, $j=1, \cdots 4$,
\begin{align*}
N_j(g, g)-N_j(h, h)=N_j(g-h, g)+N_j(h, g-h)
\end{align*}
and the Proposition \ref{lrds01}  follows from Lemma \ref{Lemnonlin}.
\end{proof}

\section{The operator $S_\varphi $.} 
\label{SFi0}
\setcounter{equation}{0}
\setcounter{theo}{0}
Given $h_0\in X _{ \theta, \rho  }$ we must solve now the problem,
\begin{align}
&\frac {\partial h} {\partial t}=\mathscr L_\varphi (h)\equiv \mathscr L(h)+T(h) \label{S5Eqfi}\\
&h(0, X)=h_0\in X _{ \theta, \rho  }. \label{S5Eqfi2}
\end{align}
\begin{prop}
\label{SFi}
For all $h_0\in X _{ \theta, \rho  }$ with $\theta\ge 0$, $\rho\ge 0$ such that $\theta+\rho <3$ there exists a unique  $h\in C((0, \infty); X _{ \theta, \rho  })\cap L^\infty _{ \text{loc} }( [0, \infty) ; X _{ \theta, \rho  })$ mild solution of  (\ref{S5Eqfi}), (\ref{S5Eqfi2}), (i.e satisfying
\begin{align*}
h(t)=S(t)h_0+\int _0^tS(t-s)(T(h(s)))ds,\,\,t>0
\end{align*}
in $X _{ \theta, \rho  }$), and there exists a constant $\kappa >0$, independent of $h_0$, such that
\begin{align}
||h(t)|| _{ \theta, \rho  }\le ||h_0|| _{ \theta, \rho  } e^{\kappa t},\,\,\forall t\ge 0. \label{minvstr1}
\end{align}
We denote $h(t)=S_\varphi (t)h_0$. The function  $h$  satisfies (\ref{S5Eqfi}) pointwisely for $X>0$, $t>0$. 
\end{prop}
\begin{proof}
Suppose that  $h(s)\in X _{ \theta, \rho  }$ for all $s>0$, $\sup _{ s>0 }||h(s)|| _{ \theta, \rho  }<\infty$ and look for a fixed point of
\begin{align*}
\Psi (h)(t, X)=S(t)h_0(X)+\int _0^tS(t-s)(T(h(s)))(X)ds.
\end{align*}
By Lemma 2.2 of \cite{m3},

\begin{align*}
||\Psi (h)|| _{ \theta, \rho  }\le ||h_0|| _{ \theta, \rho  }+\int _0^t||T(h(s))|| _{ \theta, \rho  }ds.
\end{align*}
By Corollary  \ref{TT}, $T(h(s))\in X _{ \theta-\frac {1} {2}, \rho +1 }\subset X _{ \theta, \rho  }$  and $||T(h(s)|| _{ \theta, \rho  }\le C||h(s)|| _{ \theta, \rho  }$, from where
\begin{align*}
||\Psi (h)|| _{ \theta, \rho  }\le ||h_0|| _{ \theta, \rho  }+C\int _0^t||h(s)|| _{ \theta, \rho  }ds\le ||h_0|| _{ \theta, \rho  }+Ct \sup _{ 0<s<t}||h(s)|| _{ \theta, \rho  }
\end{align*}
and then by a classical fixed point argument one first obtains a local  mild solution in  in $C( (0, T]; X _{ \theta, \rho  })\cap
L^\infty( (0, T]; X _{ \theta, \rho  })$, and then a global solution in $C((0, \infty); X _{ \theta, \rho  })\cap L^\infty _{ \text{loc} }( [0, \infty) ; X _{ \theta, \rho  })$. By Corollary 2.7 of \cite{m3}, $h_t$ and $\mathscr L_\varphi (h)$ belong to $L^\infty _{ \text{loc} }$ and $h$ satisfies the equation (\ref{S5Eqfi}) pointwise for $X>0$ and $T>0$. Uniqueness follows from Gronwall's Lemma.
\end{proof}

Similar arguments show the following,
\begin{prop}
\label{SFi2}
For all $F\in C([0, T); X _{ \theta, \rho  })$  for $\theta\ge 0$ and $\rho \ge 0$ and $\theta+\rho <3$ there exists a unique function $h\in C\left((0, T); X _{0, \min (3/2, \rho ) }\right)$, satisfying
\begin{align*}
&h(t)=\int _0^tS_\varphi  (t-s)F(s)ds,\,\,t\in (0, T).
\end{align*}
The function $h$ is such that,
\begin{align}
&\left|\left|h(t)-\int _0^tS(t-s)( F(s))ds\right|\right| _{\theta- \frac {1} {2}, \rho +1 }\le Ct^2\sup _{ 0\le s\le t } ||F(s)|| _{ \theta, \rho  }\,\,\forall t\in (0, T) \label{E6.3}\\
&||h(t)|| _{ 0, \min (3/2, \rho ) }\le C \sup _{ 0\le s\le t } ||F(s)|| _{ \theta, \rho  }
C\left(t^{1-2\theta}+t^{4-2\theta}\right).  \label{E6.4}
\end{align}
Moreover $h$ and $h_t$ belong to $L^\infty _{ \text{loc} }((0, T)\times (0, \infty))$ and,
\begin{align}
\frac {\partial h(t, X)} {\partial t}=\mathscr L_\varphi (h(t))(X)+F(t, X),\,t>0, X>0  \label{E6.5}
\end{align}
\end{prop}
\begin{proof}
By definition, the function $h$ is  given by
\begin{align*}
h(t)&=\int _0^t S_\varphi (t-s)F(s)ds\\
&=\int _0^t S (t-s)F(s)ds
+\int _0^t\int _0^{t-s} S (t-s-\sigma )T[S_\varphi (\sigma )F(s)]d\sigma ds\\
&=\int _0^t S (t-s)F(s)ds
+\int _0^t\int _s^{t} S (s'-s)T[S_\varphi (t-s')F(s)]ds'ds.
\end{align*}
By $(ii)$ in Proposition \ref{S7cor1} and  Corollary \ref{TT}, for $\theta\in [0, 1/2)$, $\rho \ge 0$, $\rho +\theta<3$,
\begin{align*}
||S(s'-s) T[S_\varphi (t-s')F(s)]|| _{ \theta-\frac {1} {2}, \rho +1 }\le ||S(s'-s) T[S_\varphi (t-s')F(s)]|| _{ 0, \theta+\rho +\frac {1} {2} }\\
\le C ||T[S_\varphi (t-s')F(s)]||_{ 0, \theta+\rho +\frac {1} {2} }\le C ||[S_\varphi (t-s')F(s)]|| _{ \theta , \rho }
\le Ce^{\kappa (t-s')}||F(s)]|| _{ \theta , \rho }.
\end{align*}
Then,
\begin{align*}
\left|\left|\int _0^t\int _s^{t} S (s'-s)T[S_\varphi (t-s')F(s)]ds'ds\right|\right| _{ \theta-\frac {1} {2}, \rho +1 }\le
\sup _{0\le s\le t}  ||F(s)|| _{ \theta, \rho  }\kappa ^{-2} \left(e^{\kappa t}-1-\kappa t \right)
\end{align*}
and (\ref{E6.3}) follows.
On the other hand, by Theorem  1.1 of \cite{m3}, if $t\in (0, 1)$,
\begin{align*}
\left|\left|\int _0^t S (t-s)F(s)ds \right|\right| _{ 0, \frac {3} {2} }\le
\sup _{ 0<X<t^2 }\left|\int _0^t S (t-s)F(s)ds \right|+
\sup _{ t^2<X<1 }\left|\int _0^t S (t-s)F(s)ds \right|+\\
+\sup _{ t^2<X<1 }X^{3/2}\left|\int _0^t S (t-s)F(s)ds \right|
\le Ct^{1-2\theta}\sup _{ 0\le s\le T }||F(s)|| _{ \theta, \rho  }+Ct^{4-2\theta}\sup _{ 0\le s\le T }||F(s)|| _{ \theta, \rho  }.
\end{align*}
and, if $t>1$,
\begin{align*}
\left|\left|\int _0^t S (t-s)F(s)ds \right|\right| _{ 0, \frac {3} {2} }\le
\sup _{ 0<X<1 }\left|\int _0^t S (t-s)F(s)ds \right|+
\sup _{ 1<X<t^2 }\left|\int _0^t S (t-s)F(s)ds \right|+\\
+\sup _{ t^2<X }X^{3/2}\left|\int _0^t S (t-s)F(s)ds \right|
\le Ct^{2-2\theta}\sup _{ 0\le s\le T }||F(s)|| _{ \theta, \rho  }+Ct^{4-2\theta}\sup _{ 0\le s\le T }||F(s)|| _{ \theta, \rho  }.
\end{align*}
which proves (\ref{E6.4}). By Proposition \ref{SFi} and classical arguments,
\begin{align*}
\frac {\partial h(t)} {\partial t}&=F(t)+\int _0^t  \mathscr L _{ \varphi  } \Big( S_\varphi (t-s) F(s)\Big)ds=F(t)+
\mathscr L _{ \varphi  } \Bigg( \int _0^t   S_\varphi (t-s) F(s)ds\Bigg)\\
&=F(t)+\mathscr L _{ \varphi  }(h(t))
\end{align*}
and property (\ref{E6.4}) follows. By Proposition \ref{S7cor2}  both $h_t$ and $\mathscr L_\varphi (h)$ belong to  $L^\infty _{ \text{loc} }((0, \infty)\times (0, \infty))$ and $h$ satisfies (\ref{E6.5}) for $t>0$, $X>0$.
\end{proof}

Denote  from now on  $\eta\in C^1_c([0, \infty])$ a function such that $\text{supp} \,(\eta)\subset [0, 1]$, $\eta(t, 0)=1$ and $\eta' \le 0$.
\begin{prop}
\label{xixixi}
For all $\theta\ge 0$, $\rho \ge 0$, $\theta+\rho <3$ and $T>0$ there exists a constant $C>0$ such that, the function $r(t, g, X)$ defined as
\begin{align}
\label{Exixixi}
&r_0(t, g, X)=\int _0^t S_\varphi (t-s)g(s)(X)ds-\mathscr W(t, g) \eta (X),\\
&\mathscr W(t, g)=\int _0^{t }L(t-s; g(s))ds+\int _0^t \int _0^{t -s}\Big(L(t-s-\sigma;  T[S_\varphi  (\sigma )g(s)])\Big)d\sigma ds.\label{Exixixi2}
\end{align}
satisfies, for $t\in (0, \min (1,T))$
\begin{align}
|r_0(t, g, X)|\le 
\begin{cases}
\label{casesr}
C\sup _{ 0\le s\le t }||g(s)|| _{ \theta, \rho  }\Bigg(\left( X^{\frac {1} {2}-\theta}+X^{\frac {1} {2}}t^{1-2\theta}\right)\1 _{ 0<X<t^2 } +\\
\hskip 5cm +t^{1-2\theta}\1 _{ t^2<X<1 }\Bigg),\,\forall X\in (0, 1)\\
\displaystyle{C\sup _{ 0\le s\le t }||g(s)|| _{ \theta, \rho  }t^{4-2\theta} X^{-\frac {3} {2} },\,\forall X>1.}
\end{cases}
\end{align}
\end{prop}
\begin{proof}
By definition,
\begin{align}
\label{mntrvll43}
\int _0^tS_\varphi (t-s)g(s)ds&=\int _0^t S(t-s)g(s)ds+\nonumber \\
&+\int _0^t\int _0^{t-s}S(t-s-\sigma   )T[S_\varphi (\sigma )g(s)]d\sigma ds.
\end{align}
Let us write several estimates, that  easily follow from our previous results and from where the Proposition follows. 
By $(iii)$ of Proposition \ref{S7cor2},
\begin{align}
\int _0^t\Big(L(t-s, g(s))\Big)ds\le C\sup _{ 0\le s\le t } ||g(s)|| _{ \theta, \rho  }t^{1-2\theta},\,\,\forall t\in (0, T).
\end{align}
On the other hand, suppose $X\in (0, t^2)$. For $X<(t-s)^2$, by $(iv)$ in Proposition \ref{S7cor2},

\begin{align*}
&\int _0^{t-s}S(t-s-\sigma)T[S_\varphi (\sigma )g(s)](X)d\sigma=\int _0^{t-s}
L\Big(t-s-\sigma; T[S_\varphi (\sigma )g(s)]\Big)d\sigma +\\
&\hskip 6.5cm + R_3\Big(t-s, T[S_\varphi (\cdot )g(s)], X \Big) \\
&|R_3(t-s, T[S_\varphi (\cdot )g(s)], X)|\le C\sup _{ 0\le \sigma \le t }||T[S_\varphi (\sigma  )g(s)]|| _{ \theta, \rho  }\times \\
&\hskip 6cm \times X^{1/2}\left(X^{-\theta}+(t-s  )^{1-2\theta}\right).
\end{align*}
and for $X>(t-s)^2$,
\begin{align*}
\int _0^{t-s}S(t-s-\sigma)T[S_\varphi (\sigma )g(s)](X)d\sigma&\le CX^{-3/2}(t-s)^{4-2\theta}\times \\
&\times \sup _{ 0\le \sigma \le t }||T[S_\varphi (\sigma )g(s)] || _{ \theta, \rho  }.
\end{align*}
By  Remark \ref{lzm453} and (\ref{minvstr1}),
\begin{align}
||T[S_\varphi (\sigma )g(s)]|| _{ \theta, \rho  }\le C||S_\varphi (\sigma )g(s)|| _{ \theta, \rho  }\le Ce^{\kappa \sigma }||g(s)|| _{ \theta, \rho  }.
\label{minvstdr2}
\end{align}
and then, for all $X\in (0, t^2)$,
\begin{align}
\int _0^t\int _0^{t-s}&S(t-s-\sigma)T[S_\varphi (\sigma )g(s)](X)d\sigma ds=\nonumber \\
&=\int _0^{t-X^{1/2}} \int _0^{t-s}[\cdots]d\sigma ds
+\int _{t-X^{1/2}}^t\int _0^{t-s}[\cdots]d\sigma ds=I_1+I_2.\label{gstnbrrl11}
\end{align}
where
\begin{align}
|I_2|\le  CX^{-3/2}\sup _{ 0\le s \le t }||T[S_\varphi (\cdot )g(s)] || _{ \theta, \rho  }dsX^{\frac {5-2\theta} {2}}=
Ce^{\kappa t}X^{1-\theta }\sup _{ 0\le s \le t }|| g(s) || _{ \theta, \rho  }. \label{gstnbrrl23}
\end{align}
and

\begin{align}
I_1&=\int _0^{t-X^{1/2}}\int _0^{t-s}
L\Big(t-s-\sigma; T[S_\varphi (\sigma )g(s)]\Big)d\sigma ds+\nonumber\\
&\hskip 2cm +\int _0^{t-X^{1/2}}R_3\Big(t-s, T[S_\varphi (\cdot )g(s)], X \Big)ds\nonumber\\
&=\int _0^{t}\int _0^{t-s}
L\Big(t-s-\sigma; T[S_\varphi (\sigma )g(s)]\Big)d\sigma ds+\label{gstnbrrl1}\\
&+\int _{t-X^{1/2}}^t\int _0^{t-s}
L\Big(t-s-\sigma; T[S_\varphi (\sigma )g(s)]\Big)d\sigma ds+\label{gstnbrrl2}\\
&+\int _0^{t-X^{1/2}}R_3\Big(t-s, T[S_\varphi (\cdot )g(s)], X \Big)ds.\label{gstnbrrl3}
\end{align}
In the term (\ref{gstnbrrl2}), $X>(t-s)^2$;  then  by $(vi)$ of  Proposition \ref{S7cor1},
\begin{align}
&\int _{t-X^{1/2}}^t\int _0^{t-s}
\left|L\Big(t-s-\sigma; T[S_\varphi (\sigma )g(s)]\Big)\right|ds\le \label{gstnbrrl5}\\
&\le C\!\!\!\!\sup _{ 0\le \sigma, s\le t }||T[S_\varphi (\sigma )g(s)]|| _{ \theta, \rho  } \int _{t-X^{1/2}}^t\!\!\!(t-s)^{-2\theta}ds
\le C e^{\kappa t }\!\!\!\!\sup _{ 0\le\sigma,  s\le t }||g(s)|| _{ \theta, \rho  }X^{\frac {1} {2}-\theta},\nonumber
\end{align}
and in the term  (\ref{gstnbrrl3}), by $(iv)$ in Proposition \ref{S7cor2},
\begin{align}
\int _0^{t-X^{1/2}}&\left|R_3\Big(t-s, T[S_\varphi (\cdot )g(s)], X \Big)\right|ds\le \nonumber\\
&\le C \sup _{ 0\le \sigma, s\le t }||T[S_\varphi (\sigma )g(s)]|| _{ \theta, \rho  }X^{1/2} \int _0^{t-X^{1/2}}
 \left( X^{- \theta }+t^{1-2\theta}\right)\nonumber\\
 &\le C e^{\kappa t }\!\!\!\!\sup _{ 0\le s\le t }||g(s)|| _{ \theta, \rho  }X^{1/2}\left(X^{-\theta}t+t^{2-2\theta} \right). \label{gstnbrrl8}
\end{align}
It follows from, (\ref{gstnbrrl11}), (\ref{gstnbrrl2})--(\ref{gstnbrrl8}) that there exists a constant $C>0$ such that
\begin{align}
&\forall X\in (0, t^2):\,\,\left | I_1-\int _0^{t}\int _0^{t-s}
L\Big(t-s-\sigma; T[S_\varphi (\sigma )g(s)]\Big)d\sigma ds\right|\le \nonumber \\
&\quad C e^{\kappa t }\!\!\!\!\sup _{ 0\le\sigma,  s\le t }||g(s)|| _{ \theta, \rho  }X^{\frac {1} {2}-\theta}+
C e^{\kappa t }\!\!\!\!\sup _{ 0\le\sigma,  s\le t }||g(s)|| _{ \theta, \rho  }X^{1/2}\left(X^{-\theta}t+t^{2-2\theta} \right). \label{gstnbrrl16}
\end{align}
and using (\ref{gstnbrrl23}) too,
\begin{align}
&\Bigg | \int _0^t\int _0^{t-s}S(t-s-\sigma)T[S_\varphi (\sigma )g(s)](X)d\sigma ds-\nonumber \\
-&\int _0^{t}\int _0^{t-s}
L\Big(t-s-\sigma; T[S_\varphi (\sigma )g(s)]\Big)d\sigma ds \Bigg|\le
\nonumber \\
& \le C e^{\kappa t }\!\!\!\!\sup _{ 0\le\sigma,  s\le t }||g(s)|| _{ \theta, \rho  }X^{\frac {1} {2}}(X^{-\theta}+X^{-\theta}t+t^{2-2\theta}),\,\,\forall X\in (0, t^2). \label{gstnbrrl34}
\end{align}

When $X>t^2$, then $X>(t-s)^2$ for all $s\in (0, t)$ and by $(v)$ of Proposition \ref{S7cor2}
\begin{align}
\int _0^t\int _0^{t-s}&\left|S(t-s-\sigma   )T[S_\varphi (\sigma )g(s)]\right|d\sigma ds \nonumber\\
&\le 
C\sup _{ 0\le \sigma, s\le t }||T[S_\varphi (\sigma )g(s)]|| _{ \theta, \rho  }
X^{-3/2}\int _0^t(t-s)^{4-2\theta} ds \nonumber \\ 
&\le C e^{\kappa t }\!\!\!\!\sup _{ 0\le s\le t }||g(s)|| _{ \theta, \rho  }X^{-3/2}t^{5-2\theta},\,\,\forall X>t^2. \label{gstnbrrl17}
\end{align}

By estimate $(vi)$ of Proposition \ref{S7cor1} and Remark  \ref{lzm453})  it easily follows,
\begin{align}
\left|\int _0^tL(t-s, g(s)ds\right|+\left|\int _0^t\int _0^{t-s }L\big(t-s-\sigma; T(g(s)\big)d\sigma ds\right|\le \nonumber\\
\le C\sup _{ 0\le s\le t }||g(s)|| _{ \theta, \rho  }t^{1-2\theta}+C\sup _{ 0\le s\le t }||T(g(s))|| _{ \theta, \rho  }t^{2-2\theta}\nonumber\\
\le C\sup _{ 0\le s\le t }||g(s)|| _{ \theta, \rho  }t^{1-2\theta}.
\label{gstnbrrl27}
\end{align}
Let us prove now (\ref{casesr}). Suppose first that $X>1$ and then $\eta(X)=0$. By  $(v)$ of Proposition \ref{S7cor2}, and  (\ref{gstnbrrl17}),
\begin{align}
|r_0(t, g, X)|\le C\sup _{ 0\le s\le t } ||g(s)|| _{ \theta, \rho  }\left(X^{-3/2}t^{4-2\theta}+e^{\kappa t }X^{-3/2}t^{5-2\theta} \right)\label{gstnbrrl29}
\end{align}
If $X\in (t^2, 1)$ then, by   (\ref{gstnbrrl17}) and (\ref{gstnbrrl27})
\begin{align}
|r_0(t, g, X)|\le C\sup _{ 0\le s\le t } ||g(s)|| _{ \theta, \rho  }t^{1-2\theta}\le C\sup _{ 0\le s\le t } ||g(s)|| _{ \theta, \rho  }X^{\frac {1} {2}-\theta}\label{gstnbrrl31}
\end{align}
When $X\in (0, t^2)$, by  (\ref{gstnbrrl34})
\begin{align}
\label{gstnbrrl37}
|r_0(t, g, X)|\le C\sup _{ 0\le s\le t } ||g(s)|| _{ \theta, \rho  }X^{\frac {1} {2}}(X^{-\theta}+t^{1-2\theta}+X^{-\theta}t+t^{2-2\theta})
\end{align}
and  (\ref{casesr}) follows from (\ref{gstnbrrl29}), (\ref{gstnbrrl31}) and (\ref{gstnbrrl37}).
\end{proof}

\begin{cor}
\label{cor9.3}
For all $h\in L^\infty((0, T); X _{ \theta, \rho  })$
\begin{align}
\left| \frac {1} {X}\int _0^t(S_\varphi (t-s)h(s))(X)ds-\frac {1} {X'}\int _0^t(S_\varphi (t-s)h(s))(X')ds \right| \le \nonumber
\\ \le C\sup _{ 0\le s\le t }\left(||h(s)|| _{ \theta', \rho ' }+||h(s)|| _{ \theta, \rho }\right)\Omega_ {2\theta}(X^{1/2}, X'^{1/2}) \label{cor9.3.2}
\end{align}
for $\theta', \rho'$ such that
\begin{align*}
(\theta')^+\le \frac {\theta+1} {2};\,\,\,\,\frac {\theta+\rho } {2}\le \theta'+\rho '+\frac {1} {2}
\end{align*}
where the function $\Omega_ {2\theta}$ is introduced  in Lemma \ref{blabla}.
\end{cor}
\begin{proof}
Write, by definition of $S_\varphi $,
\begin{align*}
\int _0^t(S_\varphi (t-s)h(s))(X)ds=\int _0^t (S(t-s)h(s))(X)ds+\\
+\underbrace{\int _0^t\int _0^{t-s}\Big[S(t-s-\sigma )T\big(S_\varphi(\sigma )h(s)  \big)\Big](X)d\sigma ds} _{ H(t, X) }
\end{align*}
and
\begin{align*}
\left|\frac {H(t, X)} {X}-\frac {H(t, X')} {X'}\right|\le \sup _{ 0\le s\le t }
\left| \left| \int _0^{t-s}S(t-s-\sigma )T\big(S_\varphi(\sigma )h(s)  \big)d\sigma\right|\right| _{\frac { \theta} {2}, \frac {\rho} {2}  }\times \\
\times  \Omega_\theta (X^{1/2}, X'^{1/2})
\end{align*}
and, using Corollary \ref{TT} and (\ref{minvstr1}) in  Proposition \ref{SFi},
\begin{align*}
\left| \left| \int _0^{t-s}S(t-s-\sigma )T\big(S_\varphi(\sigma )h(s)  \big)d\sigma\right|\right| _{\frac {\theta} {2}, \frac {\rho} {2}   }&\le 
\int _0^{t-s}\left| \left|T\big(S_\varphi(\sigma )h(s)  \big)\right|\right| _{\frac {\theta} {2}, \frac {\rho} {2}   }d\sigma \\
\le C\int _0^{t-s}\left| \left| S_\varphi(\sigma )h(s) \right|\right| _{\theta', \rho ' }d\sigma&\le 
C\int _0^{t-s}e^{\kappa (t-s)} || h(s)|| _{\theta', \rho ' }d\sigma\\
&\le C\sup _{ 0\le s\le t } || h(s)|| _{\theta', \rho ' }
\end{align*}
for $\theta', \rho'$ such that
\begin{align*}
(\theta')^+\le \frac {\theta+1} {2};\,\,\,\,\frac {\theta+\rho } {2}\le \theta'+\rho '+\frac {1} {2}.
\end{align*}
It follows,
\begin{align*}
\left| \frac {1} {X}\int _0^t(S_\varphi (t-s)h(s))(X)ds-\frac {1} {X'}\int _0^t(S_\varphi (t-s)h(s))(X')ds \right| \le \\
\\ \le C\sup _{ 0\le s\le t }\left(||h(s)|| _{ \theta', \rho ' }+||h(s)|| _{ \theta, \rho }\right)\Omega_ \theta(X^{1/2}, X'^{1/2})
\end{align*}
\end{proof}
\section{Preliminaries to the Proof of Theorem \ref{Mth1}.}
\label{Sg1g2}
\setcounter{equation}{0}
\setcounter{theo}{0}
This Section contains some  preliminary results, necessary for  the proof of Theorem \ref{Mth1}. First of all, equation (\ref{S1E7}) may still be slightly simplified using the  change of  variables
\begin{align}
&g(t)=\overline g(\overline \tau )\nonumber \\
\label{S1E09}
&\overline \tau =\int _0^t\lambda  (s)ds,\,\,\beta (\overline \tau )=\lambda  (t),\,\,\frac {\lambda  '(t)} {\lambda  (t)}=\beta '(\overline \tau )
\end{align}
to give,
\begin{align}
\label{S1E8BB}
\frac {\partial \overline g (\overline\tau )} {\partial \overline\tau } = \mathscr L_\varphi (\overline g(\overline \tau ) )+ \beta  (\overline\tau ) Q_N(\varphi )+\beta  ^{-1}(\overline\tau ) Q_N(\overline g(\overline\tau ) )-\beta '(\overline\tau )\varphi 
\end{align}
that we want to solve with the initial data $\overline g (0, X)=0$.  By some abuse of notation, and trying to avoid the excessive  proliferation of variables,  the function $\overline g(\overline \tau )$ is  still denoted as $g(t)$ in this Section and until  the last part of Section \ref{SMth1}.
\subsection{The function $\tilde g_1$.}
For $\beta \in C([0, T])$ such that $\beta (t)>0$. on $[0, T]$,  consider the function $\tilde g_1$ defined as,
\begin{align}
&\tilde g_1(t, X)=\int _0^t (S_\varphi (t-s) G _{ \beta , g }(s))(X)ds \label{tildeg1}\\
&G _{ \beta , g }(s)=\Big(\beta  (s) Q_N(\varphi )(s)+\beta  ^{-1}(s) Q_N(g(s)) \Big) \label{G1}
\end{align}  
In order to shorten notations, the function $G _{ \beta , g }(s)$ will be written $G(s)$ when no confusion is possible.
\begin{prop}
\label{Prop7.1}
For all $r\in (0, 1/2)$, $q>0$ and $T>0$, there exists a constant $C>0$ such that the function defined in (\ref{G1}) satisfies,
\begin{align}
&G _{ \beta , g }\in  L^\infty ((0, T); X_{ \theta, \rho  }),\,\,\forall \theta \ge \frac {1} {2}-r,\,\,\forall q\ge\rho +\theta-\frac {1} {2},\label{FFGG2}\\
\label{FFGG1}
&||G _{ \beta , g }(s)|| _{ \theta, \rho  }\le C\left(\beta (s)+\beta ^{-1}(s)||g(s)||^2 _{ -r, r+q }\right).
\end{align}
Moreover, $\tilde g_1\in C\left((0, T); X _{ 0, \min(\frac {3} {2}, \rho ) }\right)$, $\partial _t\tilde g_1$ and  $\mathscr L_\varphi (\tilde g_1)$ belong to $L^\infty _{ \text{loc} }((0, \infty)\times (0, \infty))$, $\partial _t\tilde g_1 (t, X)=\mathscr L_\varphi (\tilde g_1)(t, X)+G(t, X)$ for $t>0$, $X>0$
\end{prop}
\begin{proof}
By Proposition \ref{QQ0}, 
\begin{align*}
X^{\theta}|Q_N(\varphi )(X)|\le CX^{\theta+\frac {1} {2}}\1 _{ 0<X<M }
\end{align*}
and then $Q_N(\varphi )\in X _{ \theta, \rho  }$  for all $\theta\ge -1/2$ and $\rho >0$
On the other hand, 
by Lemma  \ref{Lemnonlin}, if $g\in X _{ -r, r+q }$ then, 
\begin{align}
\label{S6EQNg}
|Q_N(g(s))|\le C||g(s)||^2 _{ -r, r+q }\left(X^{-\frac {1} {2}+r}\1 _{ 0<X<1 }+X^{-\frac {1} {2}-q}\1 _{ X>1}\right).
\end{align}
If $r\ge \frac {1} {2}-\theta$ and $q+\frac {1} {2}\ge\rho +\theta$, it then follows,
\begin{align*}
\sup _{ X>0 }\left(X^\theta X^{-\frac {1} {2}+r}\1 _{ 0<X<1 }+X^{\theta+\rho }X^{-\frac {1} {2}-q}\1 _{ X>1}\right)<\infty.
\end{align*}
Then,
\begin{align*}
||G(s)|| _{ \theta, \rho  }&\le C\beta (s)||Q_N(\varphi )|| _{ \theta, \rho  }+C\beta ^{-1}(s)||Q_N(g(s)|| _{ \theta, \rho  }\nonumber \\
&\le C\left(\beta (s)+\beta ^{-1}(s)||g(s)||^2 _{ -r, r+q }\right)
\end{align*}
and the  Proposition \ref{Prop7.1} follows.
\end{proof}

By construction, 
\begin{align*}
S_\varphi (t)(h_0)=S(t)h_0+\int _0^tS(t-s)\Big(T[S(s)h_0] \Big)ds
\end{align*}
and then,
\begin{align}
\label{lngrsg1}
\tilde g_1(t, X)&=\int _0^t (S(t-s)G(s))(X)ds+\nonumber\\
&+\int _0^t \int _0^{t-s}S(t-s-\sigma )(T[S _ \varphi (\sigma )G(s)(X)])d\sigma ds
\end{align}
Let us denote now $\eta=\eta(X)$ defined for   $X>0$ such that $\text{supp} \,(\eta)\subset [0, 1]$, $\eta(t, 0)=1$ and $\eta' \le 0$.
\begin{prop}
\label{tildeg1}
Suppose that $g\in X _{ -r, r+q }$ for some $r\in (0, 1/2)$ and $q>0$ and let  $r_0(t, G, X)$ be the function defined as in Proposition \ref{xixixi}  by (\ref{Exixixi}).
\begin{align}
r_0(t, G, X)=\tilde g_1(t, X)-\mathscr W(t, G). \label{tildeg1E1}
\end{align} 
Then,  for $\theta \ge \frac {1} {2}-r$ and $\rho <q+\frac {1} {2}-\theta$, there exists a constant $C>0$ such that for all $t\in(0, 1)$
\begin{align*}
&(i) \text{For}\,\,X\in (0, 1),\\
&|r_0(t, G, X)|\le C\sup _{ 0<s<t } \left(\beta (s)||Q_N(\varphi )|| _{ \theta, \rho  }+\beta ^{-1}(s)||g(s)||^2_{ -r, r+q }\right)X^{\frac {1} {2}}(X^{-\theta}+t^{1-2\theta})
\end{align*}
\begin{align*}
&\hskip -1.2cm (ii) \text{For}\,\,X>1,\\
&\hskip -1.2cm |r_0(t, G, X)|\le  C\sup _{ 0<s<t } \left(\beta (s)||Q_N(\varphi )|| _{ \theta, \rho  }+\beta ^{-1}(s)||g(s)||^2_{ -r, r+q }\right)t^{4-2\theta}X^{-\frac {3} {2} }.
\end{align*}
\end{prop}
\begin{proof}
By Proposition \ref{xixixi}, for all  $t\in (0, 1)$,
\begin{align*}
&\int _0^t S_\varphi (t-s)G(s)(X)ds-\Bigg( \int _0^tL(t-s, G(s)ds +\nonumber \\
&\qquad +\int _0^t\int _0^{t-s }L\big(t-s-\sigma; T[S_\varphi (\sigma )G(s)]\big)d\sigma ds\Bigg) \eta (X)=r_0(t, G, X).
\end{align*}
For $X\in (0, 1)$,
\begin{align*}
|r_0(t, G, X)|\le C\sup _{ 0\le s\le t }||G(s)|| _{ \theta, \rho  }X^{\frac {1} {2}-\theta}
\end{align*}
and $||G(s)|| _{ \theta, \rho  }$ may be bounded by Proposition \ref{Prop7.1}, from where,
\begin{align*}
&\Bigg| \int _0^t(S_\varphi (t-s)G(s))(X)ds- \mathscr W(t, G)\eta(X)\Bigg|\le  \\
&\le C\sup _{ 0<s<t } \left(\beta (s)||Q_N(\varphi )|| _{ \theta, \rho  }+\beta ^{-1}(s)||g(s)||^2 _{ -r, r+q }\right)X^{\frac {1 } {2}}\left(X^{-\theta}+t^{1-2\theta} \right).
\end{align*}

For $X> 1>t^2$, by Proposition  \ref{xixixi},
\begin{align*}
|r(t, G, X)|\le C\sup _{ 0\le s\le t }||G(s)|| _{ \theta, \rho  } t^{4-2\theta}X^{-\frac {3} {2} },\,\forall X>1
\end{align*}
and the proof of point (ii) follows as for point (i).
\end{proof}
The following estimate of $\mathscr W(t, G)$ is  also useful. 
\begin{prop}
\label{WWW}
For $\theta\in [0, 1/2)$,  $\theta'>-1/2$ and $\rho \ge 0,  \rho'\ge 1 $,
\begin{align*}
|\mathscr W(t, G)|\le C\sup _{ 0\le s\le t }||G(s)|| _{ \theta, \rho  }t^{1-2\theta}+
 C\sup _{ 0\le \sigma \le t } || G(s)|| _{\frac {1} {2}+ \theta', \rho'-1  }t^{2-2\theta'}
\end{align*}
\end{prop}
\begin{proof}
By definition, for all $X\in (0, t^2)$,
\begin{align*}
|\mathscr W(t, G)|\le \int _0^t|L(t-s;G(s))|ds+\int _0^t\int _0^{t-s}\left| L(t-s-\sigma , T[S_\varphi (\sigma )G(s)]\right|d\sigma ds
\end{align*}
By $(iii)$ of  Proposition \ref{S7cor2}  with $p=\theta$,
\begin{align*}
 \int _0^t|L(t-s;G(s))|ds\le  C\sup _{ 0\le s\le t }||G(s)|| _{ \theta, \rho  }t^{1-2\theta}.
\end{align*}

Again by $(iii)$ of Proposition \ref{S7cor2},  now with $p=\theta'$, and by Corollary \ref{TT},
\begin{align*}
\int _0^{t-s}\left| L(t-s-\sigma , T[S_\varphi (\sigma )G(s)]\right|d\sigma \le 
C \sup _{ 0\le \sigma,  s\le t }||T[S_\varphi (\sigma )G(s)]|| _{ \theta', \rho'  } (t-s )^{1-2\theta'}\\
 \le 
C \int _0^{t-s}||S(\sigma )G(s)|| _{\frac {1} {2}+ \theta', \rho'-1  }(t-s-\sigma )^{-2\theta'}d\sigma.
\end{align*}
And,  since $\frac {1} {2}+ \theta'>0$ and  $\rho' \ge 1$ , by the continuity of $S(t)$ on $X _{\frac {1} {2}+ \theta' , \rho '-1}$ and (\ref{minvstdr2}),
\begin{align*}
&\int _0^{t-s}\left| L(t-s-\sigma , T[S_\varphi (\sigma )G(s)]\right|d\sigma \le C e^{\kappa  t}\sup _{ 0\le s \le t } || G(s)|| _{\frac {1} {2}+ \theta', \rho'-1  }(t-s )^{1-2\theta'}
\end{align*}
from where Proposition follows.
\end{proof}

\begin{prop}
\label{WWW2}
Suppose that $g\in C([0, T); X _{ -r, q+r }), g'\in C([0, T]; X _{ -r, q+r })$ for some $r\in [0, 1/2)$, $0\le q<3$ and $\beta \in C([0, T]), \beta' \in C([0, T])$ such that
$|\beta (t)-1|\le 1/4$ and  $|\beta' (t)-1|\le 1/4$ for $t\in [0, T]$.  

Then, for $\theta \in [\frac {1} {2}-r, \frac {1} {2})$, and $\theta+\rho \le q+\frac {1} {2}$, $\theta+\rho \le 3/2$ there exists a non negative continuous function $\tau (t)$ such that
\begin{align*}
&|\mathscr W(t, G)-\mathscr W(t, G')|\le C\tau (t)\Bigg(\sup _{ 0\le s\le t }|(\beta -\beta')(s)|\times \nonumber\\
&\times \left(||Q_N(\varphi )|| _{ \theta, \rho }+\sup _{ 0\le s\le t }||g(s)||^2 _{ -r, r+q }\right)
+\sup _{ 0\le s\le t }||g(s)-g'(s)|| _{ -r, r+q }\times \nonumber \\
&\hskip 4cm \times \sup _{ 0\le s\le t }\left( ||g(s))|| _{ -r, r+q }+||g'(s)|| _{ -r, r+q }\right)\Bigg),
\end{align*}
where $\tau(t)\to 0$ as $t\to 0$ and
\begin{align*}
\sup _{ X >0 }&X^\theta (1+X)^\rho  |r_0(t, G, X)-r_0(t, G', X)|\le
Ct\Bigg(\sup _{ 0\le s\le t }|(\beta -\beta')(s)|\times \nonumber\\
&\times \left(||Q_N(\varphi )|| _{ \theta, \rho }+\sup _{ 0\le s\le t }||g(s)||^2 _{ -r, r+q }\right)
+\sup _{ 0\le s\le t }||g(s)-g'(s)|| _{ -r, r+q }\times \nonumber \\
&\hskip 4cm \times \sup _{ 0\le s\le t }\left( ||g(s))|| _{ -r, r+q }+||g'(s)|| _{ -r, r+q }\right)\Bigg).
\end{align*}
The maps $(g, \beta )\to \mathscr W(t, G)$ and $(g, \beta )\to r_0(\cdot, g, \cdot)$ from  $X _{ -r, r+q }\times C([0, T])$ respectively to $C([0, T])$ and 
$X _{ \theta, \rho }$ are then Lipschitz continuous and there exists  $T_0>0$ small enough, such that  the Lipschitz constant is strictly less than one for $T<T_0$.
\end{prop}
\begin{proof}
Let $G $ and $G'$  of the form (\ref{G1}) with $\beta \in C([0, T))$, $\beta' \in C([0, T))$ and $g\in C([0, T); X _{ -r, r+q })$, $g'\in C([0, T); X _{ -r, r+q })$ for some $T>0$. By Proposition \ref{WWW}, for $X\in (0, t^2)$,
\begin{align}
\label{dfceg}
|\mathscr W(t, G -G')|\le C\sup _{ 0\le s\le t }||(G  -G')(s)|| _{ \theta, \rho  }t^{1-2\theta}+\nonumber \\+ C\sup _{ 0\le \sigma \le t } || (G -G')(s)|| _{\frac {1} {2}+ \theta', \rho'-1  }t^{2-2\theta'}.
\end{align}

In order to bound the right hand side of (\ref{dfceg}) we write,
\begin{align}
&|(G-G')(s)|\le |(\beta-\beta')(s)||Q_N(\varphi )|+\nonumber \\
& +|\beta^{-1}(s)-\beta'^{-1}(s)||Q_N(g(s))|+ \beta'^{-1}|Q_n(g(s))-Q_n(g'(s))|.\label{dfceg2}
\end{align} 
Then, by Proposition 7.16 of \cite{m3}, if $\theta \ge \frac {1} {2}-r$, $\theta+\rho \le q+\frac {1} {2}$,
\begin{align}
&||(\beta-\beta')(s)Q_N(\varphi )|| _{ \theta, \rho }\le  |(\beta-\beta')(s)| ||Q_N(\varphi )|| _{ \theta, \rho }\le C |(\beta-\beta')(s)|\label{dfceg3}\\
&||(\beta^{-1}(s)-\beta'^{-1}(s))Q_N(g(s))|| _{ \theta, \rho }\le C |(\beta-\beta')(s)| ||Q_N(g(s))|| _{ \theta, \rho}\nonumber\\
&\hskip 5.7cm \le C |(\beta-\beta')(s)||g(s)||^2 _{ -r, r+q }\nonumber\\
&\hskip -0.1cm ||Q_n(g(s))-Q_n(g'(s))|| _{\theta, \rho }\le C||(g-g')(s)|| _{ -r, r+q }(||g(s)|| _{ -r, r+q }+||g'(s)|| _{ -r, r+q }),\label{dfceg4}
\end{align}
and therefore,
\begin{align}
\label{dfceg6}
\sup _{ 0\le s\le t }&||(G-G ')(s)|| _{ \theta, \rho}\le
C\Bigg(\sup _{ 0\le s\le t }|(\beta -\beta')(s)|\times \nonumber\\
&\times \left( ||Q_N(\varphi )|| _{ \theta, \rho }+\sup _{ 0\le s\le t }||g(s)|| ^2_{ -r, r+q }\right)
+\sup _{ 0\le s\le t }||g(s)-g'(s)|| _{ -r, r+q }\times \nonumber \\
&\hskip 4cm \times \sup _{ 0\le s\le t }\left( ||g(s))|| _{ -r, r+q }+||g'(s)|| _{ -r, r+q }\right)\Bigg).
\end{align}

Similarly, if $\theta'=-r,\rho '=r+q+1$ with $0<r<1/2$,
\begin{align*}
&||(\beta-\beta')(s)Q_N(\varphi )|| _{\frac {1} {2}+ \theta', \rho'-1 }\le  |(\beta-\beta')(s)| ||Q_N(\varphi )|| _{\frac {1} {2}+ \theta', \rho'-1 }\le C |(\beta-\beta')(s)|\\
&||(\beta^{-1}(s)-\beta'^{-1}(s))Q_N(g(s))||  _{\frac {1} {2}+ \theta', \rho' -1}\le C |(\beta -\beta')(s)| ||Q_N(g(s))|| _{\frac {1} {2}+ \theta', \rho'-1 }\\
&\hskip 5.7cm \le C |(\beta -\beta')(s)||g(s)|| ^2_{ -r, r+q }\\
&||Q_n(g(s))-Q_n(g'(s))|| _{\frac {1} {2}+\theta', \rho '-1 }\le C||(g-g')(s)|| _{ -r, r+q }(||g(s)|| _{ -r, r+q }+||g'(s)|| _{ -r, r+q }).
\end{align*}
and then, for all $\theta\ge \frac {1} {2}-r$ and $\theta+\rho \le \frac {1} {2}+q$,
\begin{align}
\label{dfceg7}
&\sup _{ 0\le s\le t }||(G-G')(s)|| _{ \theta, \rho}\le \sup _{ 0\le s\le t }||(G-G')(s)|| _{\frac {1} {2}+ \theta', \rho'-1  }\nonumber\\
&\le C\Bigg(\sup _{ 0\le s\le t }|(\beta -\beta')(s)| \left(||Q_N(\varphi )|| _{\frac {1} {2}+ \theta', \rho'-1 }+\sup _{ 0\le s\le t }||g(s)||^2 _{ -r, r+q }\right)+
 \nonumber\\
&+\sup _{ 0\le s\le t }||g(s)-g'(s)|| _{ -r, r+q }\sup _{ 0\le s\le t }\left( ||g(s))|| _{ -r, r+q }+||g'(s)|| _{ -r, r+q }\right)\Bigg).
\end{align}
From (\ref{dfceg}), (\ref{dfceg6}), (\ref{dfceg7}),

\begin{align}
\label{dfceg9}
&|\mathscr W(t, G -G')|\le C\tau (t)\Bigg(\sup _{ 0\le s\le t }|(\beta -\beta')(s)|\times \nonumber\\
&\times \left(||Q_N(\varphi )|| _{ \theta, \rho }+\sup _{ 0\le s\le t }||g(s)||^2 _{ -r, r+q }\right)
+\sup _{ 0\le s\le t }||g(s)-g'(s)|| _{ -r, r+q }\times \nonumber \\
&\hskip 4cm \times \sup _{ 0\le s\le t }\left( ||g(s))|| _{ -r, r+q }+||g'(s)|| _{ -r, r+q }\right)\Bigg).
\end{align}
where $\tau (t)=t^{1-2\theta}+t^{2-2\theta}\to 0$ as $t\to 0$ if $\theta<1/2$.

On the other hand, since $r_0$ is linear with respect to $G$, by Proposition \ref{xixixi},
\begin{align*}
|r_0(t, G, X)-r_0(t, G', X)|\le C\sup _{ 0\le s\le t }||G(s)-G'(s)|| _{ \theta, \rho  }X^{\frac {1} {2}-\theta},\,\forall X\in (0, 1).
\end{align*}
Using now that $G$ is quadratic in $g$, it follows by  (\ref{dfceg6}), (\ref{dfceg7}), for $X\in (0,1)$, $\theta\ge 0$, $\rho \ge 0$, $\theta+\rho<q+\frac {1} {2}$

\begin{align*}
& |r_0(t, G, X)-r_0(t, G', X)|\le
CX^{\frac {1} {2}-\theta} \Bigg(\sup _{ 0\le s\le t }|(\beta -\beta')(s)|\times \nonumber\\
&\times \left(||Q_N(\varphi )|| _{ \theta, \rho }+\sup _{ 0\le s\le t }||g(s)||^2 _{ -r, r+q }\right)
+\sup _{ 0\le s\le t }||g(s)-g'(s)|| _{ -r, r+q }\times \nonumber \\
&\hskip 4cm \times \sup _{ 0\le s\le t }\left( ||g(s))|| _{ -r, r+q }+||g'(s)|| _{ -r, r+q }\right)\Bigg).
\end{align*}
and
\begin{align}
\label{dfceg10}
\sup _{ 0\le X\le \min(1, t^2)) }&X^\theta |r_0(t, G, X)-r_0(t, G', X)|\le
C t\Bigg(\sup _{ 0\le s\le t }|(\beta -\beta')(s)|\times \nonumber\\
&\times \left(||Q_N(\varphi )|| _{ \theta, \rho }+\sup _{ 0\le s\le t }||g(s)||^2 _{ -r, r+q }\right)
+\sup _{ 0\le s\le t }||g(s)-g'(s)|| _{ -r, r+q }\times \nonumber \\
&\hskip 4cm \times \sup _{ 0\le s\le t }\left( ||g(s))|| _{ -r, r+q }+||g'(s)|| _{ -r, r+q }\right)\Bigg).
\end{align}

On the other hand if $X>1$, by Proposition \ref{xixixi}
\begin{align*}
&|r_0(t, G, X)-r_0(t, G', X)|\le  C\sup _{ 0\le s\le t }||G(s)-G'(s)|| _{ \theta, \rho  }t^{4-2\theta}X^{-3/2}\\
&\le t^{4-2\theta}X^{-3/2}\Bigg(\sup _{ 0\le s\le t }|(\beta -\beta')(s)|\left(||Q_N(\varphi )|| _{ \theta, \rho }+\sup _{ 0\le s\le t }||g(s)||^2 _{ -r, r+q }\right)+\\
&+\sup _{ 0\le s\le t }||g(s)-g'(s)|| _{ -r, r+q } \sup _{ 0\le s\le t }\left( ||g(s))|| _{ -r, r+q }+||g'(s)|| _{ -r, r+q }\right)\Bigg)
\end{align*}
and, since $\boxed{\theta+\rho \le 3/2}$,
\begin{align}
\label{dfceg11}
&X^\theta(1+X)^\rho |r_0(t, G, X)-r_0(t, G', X)|\le
 t^{4-2\theta} \Bigg(\sup _{ 0\le s\le t }|(\beta -\beta')(s)|\times \nonumber\\
&\hskip 0.5cm\times \left(||Q_N(\varphi )|| _{ \theta, \rho }+\sup _{ 0\le s\le t }||g(s)||^2 _{ -r, r+q }\right)
+\sup _{ 0\le s\le t }||g(s)-g'(s)|| _{ -r, r+q }\times \nonumber \\
&\hskip 4cm \times \sup _{ 0\le s\le t }\left( ||g(s))|| _{ -r, r+q }+||g'(s)|| _{ -r, r+q }\right)\Bigg).
\end{align}
We obtain from (\ref{dfceg10}) and (\ref{dfceg11})
\begin{align}
 \label{dfceg12}
&\sup _{ X >0 }X^\theta (1+X)^\rho  |r_0(t, G, X)-r_0(t, G', X)|\le
Ct \Bigg(\sup _{ 0\le s\le t }|(\beta -\beta')(s)|\times \nonumber\\
&\qquad \times \left((||Q_N(\varphi )|| _{ \theta, \rho }+\sup _{ 0\le s\le t }||g(s)||^2 _{ -r, r+q }\right)
+\sup _{ 0\le s\le t }||g(s)-g'(s)|| _{ -r, r+q }\times \nonumber \\
&\hskip 6cm \times \sup _{ 0\le s\le t }\left( ||g(s))|| _{ -r, r+q }+||g'(s)|| _{ -r, r+q }\right)\Bigg).
\end{align}
\end{proof}
\subsection{The function  $\tilde g_2$.}
The function $\tilde g_2$ is not obtained directly solving
\begin{align*}
&\frac {\partial \tilde g_2} {\partial t}=\mathscr L_\varphi ( \tilde g_2 ) -\beta '(t)\varphi  \\
&\tilde g_2(0, X)=0,
\end{align*}
in order to avoid the use of $\beta '$ for the moment. It is then defiined in the following indirect way. 

Let us solve first the following problem,
\begin{align}
\frac {\partial \psi } {\partial t}=\mathscr L_\varphi (\psi ),\,\,\psi (0)=\varphi  \label{Epsi1}
\end{align}
with $\varphi $ defined in (\ref{defvfiR}). To this end let us define the function
\begin{align}
\label{defFi}
\Phi (X)=\mathscr L_\varphi \left( \mathscr L_\varphi (\varphi )\right) (X),\,\,X>0
\end{align} 
and prove the following
\begin{prop}
\label{propFi}
There exists two  constants $C_1>0$ and $C_2>0$ such that for all $R>1$,
\begin{align*}
&|\Phi (X)|\le C_1R^{-3/2}\left(\frac {X} {R}\right)^{\frac {1} {2}}=\frac {C_1X^{1/2}} {R^2},\,\,\forall 0<X<R,\\
&|\Phi (X)|\le C_2R^{-3/2}\left(\frac {X} {R}\right)^{-\frac {3} {2}}=C_2X^{-3/2},\,\,\forall  X>R.
\end{align*}
It follows that $\Phi \in X _{ -\frac {1} {2},\, 2 }$ and for each $\rho \in [0, 2)$
\begin{align*}
||\Phi || _{- \frac {1} {2},\, \rho  }\le \frac {||\Phi || _{ -\frac {1} {2}, 2 } } {R^{2-\rho } }. 
\end{align*}
For all $\theta\ge -1/2$ and $\rho \ge 0$ such $\theta+\rho <3/2$,
\begin{align*}
\forall \rho '\in \left(\theta+\rho +\frac {1} {2}, 2\right),\,\,\,||\Phi || _{ \theta, \rho  }\le ||\Phi || _{ \frac {1} {2}, \rho ' }\underset{R\to \infty}{\rightarrow} 0.
\end{align*}
\end{prop}
\begin{proof}
Since  $\mathscr L_\varphi (\varphi )=2Q_N(\varphi )\in X _{ \theta, \rho  }$  it follows   from the hypothesis on $\phi $ in (\ref{deffi1}), (\ref{deffi2}) that for all  $s\in \CC$ such that $\Re e (s)>-1/2$,
\begin{align*}
\mathscr M  (\Phi  )(s)&=
(s-1)\mathscr M \Big(\mathscr L_\varphi (\varphi )\Big)(s-1)\\
&=(s-1)\int _0^\infty\mathscr L_\varphi (\varphi )(X)X^{s-2}dX.
\end{align*}
Since $\mathscr L_\varphi (\varphi )\in C_c^\infty([0, \infty))$ it follows from integration by parts that for all $k\in \NN$ there exists $C_k$ such that $\left|\mathscr M  (\Phi )(s)\right|\le C_k(1+|s|^{k})^{-1}$. 

It also  follows from the hypothesis on $\phi $ and the properties of $Q_N(\varphi )$ that $ \mathscr M(\Phi )$ is analytic in  $S_{-1/2, \infty}=\left\{s\in \CC;\, \Re e (s)>-1/2 \right\}$, has  a simple pole at $s=-1/2$ may be extended as meromorphic function in  the complex plane.
By classical results on the Mellin transform,
\begin{align*}
\Phi (X)&=\int  _{ \Re e(s)=\beta  }\mathscr M(\Phi )(s)X^{-s}ds\\
&=\int  _{ \Re e(s)=\beta  }(s-1)\mathscr M \Big(\mathscr L_\varphi (\varphi )\Big)(s-1)X^{-s}ds,\,\,\beta >-1/2,
\end{align*}
where,  since $\mathscr L_\varphi (\varphi )=2Q_N(\varphi )$,
\begin{align*}
\mathscr M \Big(\mathscr L_\varphi (\varphi )\Big)(s-1)=
2\int _0^\infty Q_N(\varphi )(X)X^{s-2}dX\\
=R^{s-\frac {3} {2}}\int _0^\infty \widetilde Q_N(\phi )(X')X'^{s-2}dX'\\
=R^{s-\frac {3} {2}}\mathscr M\left( \widetilde Q_N(\phi ) \right)(s-1).
\end{align*}
It follows thet $\Phi $ may be written as
\begin{align*}
\Phi (X)=\int  _{ \Re e(s)=\beta  }(s-1)R^{s-\frac {3} {2}}\mathscr M\left( \widetilde Q_N \right)(s-1)X^{-s}ds\\
=R^{-\frac {3} {2}}\int  _{ \Re e(s)=\beta  }(s-1)\mathscr M\left( \widetilde Q_N \right)(s-1)
\left(\frac {X} {R}\right)^{-s}ds.
\end{align*}
Then, for all $k\in \NN$,
\begin{align*}
&\Phi (X)= Res\left( \mathscr M(\Phi ), s=-\frac {1} {2} \right) R^{-3/2}\left(\frac {X} {R}\right)^{\frac {1} {2}}+R^{-3/2}\mathcal O\left(\frac {X} {R}\right)^{k},\,\,\forall  X\in (0, R),
\end{align*}
and for all $\ell>0$ there exists a constant $C_\ell>0$ such that
\begin{align*}
|\Phi (X)|\le C _{ \ell }R^{-3/2}\left(\frac {X} {R}\right)^{-\ell}=C _{ \ell }R^{-\frac {3} {2}+\ell}X^{-\ell},\,\,\forall  X>R.
\end{align*}
and  $\Phi \in X _{- \frac {1} {2},\, \ell+\frac {1} {2} }$ for all $\ell>0$. Moreover, and for each $\ell \in [0, 3/2)$ there exists a constant $C_\ell >0$ such that
\begin{align*}
||\Phi || _{- \frac {1} {2},\, \ell+\frac {1} {2}  }\le \frac {C_\ell } {R^{\frac {3} {2}-\ell } } 
\end{align*}
and then, for all $\rho=\ell+\frac {1} {2} \in [0, 2)$ there exists $C_\rho >0$ such that
\begin{align*}
||\Phi || _{ -\frac {1} {2},\,\rho  }\le \frac {C_\rho  } {R^{2-\rho  } } 
\end{align*}
If $\theta\ge -1/2$, $\rho\ge 0$ and $\theta+\rho <3/2$ then for all $\rho '\in \left(\theta+\rho +\frac {1} {2}, 2\right)$,
\begin{align*}
||\Phi || _{ \theta, \rho  }\le ||\Phi || _{- \frac {1} {2}, \rho'  }\le \frac {C_\rho '} {R^{2-\rho '}}.
\end{align*}
\vskip -0.5cm 
\end{proof}
\begin{prop}
For all $\varphi $  as in (\ref{defvfiR}) there exists a function $\psi $, such that $\psi \in C((0, \infty); X _{ 0, \frac {3} {2}  }$ that satisfies (\ref{Epsi1}) for $t >0$, $X> 0$.
Moreover, if $r_2(t, X)$ and $a(t)$ are defined as
\begin{align}
&a(t)=1 +\int _0^t \mathscr W(s, \mathscr L_\varphi \left( \mathscr L_\varphi (\varphi )\right), X)ds\label{S7Da}\\
&r_2(t, X)=\psi (t, X)-a(t)\eta(X)\nonumber
\end{align}
then there exists $C>0$ constant such that for all $t\in (0, 1)$,
\begin{align}
&|r_2(t, X)|\le 
\begin{cases}
\label{estr2}
C\left(||Q_N(\varphi )|| _{ -\frac {1} {2}, \rho  }+ ||\Phi || _{ 0, \frac {3} {2}  }\right)\, tX^{\frac {1} {2}},\,\,\forall X\in (0, 1)\\
 C\left(t||Q_N(\varphi )|| _{ -\frac {1} {2}, \rho  }X^{- 3/2}+  ||\Phi || _{ 0, 3/2  }t^5X^{-3/2}\right),\,\,\forall X>1,
\end{cases}
\end{align}
and
\begin{align}
&|a'(t)|\le C||\Phi || _{ \theta, \rho  }t^{1-2\theta} \label{ap}\\
&\left|\frac {\partial r_2(t, X)} {\partial t}\right|\le 
\begin{cases}
C||\Phi || _{ -\frac {1} {2}, 2  }t^2 +C\frac {X^{1/2}} {R}  ,\,\,X\le1\\
C\frac {X^{1/2}} {R}\1 _{ X<\alpha R }+CX^{-1/2}\1 _{ \alpha R<X<\beta R }+C||\Phi || _{ \theta, \rho  }t^{4}X^{- \frac {3} {2}},\,X>1,
\end{cases}
\label{r2p}
\end{align}
\end{prop}
\begin{proof}
Under the change of variable $\zeta (t, X)=\psi (t, X)-\varphi (x)$,  problem (\ref{Epsi1}) reads,
\begin{align}
\frac {\partial \zeta } {\partial t}=\mathscr L_\varphi (\zeta )+\mathscr L_\varphi (\varphi ),\,\,\zeta (0)=0. \label{Ezeta1}
\end{align}
Since $\mathscr L_\varphi (\varphi )=2Q_N(\varphi )$ by Proposition \ref{QQ0} $\mathscr L_\varphi (\varphi )\in X _{ \theta, \rho  }$ for all $\theta\ge -1/2$ and $\rho \ge 0$. Then, by Proposition \ref{SFi2} if $\theta\ge 0$ and $\rho \ge 0$   the problem (\ref{Ezeta1}) has a  solution $\zeta $, unique in $C((0, \infty); X _{ \theta-\frac {1} {2}, \rho+1  })\cap L^\infty _{ \text{loc} }( [0, \infty) ; X _{ \theta-\frac {1} {2}, \rho  +1})$ given by,
\begin{align*}
\zeta (t)=\int _0^t S_\varphi (t-s) \Big(\mathscr L_\varphi (\varphi ) \Big)ds.
\end{align*}

Since, by Proposition (\ref{propFi}) $\Phi \in X _{-\frac {1} {2}, 2 }\subset X _{ 0, 3/2  }$ let us define  the function
\begin{align}
\label{defW}
W(t)=\int _0^tS_\varphi(t-s) \Phi (s)ds
\end{align}
By Proposition \ref{SFi2} $W\in L^\infty((0, T); X_{0,  3/2})\cap W\in C\left((0, \infty); X _{ 0,  3/2})\right)$, $\partial _tW\in  L^\infty _{ \text{loc} }(0, \infty)\times (0, \infty))$ and,
\begin{align}
&W_t(t, X)=\mathscr L_\varphi (W(t))(X)+\Phi(X),\,\,t>0, X>0,   \label{WW1}\\
&W(0)=0,\label{WW2}\\
& ||W(t)|| _{0,  3/2}\le C||\Phi || _{ -\frac {1} {2},\,2  }\,\,\forall t>0 \nonumber\\
&\forall \rho \in [0, 3/2),\,\, \exists C_\rho >0;\,\,||W(t)|| _{0,  \rho }\le C||\Phi || _{ -\frac {1} {2}, \rho   }\le C|| \Phi  || _{ -\frac {1} {2}, 2 }R^{2-\rho },\,\forall t>0 \nonumber
\end{align}
where the constants $C>0$  and $C_\rho $ are independent of $\Phi $. Since $\Phi \in X _{ -\frac {1} {2}, \rho  }$ for all $\rho \in [0, 2]$,
\begin{align}
&\forall \rho \in \left[0,  3/2\right),\,\exists C_\rho >0;\,\,\,||W(t)|| _{ 0,\rho  }\le \frac {C} {R^{2-\rho }}||\Phi || _{ \frac {1} {2},\, \rho   }
\,\,\forall t>0.
\end{align}
If the function $r_1(t, X)$ is now defined as,
\begin{align*}
r_1(t, X)=W(t, X)-\mathscr W(t, \Phi )\eta(X)
\end{align*}
By proposition \ref{xixixi}  with $\theta=0$,
\begin{align}
&|r_1(t, \Phi , X)|\le  C ||\Phi || _{ 0, 3/2  }X^{\frac {1} {2}},\,\forall X\in (0, 1)\label{lzm31}\\
&|r_1(t, X)|\le  C ||\Phi || _{ 0, 3/2  }t^{4}X^{-\frac {3} {2}},\,\forall X>1.\label{lzm33}
\end{align}
Therefore, the function 
\begin{align}
\label{defw}
w= \mathscr L_\varphi (\varphi )+W
\end{align} satisfies,
\begin{align}
&w_t=\mathscr L_\varphi (w) \label{defw1}\\
&w(0)=\mathscr L_\varphi (\varphi ) \label{defw2}
\end{align}
from where $w(t)=S_\varphi (t)\Big(\mathscr L_\varphi (\varphi )\Big)$.  On the other hand, since by construction $\zeta (t)=\int _0^t S_\varphi (t-s)\Big(\mathscr L_\varphi (\varphi )\Big)ds$ it follows by uniqueness that $\zeta (t)=\int _0^tw(s)ds=t \mathscr L_\varphi (\varphi )+\int _0^t W(s)ds$ and then
\begin{align}
\psi (t, X)&=\varphi(X) +t \mathscr L_\varphi (\varphi )(X)+\int _0^t W(s, X)ds.\label{defpsi}
\end{align}
(Notice indeed that $\psi _t(t, X)$ is then equal to $\mathscr L_\varphi (\varphi )+W(t, X)$ and $\mathscr L_\varphi (\psi (t))(X)$ is equal  to $\mathscr L_\varphi (\varphi )+t\Phi (X)+\int _0^t\mathscr L_\varphi (W(s))ds$ and then $\psi _t(t, X)=\mathscr L_\varphi (\psi (t))(X)$ using (\ref{WW1}).

If we denote now
\begin{align}
\label{defa}
&\psi (t, X)=a(t)\eta(X)+r_2(t, X)\,\,\,\text{where}\,\,\,\,a(t)=1+\int _0^t \mathscr W(s, \mathscr L_\varphi \left( \mathscr L_\varphi (\varphi )\right))ds
\end{align}
then
\begin{align*}
r_2(t, X)=\varphi (X)-a(t)\eta(X)+ t \mathscr L_\varphi (\varphi )(X)+\int _0^tr_1(s, X)ds.
\end{align*}
By (\ref{lzm31}), when $t<1$ and $X\in (0, 1)$,
\begin{align*}
&\int _0^tr_1(s, X)ds \le C ||\Phi || _{ 0, 3/2  }X^{\frac {1} {2}}\int _0^{t}ds \le C ||\Phi || _{ 0, 3/2  }tX^{1/2}
\end{align*}
If $X>1$, by  (\ref{lzm33}),
\begin{align*}
&\int _0^tr_1(s, X)ds\le  C||\Phi || _{ 0, 3/2  } X^{-\frac {3} {2}} \int _0^t s^{4}ds\le C||\Phi || _{ 0, 3/2  } t^5X^{-\frac {3} {2}}. 
\end{align*}
Use of  Proposition (\ref{QQ0}) yields then, for $\rho\ge 0$,
\begin{align*}
|r_2(t, X)|&\le t \left|\mathscr L_\varphi (\varphi )(X)\right|+C ||\Phi || _{ 0, 3/2  }tX^{1/2}\\
&\le C||Q_N(\varphi )|| _{ -\frac {1} {2}, \rho  }t X^{1/2}+ C ||\Phi || _{ 0, 3/2  }tX^{1/2},\, \,\forall X\in (0, 1),\\
|r_2(t, X)|&\le t \left|\mathscr L_\varphi (\varphi )(X)\right|+C ||\Phi || _{ 0, 3/2  }t^5X^{-3/2}\\
&\le C||Q_N(\varphi )|| _{ -\frac {1} {2}, \rho  }t X^{- 3/2}+ C ||\Phi || _{ 0, 3/2  }t^5X^{-3/2},\, \,\forall X>1.
\end{align*}

We deduce, for $t\in (0, 1)$, $\rho \in (0, 1)$,
\begin{align*}
&|r_2(t, X)|\le 
\begin{cases}
C\left(||Q_N(\varphi )|| _{ -\frac {1} {2}, \rho  } + ||\Phi || _{ 0, \frac {3} {2}  }\right)\, tX^{\frac {1} {2}},\,\,\forall X\in (0, 1)\\
 C\left(t||Q_N(\varphi )|| _{ -\frac {1} {2}, \rho  } X^{- 3/2}+  ||\Phi || _{ 0, 3/2  }t^5X^{-3/2}\right),\,\,\forall X>1.
\end{cases}
\end{align*}

On the other hand, 
\begin{align*}
\mathscr W(t, \mathscr L_\varphi \left( \mathscr L_\varphi (\varphi )\right))&=\int _0^{t }L(t-s,  \Phi )ds+\\
&+\int _0^t \int _0^{t-s }\Big(L(t-s-\sigma ,  T[S(\sigma ) \Phi)d\sigma ds
\end{align*}
with,
\begin{align*}
\left|\int _0^{t }L(t-s,  \Phi )ds\right|\le C||\Phi  || _{ -\frac {1} {2}, 2 }t^2
\end{align*}
and
\begin{align*}
&\left|\int _0^t \int _0^{t-s }\Big(L(t-s-\sigma ,  T[S(\sigma ) \Phi ])\Big)d\sigma ds \right|\le
C\int _0^t \int _0^{t-s }||T(S(\sigma ))\Phi  || _{ \theta , \rho   }\times \\
&\hskip 9cm \times (t-s-\sigma )^{-2\alpha }d\sigma ds
\end{align*}
for $\theta \ge -1/2$, $\theta +\rho  \le 2$, and then
\begin{align*}
&\left|\int _0^t \int _0^{t-s }\Big(L(t-s-\sigma ,  T[S(\sigma ) \Phi ])\Big)d\sigma ds \right|\le\\
&\le C\int _0^t \int _0^{t-s }||S(\sigma )\Phi  || _{ 0 , \frac {3} {2}  } (t-s-\sigma )^{-2\alpha }d\sigma ds\le Ct^{2-2\alpha }||\Phi  || _{ 0, 3/2 }.
\end{align*}
The choice $\theta=-1/2$ gives then
\begin{align*}
\left|\mathscr W(t, \mathscr L_\varphi \left( \mathscr L_\varphi (\varphi )\right))\right|\le 
||\Phi  || _{ -\frac {1} {2}, 2 }t^2,\,\,\forall t\in (0, 1).
\end{align*}

Since by definition,
\begin{align*}
&a'(t)=\mathscr W(t, \mathscr L_\varphi \left( \mathscr L_\varphi (\varphi )\right)),\\
&\frac {\partial r_2(t, X)} {\partial t}=-a'(t)\eta(X)+\mathscr L_\varphi (\varphi )(X)+r_1(t, X).
\end{align*}
it follows,
\begin{align*}
|a'(t)|\le C||\Phi || _{ -\frac {1} {2}, 2  }t^2,\,\,\,\forall t\in (0, 1),
\end{align*}
Using also Proposition \ref{WWW} again, (\ref{lzm31}) and (\ref{lzm33}),  for all $ \theta\ge 1/2, \rho \ge 0;\,\,\rho +\theta\le 3/2$
\begin{align*}
\left|\frac {\partial r_2(t, X)} {\partial t}\right|&\le C |a'(t)|\1 _{ X<1 }+
2|Q_N(\varphi )|+|r_1(t, X)|\\
&\le C||\Phi || _{ -\frac {1} {2}, 2  }t^2\1 _{ X<1 }+C\frac {X^{1/2}} {R}\1 _{ X<\alpha R }+CX^{-1/2}\1 _{ \alpha R<X<\beta R }+\\
&+C||\Phi || _{ -\frac {1} {2}, -2  }tX^{\frac {1} {2}}\1 _{ 0<X<1 }+
C||\Phi || _{ -\frac {1} {2}, 2  }t^{5}X^{- \frac {3} {2}}\1 _{ X>1 }
\end{align*}
and,
\begin{align*}
\left|\frac {\partial r_2(t, X)} {\partial t}\right|\le 
\begin{cases}
C||\Phi || _{ -\frac {1} {2}, 2  }t^2 +C\frac {X^{1/2}} {R}  ,\,\,X\le1\\
C\frac {X^{1/2}} {R}\1 _{ X<\alpha R }+CX^{-1/2}\1 _{ \alpha R<X<\beta R }+C||\Phi || _{ \theta, \rho  }t^{4}X^{- \frac {3} {2}},\,X>1.
\end{cases}
\end{align*}
\end{proof}
The function $\widetilde g_2$ is now defined  as,
\begin{align}
\label{lngrsg2}
\widetilde g_2(t, X)=- \beta (t)\, \varphi(X)+\psi (t, X)-\int _0^tw(t-s, X)\beta (s)ds.
\end{align}
and denote also $R_w$ to be the function such as,
\begin{align}
&w(t-s, X)=\Lambda(t-s; Q_N(\varphi ))+R_w(t-s, Q_N(\varphi ), X) \label{Rwdef} \\
&\Lambda(t-s; Q_N(\varphi ))=L(t-s; Q_N(\varphi ))+ \int _0^{t-s}L(t-s-\sigma ; Q_N(\varphi ))d\sigma 
\label{Ldef}
\end{align}
\begin{prop}
\label{mnst2}
For  $\beta \in C([0, T])$ such that $\sup _{ 0\le t\le T }|\beta (t)-1|<1/4$ and $\widetilde g_2$ defined by (\ref{lngrsg2}), define,
\begin{align}
\label{lngrsg2B}
r_3(t, X)=\widetilde g_2(t, X)+\left( \beta (t)-a(t)+\int _0^t \Lambda(t-s, Q_N(\varphi ))\beta (s)ds\right)\eta(X).
\end{align}
Then, for $\rho \ge 0$ 
\begin{align}
|r_3(t, X)|\le
\begin{cases}
&C\left(||Q_N(\varphi )|| _{ -\frac {1} {2}, \rho  }+ ||\Phi || _{ 0, \frac {3} {2}  }\right)\, tX^{\frac {1} {2}}+\\
&\hskip 2cm +C_\rho \sup _{ 0\le s\le t }|\beta (s)| ||Q_N(\varphi )|| _{ 0, \rho  }X^{\frac {1} {2}},\,\forall X\in (0, 1)\\
 &C\left(t||Q_N(\varphi )|| _{ -\frac {1} {2}, \rho  }X^{- 3/2}+  ||\Phi || _{ 0, 3/2  }t^5X^{-3/2}\right)+\\
 &\hskip 2cm +C\sup _{ 0\le s \le t }|\beta (s)| ||Q_N(\varphi )|| _{ 0, \rho  }tX^{-3/2},\,\,\forall X>1
\end{cases}
\end{align}
and,
Moreover, the map:
\begin{align*}
\beta \to \widetilde g_2(\beta)
\end{align*}
is such that, for a constant $C=C(T, ||Q_N(\varphi ) || _{ 0, \frac {3} {2} }, ||\varphi ||_\infty)>0$,
\begin{align*}
\sup_{\substack { X>0\\ 0\le t\le T }}\left|\widetilde g_2(\beta _1; t, X)-\widetilde g_2(\beta _2; t, X)\right|(1+X)^{\theta+\rho +\frac {1} {2}}\le
C(T, ||\Phi || _{ \theta, \rho  }, ||\varphi ||_\infty) \sup _{ 0\le t\le T }|\beta _1(t)-\beta _2(t)|.
\end{align*}
\end{prop}
\begin{proof}
Since by definition $w(t, X)=S_\varphi (t)Q_N(\varphi )$, by Corollary 7.14 of \cite{m3}
\begin{align*}
&w(t-s, X)=\Lambda(t-s, Q_N(\varphi ))+R_w(t-s, Q_N(\varphi ), X) \\
&R_w(t-s, Q_N(\varphi ), X)=R_2(t-s, Q_N(\varphi ), X)+ \int _0^{t-s}R_2(t-s-\sigma , Q_N(\varphi ), X)d\sigma
\end{align*}

The same argument as in the proof of  Proposition \ref{WWW} shows that,
\begin{align}
\label{mntrvll64}
\int _0^t \left|\Lambda(t-s; Q_N(\varphi ))\right|ds\le 
C|||Q_N(\varphi )|| _{ 0, \frac {3} {2} }t 
\end{align}
and then
\begin{align*}
\left|\int _0^t \beta (s)\Lambda(t-s; Q_N(\varphi ))ds\right|\le 
C\sup _{ 0\le s\le t }|\beta (s)|||Q_N(\varphi )|| _{ 0, \frac {3} {2} }t
\end{align*}
And the same argument as in the proof of Proposition \ref{xixixi} gives, for all $\rho >0$
\begin{align}
\label{mntrvll65}
\int _0^t|R_w(t-s, Q_N(\varphi ), X)|ds\le 
\begin{cases}
\displaystyle{C_\rho ||Q_N(\varphi )|| _{ 0, \rho  }X^{1/2},\,\forall X\in (0, 1)}\\
\displaystyle{ C_\rho ||Q_N(\varphi )|| _{ 0, \rho }t^4X^{-3/2},\,\forall X>1.}
\end{cases}
\end{align}

Therefore, when $X\in (0, 1)$ and $t\in (0, 1)$,
\begin{align*}
&r_3(t, X)=r_2(t, X)-\int _0^t R_w(t-s, Q_N(\varphi )) \beta (s)ds\\
&|r_3(t, X)|\le C\left(||Q_N(\varphi )|| _{ -\frac {1} {2}, \rho  }+ ||\Phi || _{ 0, \frac {3} {2}  }\right)\, tX^{\frac {1} {2}}+
C_\rho \sup _{ 0\le s\le t }|\beta (s)| ||Q_N(\varphi )|| _{ 0, \rho  }X^{\frac {1} {2}},\,\forall \rho \ge 0.
\end{align*}
For $X>1$, by the definition of $r_3$, (\ref{defa}) and (\ref{estr2}),
\begin{align*}
&r_3(t, X)=\tilde g_2(t, X)=r_2(t, X)-\int _0^tw(t-s, X)\beta (s)ds\\
&|r_3(t, X)|\le  C\left(t||Q_N(\varphi )|| _{ -\frac {1} {2}, \rho  } X^{-1/2}+  ||\Phi || _{ 0, 3/2  }t^5X^{-3/2}\right)+\\
&\hskip 6cm +\sup _{ 0\le s \le t }|\beta (s)|\int _0^t|w(t-s, X)|ds\\
&\quad \le  C\left(t||Q_N(\varphi )|| _{ -\frac {1} {2}, \rho  } X^{-1/2}+  ||\Phi || _{ 0, 3/2  }t^5X^{-3/2}\right)+\\
&\hskip 5cm+C\sup _{ 0\le s \le t }|\beta (s)|\int _0^t|S_\varphi (t-s)(Q_N(\varphi ))(X)|ds.
\end{align*}
Arguing as in the proof of Proposition \ref{SFi2} it follows that for all $\rho\ge 3/2$ there exists $C>0$,
\begin{align*}
|r_3(t, X)| \le &C\left(t||Q_N(\varphi )|| _{ -\frac {1} {2}, \rho  }X^{-1/2}+  ||\Phi || _{ 0, 3/2  }t^5X^{-3/2}\right)+\\
&\qquad +C\sup _{ 0\le s \le t }|\beta (s)| ||Q_N(\varphi )|| _{ 0, \rho  }tX^{-3/2},\,\,\forall X>1,
\end{align*}
Since the functions $a$ and  $w$ are independent of $\beta $, we may write,
\begin{align*}
|\widetilde g_2(\beta _1; t, X)-\widetilde g_2(\beta _2; t, X)| \le
|\beta _1(t)-\beta _2(t)||\varphi (X)|+\int _0^t|w(t-s, X)||\beta _1(s)-\beta _2(s)|ds.
\end{align*}
where, by hypothesis,
\begin{align*}
|\beta _1(t)-\beta _2(t)||\varphi (X)|\le \sup _{ 0\le t\le T }|\beta _1(t)-\beta _2(t)| ||\varphi || _{ \infty },\,\,\forall X>0.
\end{align*}
On the other hand,
\begin{align*}
\int _0^t|w(t-s, X)||\beta _1(s)-\beta _2(s)|ds\le \sup _{ 0\le t\le T }|\beta _1(t)-\beta _2(t)|\int _0^t|w(t-s, X)|ds.
\end{align*}
By (\ref{Rwdef}),  (\ref{mntrvll64}) and (\ref{mntrvll65}), if  $X\in (0, 1)$,
\begin{align*}
\int _0^t|w(t-s, X)|ds\le C|||Q_N(\varphi )|| _{ 0, \frac {3} {2} }t+C_\rho ||Q_N(\varphi )|| _{ 0, \rho  }X^{1/2}
\end{align*}
for all $\rho >0$, and then,
\begin{align*}
\sup_{\substack { 0<X<1\\ 0\le t\le T }}\left|\widetilde g_2(\beta _1; t, X)-\widetilde g_2(\beta _2; t, X)\right|&\le
C\sup _{ 0\le t\le T }|\beta _1(t)-\beta _2(t)|\times \\
&\times \left( ||\varphi ||_\infty+ |||Q_N(\varphi )|| _{ 0, \frac {3} {2} }(T+X^{1/2})  \right).
\end{align*}

When $X>1$, by (\ref{Rwdef}),  (\ref{mntrvll64}) and (\ref{mntrvll65}) again
\begin{align*}
\int _0^t|w(t-s, X)|ds\le C|||Q_N(\varphi )|| _{ 0, \frac {3} {2} }t+ C_\rho ||Q_N(\varphi )|| _{ 0, \rho }t^4X^{-3/2},\,\forall X>1.
\end{align*}
and
\begin{align*}
\sup_{\substack { X>1\\ 0\le t\le T }}\Big|\widetilde g_2(\beta _1; t, X)-&\widetilde g_2(\beta _2; t, X)\Big|\le
C\sup _{ 0\le t\le T }|\beta _1(t)-\beta _2(t)|\times \\
&\times \left( ||\varphi ||_\infty+   ||Q_N(\varphi )|| _{ 0, \frac {3} {2} }\left(T+T^4X^{-3/2}\right)\right).
\end{align*}
\end{proof}
\section{Proof of Theorem \ref{Mth1}}
\label{SMth1}
\setcounter{equation}{0}
\setcounter{theo}{0}
As indicated at the beginning of Section \ref{Sg1g2},  the function $\overline g(\overline \tau )$ introduced in (\ref{S1E09}) is, with some abuse of notation still denoted $g(t)$ until Subsection  \ref{endp} at the end  of  this Section. 

Let us first of all  denote, for $T>0, K>0$, 
\begin{align*}
&B _{ T, X _{ -r, r+q } }(0,  K)=\left\{h\in C([0, T); X _{ -r, r+q }); \sup _{ 0\le s\le T }||h(s)|| _{ -r, r+q }< K
\right\}\\
&\text{and}\,\,B _{ T}(1, 1/4)=\left\{\beta \in C([0, T]);\,\,\sup _{ 0\le s\le T }|\beta (s)-1|<1/4 \right\}
\end{align*}

\subsection{Overview of the argument.}
\label{overvw}
The local classical solution $f(t, X)$ of  equation (\ref{lzm75})  in Theorem \ref{Mth1} is obtained in the following Subsection as
\begin{align*}
f(t, X)=\widetilde f(\overline \tau , X)=\frac{\beta (\overline \tau )\varphi  _0(X)+\overline g(\overline \tau , X)}{X},\,\,\forall \overline \tau \in (0, T^*),\,\forall X>0
\end{align*}
for some $T^*>0$, where  $t$ and $\overline \tau $ are related by (\ref{S1E09}) and $\overline g$ is obtained in the following Subsection through a fixed point argument applied to an operator defined as follows.\\
For any given  function $g\in B _{ T, X _{ -r, r+q } }(0,  K)$ we obtain a function $\tilde \beta \in B _{ T}(1, 1/4)$ such that
\begin{align}
&\tilde \beta (t)=\mathscr W(t, G; X)+a(t)-\int _0^t \Lambda(t-s, Q_N(\varphi ))\tilde\beta (s)ds,\label{pnyx}\\
&G(t)=\Big(\tilde\beta  (t) Q_N(\varphi ) +\tilde \beta  ^{-1}(t) Q_N(g(t)) \Big)\label{mnst3}
\end{align}
For that function $\tilde \beta $, the two functions   $\widetilde g_1$, $\widetilde g_2$ (introduced respectively in (\ref{tildeg1}), (\ref{lngrsg2})) satisfy (by Proposition \ref{tildeg1} and Proposition \ref{mnst2}),
\begin{align}
\widetilde g_1(t, X)+\widetilde g_2(t, X)&=\left(\mathscr W(t, G; X)- \tilde \beta (t)+a(t)-\int _0^t\Lambda(t-s, Q_N(\varphi ) )\tilde \beta (s)ds\right)\times \nonumber  \\
&\hskip 4.5cm \times \eta(t, X)+r_0(t, G, X))+r_3(t, X)\label{mnst4} \\
&=r_0(t, G, X))+r_3(t, X)\label{mnst4b}
\end{align}
The map, $g\mapsto \tilde \beta $ is proved to be a contraction from $B _{ T, X _{ -r, r+q } }(0,  K)$ to $B _{ T}(1, 1/4)$.\\
 We then denote $\tilde g=\widetilde g_1+\widetilde g_2 $ and define the operator $\mathcal T$ as $\mathcal T(g)=\widetilde g$. A fixed point argument is  applied to the operator $\mathcal T$ in order to obtain a function $\overline g$ such that  $T(\overline g)=\overline g$.\\
  For  the last step of the proof  one must recover the notation $\overline \tau$ introduced in (\ref{S1E09}) and then denote as $\overline{g}(\overline\tau )$ the function obtained through the fixed point argument on $\mathcal T$.\\
The   $ f(t)$ local classical solution $ f(t)$  of  (\ref{lzm75})  in Theorem \ref{Mth1},  is then obtained, first defining
\begin{align*}
\widetilde f(\overline \tau , X)=\frac{\beta (\overline \tau )\varphi  _0(X)+\overline g(\overline \tau , X)}{X},\,\,\forall \overline \tau \in (0, T^*),\,\forall X>0
\end{align*}
for some $T^*>0$ and then inverting the change of time variable (\ref{S1E09}).

\subsection{Existence of the function $\overline g$.}
\label{fpa}
In this Subsection we  prove the existence of the function $\overline g $. To this end, as described in Subsection \ref{overvw} we first obtain the following.
\begin{prop} 
\label{pnyxB}
Given $K>0$, $r\in [0, 1/2)$, $0\le q<3$,

(i) There exists $T_1$ sufficiently small  such that  there exists $\beta $ that satisfies (\ref{pnyx}) for $t\in (0, T_1)$,  for $g$ such that $||g|| _{ -r, r+q }\le K$.

(ii)  There exists $T_2$ such that the map $g\to \beta $ is a contraction from  $B _{T, X _{ -r, r+q } }(0,  K)$ into $C([0, T])$ for $T\le T_2$.
\end{prop}
\begin{proof}
The function $\beta $ is obtained as a fixed point in $C([0, T])$ of
\begin{align*}
\mathscr T (\beta )=\mathscr W(t, G  _{ \beta , g } )+a(t)-\int _0^t \Lambda(t-s, Q_N(\varphi ) )\beta (s)ds.
\end{align*}
First, for all $t\in [0, T]$,
\begin{align*}
\mathscr T (\beta )(t)-1=\mathscr W(t, G  _{ \beta , g } )+(a(t)-1)-\int _0^t \Lambda(t-s,  Q_N(\varphi ))|\, |\beta (s)|ds.
\end{align*}
By Proposition  \ref{WWW},  if $\theta\ge \frac {1} {2}-r$ and $\rho +\theta \le q+\frac {1} {2}$, then (\ref{FFGG1}) holds, and
\begin{align*}
|\mathscr W(t, G  _{ \beta , g })|\le C\sup _{ 0\le s\le t }||G  _{ \beta , g } (s)|| _{ \theta, \rho  }t^{1-2\theta}+
 C\sup _{ 0\le \sigma \le t } || G  _{ \beta , g } (s)|| _{\frac {1} {2}+ \theta, \rho-1  }t^{2-2\theta}\\
 \le  C\left( t^{1-2\theta}+ t^{2-2\theta}\right)\sup _{ 0\le s\le t }\left( \beta (s)+\beta ^{-1}(s)||g(s)||^2 _{ -r, r+q }\right)\\
 \le C\left( T^{1-2\theta}+ T^{2-2\theta}\right)\left( 1+\sup _{ 0\le s\le T }||g(s)||^2 _{ -r, r+q }\right).
\end{align*}
By definition of $a$ and Proposition \ref{WWW}
\begin{align*}
|a(t)-1|&\le C(\varphi )t+\int _0^t |\mathscr W(s, \mathscr L_\varphi \left( \mathscr L_\varphi (\varphi )\right))|ds\le C(\varphi )t+C||\Phi  || _{ -\frac {1} {2}, 2 }t^3
\end{align*}

By  (\ref{mntrvll64}),
\begin{align}
\int _0^t|\Lambda (t-s, Q_N(\varphi ))||\beta (s)|ds\le
C\sup _{ 0\le s\le t }|\beta (s)||||Q_N(\varphi )|| _{ 0, \frac {3} {2} }t.\label{lmns}
\end{align}

All this shows that for $T<1$,
\begin{align}
\sup _{ 0\le t\le T }\left|\mathscr T(\beta )(t)-1\right|\le   CT^{1-2\theta}\left( 1+\sup _{ 0\le s\le T }||g(s)|| ^2_{ -r, r+q }\right)+\\
+C(\varphi )T+ C||\Phi  || _{ -\frac {1} {2}, 2 }t^3+C\sup _{ 0\le s\le t }|\beta (s)||||Q_N(\varphi )|| _{ 0, \frac {3} {2} }T
\end{align}

If we denote,
\begin{align*}
B _{ T}(1, 1/4)=\left\{\beta \in C([0, T]);\,\,\sup _{ 0\le s\le T }|\beta (s)-1|<1/4 \right\}
\end{align*}
for $\theta\in (0, 1/2)$, and $T$ sufficiently small,
\begin{align*}
&CT^{1-2\theta}\left( 1+\sup _{ 0\le s\le T }||g(s)|| ^2_{ -r, r+q }\right)+\\
&+C(\varphi )T+ C||\Phi  || _{ -\frac {1} {2}, 2 }t^3+C\sup _{ 0\le s\le t }|\beta (s)||||Q_N(\varphi )|| _{ 0, \frac {3} {2} }T<1
\end{align*}
the map $\mathscr T$ sends then the ball $B_T(1, 1/4)$ into itself. 

On the other hand, since $\mathscr W$ is linear with respect to $G $, for $\beta_1$ and $\beta _2$ in  the ball $B(1, 1/4)$  and $t\in (0, T)$,
\begin{align}
\label{lngrs}
|\mathscr T(\beta _1-\beta _2)(t)|\le |\mathscr W(t, G _{ \beta _1, g }-G _{ \beta _2, g })|+
\int _0^t |\Lambda (t-s, \Phi)| |\beta_1 (s)-\beta _2(s)|ds.
\end{align}
By Proposition \ref{WWW},
\begin{align*}
|\mathscr W(t, G _{ \beta _1, g }-G _{ \beta _2, g }; X)|\le C\sup _{ 0\le s\le t }||(G _{ \beta _1, g }-G _{ \beta _2, g })(s)|| _{ \theta, \rho  }t^{1-2\theta}+\\+
 C\sup _{ 0\le \sigma \le t } || (G _{ \beta _1, g }-G _{ \beta _2, g })(s)|| _{\frac {1} {2}+ \theta, \rho-1  }t^{2-2\theta}.
\end{align*}
By definition of $G$, if  $\theta\ge \frac {1} {2}-r$ and $\rho +\theta \le q+\frac {1} {2}$, 
\begin{align*}
&||(G _{ \beta _1, g }-G _{ \beta _2, g })(s)|| _{ \theta, \rho  }\le |\beta _1(s)-\beta _2(s)|||Q_N(\varphi )||_{ \theta, \rho  }+\\
&\hskip 5cm +\frac {|\beta _1(s)-\beta _2(s)|} {\beta _1(s)\beta _2(s)}||Q_N(g(s))||_{ \theta, \rho  }\\
&\le  |\beta _1(s)-\beta _2(s)|||Q_N(\varphi )||_{ \theta, \rho  }+C\frac {|\beta _1(s)-\beta _2(s)|} {\beta _1(s)\beta _2(s)}||g(s)||^2 _{ -r, r+q }\\
&\le  |\beta _1(s)-\beta _2(s)|||Q_N(\varphi )||_{ \theta, \rho  }+C|\beta _1(s)-\beta _2(s)|\,||g(s)||^2 _{ -r, r+q }
\end{align*}
and
\begin{align*}
&||(G _{ \beta _1, g }-G _{ \beta _2, g })(s)||   _{\frac {1} {2}+ \theta, \rho-1  }\le |\beta _1(s)-\beta _2(s)|||Q_N(\varphi )||  _{\frac {1} {2}+ \theta, \rho-1  }+\\
&\hskip 5cm +\frac {|\beta _1(s)-\beta _2(s)|} {\beta _1(s)\beta _2(s)}||Q_N(g(s))||  _{\frac {1} {2}+ \theta, \rho-1  }\\
&\le  |\beta _1(s)-\beta _2(s)|||Q_N(\varphi )||  _{\frac {1} {2}+ \theta, \rho-1  }+C\frac {|\beta _1(s)-\beta _2(s)|} {\beta _1(s)\beta _2(s)}||g(s)||^2 _{ -r, r+q }\\
&\le  |\beta _1(s)-\beta _2(s)|||Q_N(\varphi )||  _{\frac {1} {2}+ \theta, \rho-1  }+C |\beta _1(s)-\beta _2(s)|\,||g(s)||^2 _{ -r, r+q }.
\end{align*}
It follows,
\begin{align}
|\mathscr W(t, G _{ \beta _1, g }-&G _{ \beta _2, g })|\le  C\sup _{ 0\le s\le t }|\beta _1(s)-\beta _2(s)|T^{1-2\theta}\times \nonumber \\
&\times \Bigg( ||Q_N(\varphi )|| _{ \theta, \rho  }+
 ||Q_N(\varphi )|| _{ \frac {1} {2}+\theta, \rho -1 }+\sup _{ 0\le s\le T }||g(s)||^2 _{ -r, r+q } \label{mntrvll}
\Bigg).
\end{align}
By (\ref{mntrvll64}), the second term in the right hand side of (\ref{lngrs} is estimated using  (\ref{mntrvll64}),
\begin{align}
\int _0^t |\Lambda (t-s, Q_N(\varphi ) )| |\beta_1 (s)-\beta _2(s)|ds& \le C\sup _{ 0\le s\le t }|\beta _1(s)-\beta _2(s)|
\int _0^t | \Lambda (t-s, Q_N(\varphi ) )|ds \nonumber\\
& \le C\sup _{ 0\le s\le t }|\beta _1(s)-\beta _2(s)|||Q_N(\varphi )|| _{ 0, \frac {3} {2} }T.\label{lrds}
\end{align}
Estimates (\ref{mntrvll}) and (\ref{lrds}) show that for $T$ sufficiently small the map $\mathscr T$ is a contraction from the ball $B_T(1, 1/4)$  into itself and has then a fixed point $\beta $. This shows point (i).

Suppose now that  $g\in B _{ T, X _{ -r, r+q } }(0, K)$ and $h\in B _{T,  X _{ -r, r+q } }(0,  K)$ and denote $\beta _g$ and $\beta _h$ the solutions of
\begin{align*}
\beta _g=\mathscr W(t, G _{ g, \beta _g })+a(t)-\int _0^t\Lambda(t-s, Q_N(\varphi ))\beta_g (s)ds\\
\beta _h=\mathscr W(t, G _{ h, \beta _h })+a(t)-\int _0^t  \Lambda(t-s, Q_N(\varphi ))\beta_h (s)ds
\end{align*}
and then,
\begin{align}
\label{drn}
\beta _g-\beta _h=\mathscr W(t, G _{ g, \beta _g }-G _{ h, \beta _h })-
\int _0^t\Lambda(t-s, Q_N(\varphi ))(\beta_g-\beta _h) (s)ds.
\end{align}
The second term in the right hand side  of (\ref{drn}) is easily bounded using again (\ref{mntrvll64}),
\begin{align}
\label{lzm1}
\left|\int _0^t\Lambda(t-s, Q_N(\varphi ))(\beta_g-\beta _h) (s)ds\right|\le \sup _{ 0\le s\le t }|(\beta_g-\beta _h)(s)|||Q_N(\varphi )|| _{ 0, \frac {3} {2} }T
\end{align}

On the other hand, by Proposition \ref{WWW}, for $\theta' \in [0, 1/2)$, $\theta'' >-1$, $\rho' >1$ and $\rho'' >2$, the first term  in the right hand side  of (\ref{drn})  
\begin{align}
&|\mathscr W(t, G _{ g, \beta _g }-G _{ h, \beta _h })|\le  C\sup _{ 0\le s\le t }||(G _{ g, \beta _g }-G _{ h, \beta _h })(s)|| _{ \theta', \rho'  }t^{1-2\theta'}+\nonumber\\
&\hskip 4cm + C\sup _{ 0\le \sigma \le t } || (G _{ g, \beta _g }-G _{ h, \beta _h })(s)|| _{\frac {1} {2}+ \theta''\!,\,  \rho''-1  }t^{2-2\theta''},\label{lmns1}
\end{align}
In order to bound the right hand side of (\ref{lmns1}) we write,
\begin{align}
&|(G _{ g, \beta _g }-G _{ h, \beta _h })(s)|\le |(\beta _g-\beta _h)(s)||Q_N(\varphi )|+\nonumber \\
& +|\beta _g^{-1}(s)-\beta _h^{-1}(s)||Q_N(g(s))|+ \beta _h^{-1}|Q_n(g(s))-Q_n(h(s))|.\label{lmns}
\end{align} 
Then, by Proposition 7.15 of \cite{m3}, if $\theta'=\frac {1} {2}-r$, $\rho '=r+q$,
\begin{align}
&||(\beta _g-\beta _h)(s)Q_N(\varphi )|| _{ \theta', \rho ' }\le  |(\beta _g-\beta _h)(s)| ||Q_N(\varphi )|| _{ \theta', \rho ' }\le C |(\beta _g-\beta _h)(s)|\label{lzm2}\\
&||(\beta _g^{-1}(s)-\beta _h^{-1}(s))Q_N(g(s))|| _{ \theta', \rho ' }\le C |(\beta _g-\beta _h)(s)| ||Q_N(g(s))|| _{ \theta', \rho ' }\nonumber\\
&\hskip 5.7cm \le C |(\beta _g-\beta _h)(s)||g(s)|| ^2_{ -r, r+q }.\label{lzm3}\
\end{align}
 By Proposition 7.15 of \cite{m3}, the same choices of $\theta'$   and $\theta''$ give,
\begin{align}
||Q_n(g(s))-Q_n(h(s))|| _{\theta', \rho ' }\le C||(g-h)(s)|| _{ -r, r+q }(||g(s)|| _{ -r, r+q }+||h(s)|| _{ -r, r+q }),\label{lzm4}
\end{align}
and therefore,
\begin{align}
\label{lzm5}
\sup _{ 0\le s\le t }||(G _{ g, \beta _g }-G _{ h, \beta _h })(s)|| _{ \theta', \rho'  }t^{1-2\theta'}\le
C t^{1-2\theta'}\Bigg(\sup _{ 0\le s\le t }|(\beta _g-\beta _h)(s)|\times \nonumber\\
\times \left(1+\sup _{ 0\le s\le t }||g(s)||^2 _{ -r, r+q }\right)
+\sup _{ 0\le s\le t }||g(s)-h(s)|| _{ -r, r+q }\times \nonumber \\
\hskip 4cm \times \sup _{ 0\le s\le t }\left( ||g(s))|| _{ -r, r+q }+||h(s)|| _{ -r, r+q }\right)\Bigg)
\end{align}
Similarly, if $\theta''=-r,\rho ''=1+r+q=1-\theta''+q$ with $0<r<1/2$,
\begin{align*}
&||(\beta _g-\beta _h)(s)Q_N(\varphi )|| _{\frac {1} {2}+ \theta'', \rho'' }\le  |(\beta _g-\beta _h)(s)| ||Q_N(\varphi )|| _{\frac {1} {2}+ \theta'', \rho'' }\le C |(\beta _g-\beta _h)(s)|\\
&||(\beta _g^{-1}(s)-\beta _h^{-1}(s))Q_N(g(s))||  _{\frac {1} {2}+ \theta'', \rho'' }\le C |(\beta _g-\beta _h)(s)| ||Q_N(g(s))|| _{\frac {1} {2}+ \theta'', \rho'' }\\
&\hskip 5.7cm \le C |(\beta _g-\beta _h)(s)||g(s)||^2 _{ -r, r+q }.
\end{align*}
and
\begin{align*}
||Q_n(g(s))-Q_n(h(s))|| _{\frac {1} {2}+\theta'', \rho ''-1 }\le C||(g-h)(s)|| _{ -r, r+q }(||g(s)|| _{ -r, r+q }+||h(s)|| _{ -r, r+q }).
\end{align*}
and then,
\begin{align}
\label{lzm6}
\sup _{ 0\le s\le t }||(G _{ g, \beta _g }-G _{ h, \beta _h })(s)|| _{ \theta'', \rho''  }t^{1-2\theta''}\le
C t^{1-2\theta''}\Bigg(\sup _{ 0\le s\le t }|(\beta _g-\beta _h)(s)|\times \nonumber\\
\times \left(1+\sup _{ 0\le s\le t }||g(s)||^2 _{ -r, r+q }\right)
+\sup _{ 0\le s\le t }||g(s)-h(s)|| _{ -r, r+q }\times \nonumber\\
\hskip 4cm \times \sup _{ 0\le s\le t }\left( ||g(s))|| _{ -r, r+q }+||h(s)|| _{ -r, r+q }\right)\Bigg).
\end{align}
It follows from (\ref{lmns1}), (\ref{lzm5}) and (\ref{lzm6}),

\begin{align}
\label{lzm7}
&|\mathscr W(t, G _{ g, \beta _g }-G _{ h, \beta _h })|\le C \left(t^{1-2\theta'}+t^{1-2\theta''}\right)\Bigg(\sup _{ 0\le s\le t }|(\beta _g-\beta _h)(s)|\times \nonumber\\
&\times \left(1+\sup _{ 0\le s\le t }||g(s)|| ^2_{ -r, r+q }\right)
+\sup _{ 0\le s\le t }||g(s)-h(s)|| _{ -r, r+q }\times \nonumber\\
&\hskip 4cm \times \sup _{ 0\le s\le t }\left( ||g(s))|| _{ -r, r+q }+||h(s)|| _{ -r, r+q }\right)\Bigg).
\end{align}
By (\ref{drn}), (\ref{lzm1}) and (\ref{lzm7}), for $t\in (0, T)$ small, so that $t^{1-2\theta''}<t^{1-2\theta'}$,
\begin{align}
\label{drn3}
\sup _{ 0\le s\le t }&|(\beta _g-\beta _h)(s)|\le   C T^{1-2\theta'}\Bigg(\sup _{ 0\le s\le t }|(\beta _g-\beta _h)(s)|\left(1+ K^2\right)\nonumber\\
&+R\sup _{ 0\le s\le t }||g(s)-h(s)|| _{ -r, r+q }\Bigg)+\sup _{ 0\le s\le t }|(\beta_g-\beta _h)(s)|
||Q_N(\varphi )|| _{ 0, \frac {3} {2} }T
\end{align}
If
\begin{align}
\label{pnyx2}
C T^{1-2\theta'}(1+ K^2)+||Q_N(\varphi )|| _{ 0, \frac {3} {2} }T<\frac {1} {2}
\end{align}
then,
\begin{align}
\label{drn3}
\sup _{ 0\le s\le t }|(\beta _g-\beta _h)(s)|\le  & C T^{1-2\theta'} K\sup _{ 0\le s\le t }||g(s)-h(s)|| _{ -r, r+q }
\end{align}
and  the point (ii) follows if $T_2$ satisfies  (\ref{pnyx2}) and $C T_2^{1-2\theta'} K<1$.
\end{proof}

It is now possible to prove the existence of the function $\overline g$ applying a fixed point argument to the operator $\mathcal T$, mentioned in Subsection \ref{overvw}),  defined as follows.
Define now the operator $\mathcal T$ (mentioned in Subsection \ref{overvw})  on the set  $B _{ T, X _{ -r, q+r } }(0, M)$ as follows:

(i) given $g\in B _{ T, X _{ -r, q+r } }(0, M)$ let $\tilde \beta $ the solution of equation (\ref{pnyx}) given by Proposition \ref{pnyxB}

(ii) define  then $\tilde g_1$  by (\ref{lngrsg1}) and  $\tilde g_2$ by (\ref{lngrsg2}) with $\beta $ replaced by $\tilde \beta $ in formulas  (\ref{lngrsg1}) and (\ref{lngrsg2}), 

(iii) define finally $\tilde g=\tilde g_1+\tilde g_2$ and $T(g)=\tilde g$.
 \begin{align*}
\mathcal T(g)=\widetilde g=\widetilde g_1+\widetilde g_2.
\end{align*}
Let us prove now that  the operator $\mathcal T$ has  a fixed point on  $B _{ T, X _{ -r, q+r } }(0, M)$ under suitable conditions on the parameters $M$ and $T$, and also on the parameter $R$ that appears in the function $\varphi $ (cf. (\ref{defvfiR}))
\begin{prop}
\label{InnD+1}
There exists $M^*>0$, $T^*>0$ and $R^*>0$ such that  the map  $\mathcal T$ has a fixed point in $B _{ T^*, X _{ -r, q+r } }(0, M^*)$.
\end{prop}
\begin{proof}
Suppose that $g\in B _{ T, X _{ -r, q+r } }(0, M)$ and $\beta \in B_T(1, 1/4)$) and consider the function
\begin{align*}
\mathcal T(g)=\widetilde g=\widetilde g_1+\widetilde g_2
\end{align*}
where $\tilde g_1$ is given by (\ref{lngrsg1}),  $\tilde g_2$ by (\ref{lngrsg2}),
and with  $\widetilde \beta $ solution to (\ref{pnyx}) given by Proposition \ref{pnyxB}. It follows
\begin{align*}
\widetilde g(t, X)=r_0(t, G, X)+r_3(t, X)
\end{align*}
We first claim that for $T$ small enough $\widetilde g(t)\in B _{ T, X _{ -r, q+r } }(0, M)$. By Proposition \ref{tildeg1}, for $\theta \ge \frac {1} {2}-r$ and $\rho <q+\frac {1} {2}-\theta$, there exists a constant $C>0$ such that, for $t \in (0, T)$, $T<1$ 
\begin{align*}
|r_0(t, G, X)|\le 
\begin{cases}
C\displaystyle{\sup _{ 0<s<t }} \left(|\beta (s)|||Q_N(\varphi )|| _{ \theta, \rho  }+|\beta ^{-1}(s)|||g(s)||^2_{ -r, r+q }\right)t^{\frac {1} {2}-\theta},\,X\in (0, 1)\\
C\displaystyle{\sup _{ 0<s<t }} \left(|\beta (s)|||Q_N(\varphi )|| _{ \theta, \rho  }+|\beta ^{-1}(s)|||g(s)||^2_{ -r, r+q }\right)t^{4-2\theta}X^{-\frac {3} {2} },\,\,X>1,
\end{cases}
\end{align*}
from where, if moreover
\begin{align*}
\frac {1} {2}-\theta\ge r\,\,\text{and}\,\,
q<\frac {3} {2},
\end{align*}
\begin{align*}
&||r_0(t)|| _{ -r, q+r }\le \sup _{ 0<X<1 }|r_0(t, X)|X^{-r}+\sup _{ X>1 }|r_0(t, X)|X^{q}\\
&\le C\sup _{ 0<s<t }\left(|\beta (s)|\,||Q_N(\varphi )|| _{ \theta, \rho  }+|\beta ^{-1}(s)|\,||g(s)||^2_{ -r, r+q }\right)\sup _{ 0<X<1 }X^{\frac {1} {2}-\theta-r}+\\
&+C\sup _{ 0<s<t } \left(|\beta (s)|||Q_N(\varphi )|| _{ \theta, \rho  }+|\beta ^{-1}(s)|\,||g(s)||^2_{ -r, r+q }\right)t^{4-2\theta}\sup _{ X>1 }X^{-\frac {3} {2} }X^{q}\\
&\le C\sup _{ 0<s<t }\left(|\beta (s)|\,|||Q_N(\varphi )|| _{ \theta, \rho  }+|\beta ^{-1}(s)|\,||g(s)||^2_{ -r, r+q }\right)\left(1+t^{4-2\theta}\right)\\
&\le  C\left(\left(R^{-1}+R^{-\frac {1} {2}+\theta+\rho }\right)+M^2\right)\left(1+T^{4-2\theta}\right).
\end{align*}
On the other hand,   if $0<r<\frac {1} {2}$ and $0<q< \frac {3} {2}$ and using Proposition \ref{QQ0},
\begin{align*}
||r_3(t)|| _{ -r, q+r } & \le C\sup _{ 0<X<1 } X^{\frac {1} {2}-r}\left(||Q_N(\varphi )|| _{ -\frac {1} {2}, \rho  }
\left(t+\sup _{ 0\le s\le T }|\beta (s)|\right)+Ct||\Phi || _{ 0, \frac {3} {2} }\right)\\
&+C\sup _{ X>1 } X^{-3/2+q}\left(||Q_N(\varphi )|| _{ -\frac {1} {2}, \rho  }
\left(t+t\sup _{ 0\le s\le T }|\beta (s)|\right)+Ct^5||\Phi || _{ 0, \frac {3} {2} }\right)\\
& \le C\left(||Q_N(\varphi )|| _{ -\frac {1} {2}, \rho  }
\left(t+ \frac {5} {4}\right)+Ct||\Phi || _{ 0, \frac {3} {2} }\right)\\
& \le C\left(\left(R^{-1}+R^{-1+\rho }\right)
\left(t+ \frac {5} {4}\right)+||\Phi || _{ -\frac {1} {2}, 2 } t R^{-\frac {1} {2} }\right).
\end{align*}
Therefore, if  $r\in (0, 1/2)$, $\frac {1} {2}-\theta= r$,  $0<\rho <r$, $0<q<3/2$, then
\begin{align*}
||r_0(t)|| _{ -r, q+r }+||r_3(t)|| _{ -r, q+r }&\le C\left(R^{-1}+R^{-\frac {1} {2}+\theta+\rho } +R^{-1+\rho }+M^2+t||\Phi || _{ 0, \frac {3} {2} }\right) \\
&\le C\left(R^{-r+\rho }+M^2+||\Phi || _{ -\frac {1} {2}, 2 } t R^{-\frac {1} {2} } \right).
\end{align*}
It follows that for $R^* $ sufficiently large and $T^* $ and $M^* $ sufficiently  small, $r_0+r_3 \in B _{ T^*, X _{ -r, q+r } }(0, M^*)$. By the contractivity of the map $g\to \beta $, proved in (ii) of Proposition \ref{pnyxB}, and the Lipschitz property of the maps  $(g, \beta )\to r_0(\cdot, g, \cdot)$  and $\beta \to \widetilde g_2$ proved respectively  in Proposition \ref{WWW2}  and  Proposition \ref{mnst2},
if $T$ is small enough the map $\mathcal T$ is a contraction  and  has then a fixed point $g\in B _{ T^*, X _{ -r, q+r } }(0, M^*)$.

We  now  prove  that  $g$ satisfies (\ref{Mth1ELD}) and (\ref{l2e1E2}).
Let us show tfirst  that $\beta \in W^{1, \infty}(0, T^*)$, using the expression
\begin{align*}
&\beta(t) =\mathscr W(t, G  _{ \beta , g } )+a(t)-\int _0^t \Lambda(t-s, Q_N(\varphi ) )\beta (s)ds\\
&\Lambda(t-s; Q_N(\varphi ))=L(t-s; Q_N(\varphi ))+ \int _0^{t-s}L(t-s-\sigma ; Q_N(\varphi ))d\sigma 
\end{align*}
where $\Lambda(t; v)$ is defined in  (\ref{Ldef}) and $L(t, v)$ in Proposition \ref{S7cor1}
We have first,
\begin{align*}
\frac {d} {dt}\int _0^t \Lambda(t-s, Q_N(\varphi ) )\beta (s)ds&=\Lambda(0, Q_N(\varphi ) )\beta (t)-
\int _0^{t }\frac {d} {dt}\Lambda(t-s, Q_N(\varphi ) )\beta (s)\\
&=L(0; Q_N(\varphi ))-\int _0^{t }\frac {d} {dt}\Lambda(t-s, Q_N(\varphi ) )\beta (s)
\end{align*}
and by the definition of $L(t-s; Q_N(\varphi ))$ in Proposition \ref{S7cor1},
\begin{align*}
L(0, Q_N(\varphi ) )=\lim _{ t\to 0 } \frac {6 B(1)}{2i\pi ^3}\int  _{ \mathscr Re(r )=\beta  }\frac {\Gamma(r )} {B(r )}t^{-r }
\left(\int _0^{t} Q_N(\varphi )(\zeta^2 )\zeta ^{-1+ r }d\zeta  \right)dr=0.
\end{align*}
On the other hand,
\begin{align}
\frac {d} {dt} \Lambda(t-s, Q_N(\varphi ) )=\frac {d} {dt}L(t-s; Q_N(\varphi ))+L(0, Q_N(\varphi ))+ \nonumber\\
+ \int _0^{t-s}\frac {d} {dt}L(t-s-\sigma ; Q_N(\varphi ))d\sigma \label{S8EBder}
\end{align}
where
\begin{align}
&\frac {d} {dt}L(t-s; Q_N(\varphi ))=- \frac {6 B(1)r} {2i\pi ^3}\int  _{ \mathscr Re(r )=\beta  }\frac {\Gamma(r )} {B(r )}(t-s)^{-r-1 }\times \nonumber \\
&\hskip 7cm \times \left(\int _0^{t-s} Q_N(\varphi )(\zeta^2 )\zeta ^{-1+ r }d\zeta  \right)dr+\nonumber\\
&\hskip 4cm + \frac {6 B(1)} {2i\pi ^3}\int  _{ \mathscr Re(r )=\beta  }\frac {\Gamma(r )} {B(r )}(t-s)^{-1 } \left( Q_N(\varphi )((t-s)^2 ) \right)dr
\label{S8EBder1}
\end{align}
where, by Proposition \ref{QQ0}, $|Q_N(\varphi )(\zeta ^2|\le C \zeta $ and the integrals in right hand side of (\ref{S8EBder1}) are uniformly convergent. Moreover
\begin{align*}
\left| \frac {d} {dt}L(t-s-\sigma ; Q_N(\varphi ))\right|\le  C\int  _{ \mathscr Re(r )=\beta  }\frac {|\Gamma(r )|} {|B(r )|}|dr|<\infty.
\end{align*}
and the second integral in the right hand side of (\ref{S8EBder}) is uniformly convergent and bounded on $[0, T^*]$. This and classical measure theory arguments  show that
$\int _0^t \Lambda(t-s, Q_N(\varphi ) )\beta (s)ds\in W^{1, \infty}(0, T)$. The same method shows $\mathscr W(t, G  _{ \beta , g } )\in W^{1, \infty}(0, T)$ and since  the function $a$ defined in (\ref{S7Da}) satisfies  (\ref{ap}) we conclude $\beta \in W^{1, \infty}(0, T)$.\\
The function $g$ satisfies, $g=\mathcal T(g)=g_1+g_2$where  the function $g_2$,  by (\ref{lngrsg2}] satisfies 
\begin{align*}
g_2(t, X)=- \beta (t)\, \varphi(X)+\psi (t, X)-\int _0^tw(t-s, X)\beta (s)ds.
\end{align*}
Taking times derivatives yields,
\begin{align*}
{g_2}_t(t, X)&=- \beta' (t)\, \varphi(X)+\psi_t (t, X)-w(0, X)\beta (t)-\int _0^tw_t(t-s, X)\beta (s)ds\\
&=- \beta' (t)\, \varphi(X)+\mathscr L_\varphi (\psi (t))(X)-\mathscr L_\varphi (\varphi )\beta (t)-\int _0^t \mathscr L_\varphi (w(t-s))( X)\beta (s)ds\\
&=- \beta' (t)\, \varphi(X)+\mathscr L_\varphi (\psi (t))(X)-\mathscr L_\varphi \left(\beta (t)\varphi +\int _0^t (w(t-s))( X)\beta (s)ds\right)\\
&=-\beta' (t)\, \varphi(X)+\mathscr L_\varphi (\psi (t))(X)+\mathscr L_\varphi (g_2(t, X)-\psi (t, X))\\
&=-\beta' (t)\, \varphi(X)+\mathscr L_\varphi (g_2(t, X)).
\end{align*}
It follows that, $g=g_1+g_2$ satisfies
\begin{align*}
g(t, X) = \int _0^t S _{ \varphi  }(t-s)\left(\beta  (s) Q_N(\varphi )+\beta  ^{-1}(s) Q_N( g(t) )-\beta '(s)\varphi\right)ds 
\end{align*}
Since $\beta  Q_N(\varphi )+\beta  ^{-1} Q_N(g)-\beta '\varphi \in X _{ \theta, \rho  }$, it follows from Proposition \ref{S7cor2}, (i) that  $\left|\frac {\partial \overline g } {\partial \overline \tau} \right|+\left| \mathscr L(\overline g )\right|\in L^\infty _{ \text{loc} }((0, \infty)\times (0, T^*))$.
Therefore $g$  satisfies  
\begin{align}
\frac {\partial  g (t)} {\partial t } = \mathscr L_\varphi (  g(t ) )+ \beta  (t) Q_N(\varphi )+\beta  ^{-1}(t) Q_N( g(t) )-\beta '(t)\varphi  \label{mm3tb}
\end{align}
pointwise almost everywhere in $(0, T^*)\times (0, \infty)$ and by (\ref{Mth1ELD}), equation (\ref{mm3tb}) is  satisfied in  $L _{ \text{loc} }^\infty([0, T^{**}); L^1 _{ \text{loc} }([0, \infty))$ since all its  terms  belong to that space.

\subsection{End of the Proof of Theorem \ref{Mth1}.}
\label{endp}
In order to finish  the proof of Theorem \ref{Mth1}  the notation $\overline g$, and $\overline \tau $ for the time variable  introduced in (\ref{S1E09}) must be recovered.  (In our abuse of notation they have been denoted $g$ and  $t$ up to now). 
Of course, the function $\overline g$ still satisfies (\ref{Mth1ELD}) and (\ref{Mth1ELD0})--(\ref{Mth1ELD}) and solves  in  $L _{ \text{loc} }^\infty([0, T^{**}); L^1 _{ \text{loc} }([0, \infty); dX)$ since all its  terms  belong to that space. 

Define now, 
\begin{align*}
\widetilde f(\overline \tau , X)=\frac{\beta (\overline \tau )\varphi  _0(X)+\overline g(\overline \tau , X)}{X},\,\,\forall \overline \tau \in (0, T^*),
\end{align*}
where $\overline g\in B _{ T^*, X _{ -r, q+r } }(0, M^*) $ and $\beta\in  B_T^*(1, 1/4) $ are the functions given by Proposition \ref{InnD+1}. 
Since $\overline g=\mathcal T(\overline g)=\overline g_1+\overline g_2$.
By definition,
\begin{align}
\widetilde f(\overline \tau , X)&=X^{-1}\Big(\beta (\overline \tau )\varphi (X)+\overline g_1(\overline \tau , X)+\overline g_2(\tau , X)\Big)\nonumber\\
&=X^{-1}\Big( \overline g_1(\overline \tau , X) +\psi (\overline \tau,  X)-\int _0^{\overline \tau }w(\overline \tau -s, X)\beta (s)ds\Big)
\label{foverlp}
\end{align}
where $\psi $ is as in (\ref{defpsi}), $w$ as in (\ref{defw}) and $\overline g_1$ as in  Proposition \ref{tildeg1}. Now,

\noindent
1.- \underline{The function $ \widetilde f(\overline \tau , X)$ is differentiable with respect to $\overline \tau $.} Because so are all the  terms  at the right hand side of  (\ref{foverlp}):  $ \overline g_1$ from Proposition \ref{Prop7.1}, $\psi $ from  (\ref{ap}), (\ref{r2p}), (\ref{defa}) and $w$ from the properties of the function $W$ in (\ref{defW}). \\

\noindent
2.- \underline{The function $\widetilde f$ satisfies:} $\beta({ \overline \tau }) \widetilde f_{ \overline \tau }({ \overline \tau }, X)=X^{-1}Q_N(\widetilde f({ \overline \tau }))(X)$ a.e. in $\overline \tau \in (0, T^*), X>0$. By construction and using (\ref{defw1}), (\ref{defw2}), for almost every $\overline \tau \in (0, T^*), X>0$,
\begin{align*}
X\widetilde f_{\overline \tau} (\overline \tau , X)&=(\overline g_1)_{\overline \tau}  +\psi_{\overline \tau}  (\overline \tau , X)-\left(\int _0^tw(t-s, X)\beta (s)ds\right) _{ \overline \tau }\\
&=(\overline g_1)_{ \overline \tau } +\psi_{ \overline \tau } ({ \overline \tau }, X)-\beta ({ \overline \tau })\mathscr L_\varphi (\varphi )-\mathscr L_\varphi \left(\int _0^{ \overline \tau }w({ \overline \tau }-s, X)\beta (s)ds \right)\\
&=(\overline g_1)_{ \overline \tau } +\psi_{ \overline \tau } ({ \overline \tau }, X)-\beta ({ \overline \tau })\mathscr L_\varphi (\varphi )-\mathscr L_\varphi \left(\overline g_1 +\psi ({ \overline \tau }, X)-X\widetilde f({ \overline \tau })\right)\\
&=(\widetilde g_1)_{ \overline \tau } -\mathscr L_\varphi \left(\overline g_1 \right)+\psi_{ \overline \tau } ({ \overline \tau }, X)-\mathscr L_\varphi \left( \psi ({ \overline \tau }, X)\right) -\beta ({ \overline \tau })\mathscr L_\varphi (\varphi )+\mathscr L_\varphi \left(X\widetilde f({ \overline \tau })\right)
\end{align*}
By Proposition \ref{Prop7.1}, and (\ref{Epsi1}) (proved just below (\ref{defpsi})),
\begin{align*}
X\widetilde f_{\overline \tau} (\overline \tau , X)= \beta  ({ \overline \tau }) Q_N(\varphi )+\beta  ^{-1}({ \overline \tau }) Q_N(\widetilde g({ \overline \tau })) -\beta ({ \overline \tau })\mathscr L_\varphi (\varphi )+\mathscr L_\varphi \left(X\widetilde f({ \overline \tau })\right).
\end{align*}
But, as shown by (\ref{S1InnD6}) and  (\ref{S1E6Y1}), (\ref{S1E6Y2}),
\begin{align}
Q_N(\widetilde f(\overline \tau ))=\beta \mathscr L_\varphi (\overline g)+Q_N(\overline g)+\beta ^2Q_N(\varphi ) \label{S1E6Y4}
\end{align}
and then,
\begin{align*}
&\beta  ({ \overline \tau }) Q_N(\varphi )({ \overline \tau })+\beta  ^{-1}({ \overline \tau }) Q_N(\overline g({ \overline \tau })) -\beta ({ \overline \tau })\mathscr L_\varphi (\varphi )+\mathscr L_\varphi \left(X\widetilde f({ \overline \tau })\right)=\\
&\beta  ({ \overline \tau }) Q_N(\varphi )({ \overline \tau })-\mathscr L_\varphi (\overline g)+\beta  ^{-1}({ \overline \tau }) Q_N(\widetilde f({ \overline \tau }))-\beta ({ \overline \tau })Q_N(\varphi )-\beta ({ \overline \tau })\mathscr L_\varphi (\varphi )+\mathscr L_\varphi \left(X\widetilde f({ \overline \tau })\right)
\end{align*}
Since $\mathscr L_\varphi  (\overline g)+\beta \mathscr L(\varphi )=\mathscr L(X\widetilde f)$, it follows
\begin{align*}
\beta({ \overline \tau }) \frac {\partial \widetilde f} {\partial \overline \tau}({ \overline \tau }, X)=X^{-1}Q_N(\widetilde f({ \overline \tau }))(X),\,\,a. e. \, (0, T^*)\times (0, \infty).
\end{align*}

3.- Inverse now the change of variables in (\ref{S1E09}) and, using  the time variable $t$ such that $dt=\beta ^{-1}( \overline \tau)d\overline \tau $, define the functions 
\begin{align}
&t=\int _0^{\overline{\tau }}\frac {d\sigma } {\beta (\sigma )},\,\,\,T^{**}=\int _0^{T^*}\frac {d\sigma } {\beta (\sigma )}\\
&f(t, X)=\widetilde f({ \overline \tau }, X),\,\,\forall t\in (0, T^{**})\label{deff}\\
&\lambda (t)=\beta (\overline \tau),\,\,\,g(t, X)=\overline g(\overline \tau, X),\,\,\forall t\in (0, T^{**})\label{deflam}
\end{align}
Then $f, \lambda $ satisfy equation (\ref{lzm75}), in $L _{ \text{loc} }^\infty([0, T^{**}); L^1 _{ \text{loc} }([0, \infty);XdX)$.
 This ends the proof of Theorem \ref{Mth1}.
\end{proof}
\section{Proof of Theorem \ref{Mth2.0}}
\label{SMth2.0}
\setcounter{equation}{0}
\setcounter{theo}{0}
Theorem \ref{Mth2.0} will follow from Theorem \ref{Mth1} if the function $n$ may be defined and the change of time variable (\ref{S2NewTime}) may be inverted.
\begin{prop}
\label{Psphn}
Let  $f$ be the function defined in (\ref{deff}). Then for all $t\in (0, T^{**})$,
\begin{equation*}
\int _0^\infty \widetilde Q(f(t), f(t))(X)X^{1/2}dX=\lim _{ \delta \to 0 }\int _\delta ^\infty \widetilde Q(f(t), f(t))(X)X^{1/2}dX
=-\frac {\pi ^2} {3}
\end{equation*}
\end{prop}

\begin{proof}
We  closely follow   the argument, and some of the notations of Proposition 2 in  \cite{S}. 

Notice first that  $f(t)\in L^1(\delta , \infty)$ for all $t\in (0, T^{**})$ and for all $\delta >0$  by  the behavior of $\widetilde g(\overline \tau , X)$ as $X\to \infty$. Define then for $d>0$, 
\begin{align*}
f_<(t, X)=f(t, X)\1 _{ \{0<X<d \}};\,\,\,f_>(t, X)=f(t, X)\1 _{ \{X>d \}}
\end{align*}
and for $\delta >0$,
\begin{align*}
&A_\delta (h , \psi )=\int  _{ \delta  }^\infty  \widetilde Q(h, \psi )(X)\sqrt XdX=A^{(1)}_\delta (h , \psi )+A^{(2)}_\delta (h , \psi )\\
&A^{(1)}_\delta (h , \psi )=\int _\delta ^\infty \widetilde Q(h , \psi )(X)\1 _{ \{0<Y<X\} } \sqrt X dX\\
&A^{(2)}_\delta (h , \psi )=2\int _\delta ^\infty  \widetilde Q(h , \psi )(X)\1 _{ \{Y>X\} } \sqrt X dX.
\end{align*}
Of course, 
\begin{align*}
&A^{(1)}_\delta (h , \psi )=\int _\delta ^\infty dX\int _0^X dY\Bigg(
h (X-Y)\psi (Y)-h(X)\psi (X-Y)-h (X)\psi (Y)\Bigg)\\
&A^{(2)}_\delta (h , \psi )=2\int _\delta ^\infty dX\int _X^\infty dY\Bigg(
h (X)\psi (Y) +h (Y-X)\psi (Y)-h (X)\psi (Y-X)\Bigg).
\end{align*}
Then, 
\begin{align*}
A_\delta (f , f)= A_\delta (f_< , f_<)+A_\delta (f_< , f_>)+A_\delta (f_> , f_<)+A_\delta (f_> , f_>).
\end{align*}
and by the integrability properties of $f(t)$ one has 
\begin{align*}
&\int _0^\infty \left|\widetilde Q(f_>(t) , f_>(t) )(X)\right| \sqrt XdX<\infty\\
&\int_0^\infty \left|\widetilde Q(f_<(t) , f_>(t) )(X)\right| \sqrt XdX+\int_0^\infty \left|\widetilde Q(f_>(t) , f_<(t) )(X)\right| \sqrt XdX<\infty
\end{align*}
and 
\begin{align*}
&\lim _{ \delta \to 0 }A_\delta (f_>(t), f_>(t))=A_0 (f_>(t), f_>(t))=0\\
&\lim _{ \delta \to 0 }\left(A_\delta (f_<(t), f_>(t))+ A_\delta (f_>(t), f_<(t))\right)=\\
&\qquad =\left(A_ 0(f_<(t), f_>(t))+ A_0 (f_>(t), f_<(t))\right)=0.
\end{align*}
We are then left with $A_\delta (f_<(t), f_<(t))$, where by construction, when $d <R$,
\begin{align*}
f_<(t, X)&= \frac {\lambda (t)\varphi (X)}{X}\1 _{ \{0<X<d\} }+\frac{g(t, X) } {X}\1 _{ \{0<X<d\} }\\
&= \frac {\lambda (t) }{X}\1 _{ \{0<X<d\} }+\frac{g(t, X) } {X}\1 _{ \{0<X<d\} } \equiv \lambda (t)F_1(X)+h(t, X)
\end{align*}
where $F_1=X^{-1}$ as in the Introduction. The terms in $A_\delta (f_<(t), f_<(t))$ may then reorganized as follows,
\begin{align*}
A_\delta (f_<(t), f_<(t))=\lambda (t)^2A_\delta ( F_1, \,& F_1 )+\lambda (t)A_\delta ( F_1, h(t))+\\
&+\lambda (t)A_\delta ( h(t), F_1)+A_\delta (h(t), h(t)).
\end{align*}
An explicit computation gives now, for all $d<R$
\begin{align*}
\lim _{ \delta \to 0 }A^{(1)}_\delta ( F_1, \,& F_1 )=\frac {\pi ^2} {3};\,\,\,\,\lim _{ \delta \to 0 }A^{(2)}_\delta ( F_1, F_1 )=-2\frac {\pi ^2} {3}.
\end{align*}
It follows from the properties of $g$ that $h(t)\in L^1(0, \infty)$ and then $A_\delta (h(t), h(t))\to A_0 (h(t), h(t))=0$ as $\delta \to 0$.
The terms in the sum $ A_\delta ( F_1, h(t))+ A_\delta ( h(t), F_1)$ may be rearranged as follows,
\begin{align*}
 A_\delta ( F_1, h(t))+ A_\delta ( h(t), F_1)= \Big(A^{(1)}_\delta ( F_1, h(t))+ A^{(1)}_\delta ( h(t), F_1)\Big)+\\
 +\Big(A^{(1)}_\delta ( F_1, h(t))+ A^{(1)}_\delta ( h(t), F_1)\Big)
\end{align*}
and
\begin{align*}
A^{(1)}_\delta ( F_1, h(t))+ &A^{(1)}_\delta ( h(t), F_1)=\int _\delta ^\infty dX\int _0^X dY\Bigg( (h(t, X-Y)-h(t, X))F_1(Y)+\\
&+(h(t, Y)-h(t, X))F_1(X-Y)-(h(t, X-Y)+h(t, Y))F_1(X)
\Bigg)
\end{align*}
But since $F_1(X)=h(t, X)=0$ when $X>d$, then for all $\delta\in (0, d)$, as it may be seen in a simple picture,
\begin{align*}
&\left| A^{(1)}_\delta ( F_1, h(t))+ A^{(1)}_\delta ( h(t), F_1)\right| \le J_1+J_2\\
&J_1=\int _0 ^d dX \int _0^X dY\Theta (X,Y)=\iint  _{ D_1 }\Theta (X,Y)dX dY\\
&J_2=\int _d^{2d} dX \int _{X-d}^d dY\Theta(X, Y)=\iint  _{ D_2 }\Theta (X,Y)dX dY\\
&\Theta(X, Y)=\Bigg( |h(t, X-Y)-h(t, X)|F_1(Y)+\\
&+|h(t, Y)-h(t, X)|F_1(X-Y)-|h(t, X-Y)+h(t, Y)|F_1(X)\Bigg)
\end{align*}
where
\begin{align}
&D_1=\left\{(X, Y)\in \RR^2;\, 0<X<d,\,0<Y<X \right\}\label{S9D1}\\
&D_2=\left\{(X, Y)\in \RR^2;\, d<X<2d,\,X-d<Y<d \right\}.\label{S9D2}
\end{align}
By construction $h(t, X)\equiv g(t, X)/X$, $F_1(X)=X^{-1}$ on $D_1$ and moreover,
\begin{align*}
&g(t, X)=\overline g(\overline \tau , X)=\mathcal T(\overline g(\overline \tau ))(X)=\overline g_1(\overline \tau , X)+\overline g_1(\overline \tau , X)\\
&\overline g_1(\overline \tau , X)=\int _0^tS_\varphi (t-s)G _{ \beta, \overline g }(s)(X)ds\\
&\overline g_2(\overline \tau , X)=- \beta (t)\, \varphi(X)+\psi (t, X)-\int _0^tw(t-s, X)\beta (s)ds.
\end{align*}
Since $g\in L^\infty ((0, T^{**}); X _{ -r, r+q })$  it follows  $G _{ \beta , \overline g }\in L^\infty((0, T^{**}); X _{ \frac {1} {2}-r, r+q }$ and then 
$G _{ \beta , \overline g }\in L^\infty((0, T^{**}); X _{ \frac {1} {2}-r, r+\tilde q }$ for all $\tilde q\in (0, 1)$.
By Corollary \ref{cor9.3}, with $\theta=\frac {1} {2}-r$, $\rho =r+\tilde q$ where $\tilde q\in (0, 1)$, and $\theta'=\theta$, $\rho '=\rho $,
\begin{align}
\left| \frac {1} {X}\overline g_1(\overline \tau , X)-  \frac {1} {X'}\overline g_1(\overline \tau , X') \right| \le
C\sup _{ 0\le s\le \overline \tau }\Big(||G _{ \beta, \overline g }(s)|| _{ \theta', \rho ' }+\nonumber\\
+||G _{ \beta, \overline g }(s)|| _{ \theta, \rho }\Big)\Omega_ {2\theta}(X^{1/2}, X'^{1/2})\nonumber \\
=C\sup _{ 0\le s\le \overline \tau }||G _{ \beta, \overline g }(s)|| _{ \theta, \rho }\Omega_ {2\theta}(X^{1/2}, X'^{1/2}).\label{g1E1}
\end{align}

On the other hand,  since $\varphi (X)=\varphi (X')=1$ for all  $X\in (0,  2d)$ if $2d <R$,
\begin{align}
&\left| \frac {1} {X}\overline g_2(\overline \tau , X)-  \frac {1} {X'}\overline g_2(\overline \tau , X') \right| \le
\left| \frac {\psi (t, X)} {X}-\frac {\psi (t, X')} {X'} \right|+\nonumber\\
&+\left|\frac {1} {X}\int _0^tw(t-s, X)\beta (s)ds-\frac {1} {X'}\int _0^tw(t-s, X')\beta (s)ds \right|\label{w1}
\end{align}
By (\ref{WW1}) and (\ref{defpsi}), 
\begin{align*}
&\psi (t, X)=\varphi(X) +t \mathscr L_\varphi (\varphi )(X)+\int _0^t  \int _0^s\mathscr (S _{ \varphi  }(s-\sigma ) \Phi)(X) d\sigma ds
\end{align*}
where $  \mathscr L_\varphi (\varphi )=2Q_N(\varphi ) =0$ if $X<\alpha R$ and $\Phi $ is defined in (\ref{defFi}). Since $\Phi \in X _{ -\frac {1} {2}, 2 }\subset X _{ \frac {1} {2}-r, r+\tilde q }$ for $\tilde q\in (0, 1)$, by Corollary \ref{cor9.3}, again with $\theta=\frac {1} {2}-r$, $\rho =r+\tilde q$ where $\tilde q\in (0, 1)$, and $\theta'=\theta$, $\rho '=\rho $,
\begin{align}
&\left|\frac {1} {X}\int _0^t  \int _0^s\mathscr (S _{ \varphi  }(s-\sigma ) \Phi)(X) d\sigma ds
- \frac {1} {X'}\int _0^t  \int _0^s\mathscr (S _{ \varphi  }(s-\sigma ) \Phi)(X) d\sigma ds\right|\le \nonumber  \\
&\quad \le \int _0^t  \left|\frac {1} {X}\int _0^s\mathscr (S _{ \varphi  }(s-\sigma ) \Phi)(X) d\sigma 
- \frac {1} {X'}  \int _0^s\mathscr (S _{ \varphi  }(s-\sigma ) \Phi)(X) d\sigma \right|ds \nonumber\\
&\quad \le \left( ||\Phi || _{\theta, \rho  }+||\Phi || _{\theta', \rho'  } \right)\Omega_{2\theta} (X^{1/2}, X'^{1/2})
\le C ||\Phi || _{\theta, \rho  }\Omega_{2\theta} (X^{1/2}, X'^{1/2}). \label{g2E1}
\end{align}
By (\ref{defw1}),  since    $\mathscr L_\varphi (\varphi )=2Q_N(\varphi )\in  X _{ \frac {1} {2}-r, r+\tilde q }$ too for the same $\theta$ and $\tilde q$ as above, it follows using Corollary \ref{cor9.3} in  the last term at the right hand side of  (\ref{w1}),
\begin{align}
&\left|\frac {1} {X}\int _0^tw(t-s, X)\beta (s)ds-\frac {1} {X'}\int _0^tw(t-s, X')\beta (s)ds \right|=\nonumber\\
&=\left|\frac {1} {X}\int _0^t \mathscr S_\varphi (t-s)(\mathscr L_\varphi (\varphi ))(X)\beta (s)ds-
\frac {1} {X'}\int _0^t  \mathscr S_\varphi (t-s)(\mathscr L_\varphi (\varphi ))(X')\beta (s)ds \right| \nonumber\\
&\le C \left(\sup _{ 0\le s \le t }||\beta (s)\mathscr L_\varphi (\varphi )|| _{ \theta, \rho  }+
||\beta (s)\mathscr L_\varphi (\varphi )|| _{ \theta', \rho'  }\right)\Omega  _{ 2\theta }(X^{1/2}, X'^{1/2})\nonumber\\
&= C \sup _{ 0\le s \le t }||\beta (s)\mathscr L_\varphi (\varphi )|| _{ \theta, \rho  }\Omega  _{ 2\theta }(X^{1/2}, X'^{1/2}). \label{g2E2}
\end{align}
From (\ref{g1E1}), (\ref{g2E1}) and (\ref{g2E2}), 
\begin{align*}
\left| \frac {1} {X}  g(t , X)-  \frac {1} {X'} g(t , X') \right|& \le C\sup _{ 0\le s\le t }\Big(||G _{ \beta, \overline g }(s)|| _{ \theta, \rho }+\nonumber\\
&+ ||\Phi || _{\theta, \rho  }+||\beta (s)\mathscr L_\varphi (\varphi )|| _{ \theta, \rho  }
\Big)\Omega_ {2\theta}(X^{1/2}, X'^{1/2})
\end{align*}
and it follows from  Proposition \ref{P52} that $ \Theta(X, Y) \in L^1(D_1\cup D_2)$ and,

\begin{align*}
\lim _{d\to 0}\left| A^{(1)}_\delta ( F_1, h(t))+ A^{(1)}_\delta ( h(t), F_1)\right|=0.
\end{align*}
A similar argument shows,
\begin{align*}
\lim _{d\to 0}\left| A^{(2)}_\delta ( F_1, h(t))+ A^{(2)}_\delta ( h(t), F_1)\right|=0.
\end{align*}
\vskip -0.5cm 
\end{proof}
\subsection{End of the proof of Theorem \ref{Mth2.0}}
Let us define now the new time variable $\tau $,
\begin{align*}
\tau = \frac {3} {\pi ^2}\log \left(\frac {\pi ^2 t} {3} +1\right),\,\forall t\ge 0.
\end{align*}
Since $t\to \tau $ is a strictly increasing bijection between $(0, T^{**})$ and $(0,  T)$ with
\begin{align*}
 T=\frac {3} {\pi ^2}\log \left(\frac {\pi ^2 T^{**}} {3} +1\right)
\end{align*}
define also,
\begin{align*}
&n(\tau )=e^{\frac {\pi ^2 \tau } {3}},\forall \tau >0,\\
&F(\tau , X)=f(t, X),\,\,\forall \tau \in (0,  T),\\
&\text{where},\,\, t=\frac {3} {\pi ^2}\left( e^{\frac {\pi ^2 \tau } {3}}-1\right).
\end{align*}
Then, for $X>0$ and $\tau \in (0,  T)$,
\begin{align*}
\frac {\partial F(\tau, X )} {\partial \tau }&=\frac {\partial f(t, X)} {\partial t}\frac {dt} {d\tau }= n(\tau )\widetilde Q_N(F(\tau ), F(\tau ))(X)
\end{align*}
and by Proposition \ref{Psphn}, 
\begin{align*}
\frac {n'(\tau )} {n(\tau)}=\frac {\pi ^2} {3}=-\int _0^\infty \widetilde Q(F(\tau ), F(\tau ))(X)X^{1/2}dX.
\end{align*}
Since $g(t)\in X _{ -r, r+q }$ for all $t\in T^{**}$, property (\ref{Mth2.0E1}) is fulfilled. 

Property (\ref{Mth2.0E2}) follows from the very definition of $n(\tau )$. As for property (\ref{Mth2.0E3}), let us prove first that the function
\begin{align*}
\Phi (X, Y)=X\Big(|F(|X-Y|)|(|F(Y)|+|F(X)|)+ |F(X)||F(Y)|)\Big)\1 _{ Y>X/2 }
\end{align*}
belongs to $\Phi \in L^1((0, \infty)\times (0, \infty))$. To this end, we first write,

\begin{align*}
&\int _0^\infty \int _0^\infty \Phi (X, Y) dYdX=I_1+I_2\\
&I_1=\int _0^1 \int  _{ X/2 }^\infty  \Phi (X, Y) dYdX,\,\,\,I_2=\int _1^\infty \int  _{ X/2 }^\infty  \Phi (X, Y) dYdX.
\end{align*}
In the first term $I_1$,
\begin{align*}
I_1\le\int _0^1 \int  _{ X/2 }^2  \Phi (X, Y) dYdX+\int _0^1 \int _2^\infty \Phi (X, Y)dYdX=I _{ 1,1 }+I _{ 1,2 }
\end{align*}
Since in $I _{ 1, 1 }$,
\begin{align*}
|F(Y)|\le CX^{r-1},\,\,|F(|X-Y|)|\le C|X-Y|^{r-1},\,\,\,|F(X)|\le CX^{r-1},
\end{align*}
it follows,
\begin{align*}
I _{ 1,1 }\le 
\int _0^1\int _{ \frac {X} {2} }^2X\Big(|F(|X-Y|)||F(Y)
-F(X)|+|F(X)||F(Y)|\Big)dYdX<\infty
\end{align*}
and the same conclusion holds for $I _{ 1, 2 }$ since in that integral,
\begin{align*}
|F(Y)|\le CY^{-1-q},\,\,|F(|X-Y|)|\le C|Y|^{-1-q},\,\,\,|F(X)|\le CX^{r-1}.
\end{align*}
In the term $I_2$, with a similar argument
\begin{align*}
I_2\le \int _1^\infty \int _{ \frac {X} {2} }^{2X}\Phi (X, Y)dYdX+\int _1^\infty X^{3/2}\int  _{ 2X }^\infty\Phi (X, Y)dYdX=I _{ 2, 1 }+I _{ 2,2 }
\end{align*}
In the  integral $I _{ 2, 1 }$,
\begin{align*}
|F(Y)|\le CY^{-1-q},\,\,|F(|X-Y|)|\le CY^{r-1},\,\,\,|F(X)|\le CX^{-1-q}.
\end{align*}
from where,
\begin{align*}
I _{ 2, 1 }\le C\int _1^\infty  X\int _{ \frac {X} {2} }^{2X} 
\Big(Y^{-1+r} (Y^{-q-1}+ X^{-q-1})+\\
+X^{-q-1}Y^{-q-1}\Big)dYdX\le C\int _1^\infty X^{-q+r}dX 
\end{align*}
and  $I _{ 2, 1 }$ finite if $q-r>1$. In the integral  $I _{ 2, 2 }$,
\begin{align*}
|F(Y)|\le CY^{-1-q},\,\,|F(|X-Y|)|\le CY^{r-1},\,\,\,|F(X)|\le CX^{-1-q}.
\end{align*}
from where  $I _{ 2, 2 }<\infty$ and the property $\Phi \in L^1((0, \infty)\times (0, \infty))$ is proved. Now, (\ref{Mth2.0E3})
follows  from the symmetry properties of the collision integral $\widetilde Q(F, F)$.
\section{Appendix}
\setcounter{equation}{0}
\setcounter{theo}{0}
\subsection{From Nordheim equation to system (\ref{S1E00})--(\ref{S1E01}). }
\label{deduction}
The Nordheim equation was obtained by L. W. Nordheim (cf \cite{No}). It describes the particle's density $G$ of a non condensed, weakly interacting, gas of bosons. In the simplest situation of a spatially 
homogeneous, isotropic gas of bosons, it reads  after some adimensionalization  as  follows,
\begin{align}
&\frac {\partial G (\tau , X)} {\partial \tau }=C _{ q }(G(\tau ), G(\tau ))(X),\,\,\forall \tau >0,\,X>0,\label{SANoE}\\
&C _{ q }(G, G)(X)=\int  _{ D(X) }\frac{\min (\sqrt X, \sqrt {X_1}, \sqrt {X_2}, \sqrt {X_3})}{\sqrt X}q(G) dX_2 dX_3\nonumber \\
&q(G)(X)=G(X_2)G(X_3)(1+G(X))(1+G(X_1))-G(X)G(X_1)(1+G(X2))(1+G(X_3)) \nonumber\\
&D(X)=\left\{(X_2, X_3)\in \RR_+^2;\,\,X_2+X_3>X\right\},\,\,X_1\equiv X_2+X_3-X. \nonumber
\end{align}
If only cubic terms are kept in $q$, then $q(G)$ is replaced by
\begin{align*}
p(G)(X)=G(X_2)G(X_3)(G(X) +G(X_1))-G(X)G(X_1)(G(X2+G(X_3))
\end{align*}
and obtain  the corresponding equation, with the obvious notation for $C_p$,
\begin{align*}
&\frac {\partial G (\tau , X)} {\partial \tau }=C _{ p }(G(\tau ), G(\tau ))(X),\,\,\forall \tau >0,\,X>0.
\end{align*}
That is the kinetic wave equation describing  a nonlinear system of interacting  waves satisfying the cubic Schr\"odinger. \\
 Below the critical temperature,  correlations appear in the condensed Bose gas  between the superfluid 
component and the normal fluid part corresponding to the excitations.  As a consequence, in the hydrodynamic regime, a collision integral $C _{ 1, 2 }  $ describing $1\leftrightarrow 2$ splitting of an excitation into two others in the presence of the condensate is needed. A kinetic equation in a uniform Bose gas which includes these processes was obtained in references \cite{Kirkpatrick} and \cite{Eckern} based on a microscopic description of the Bose gas. However, as explained in \cite{SK1} for the operator $C _{ q }$, a formal  calculation gives the correct expression of the term $C _{ 1, 2 }$ in the limit where the condensate is very small. The same argument works with  the collision integral $C_p$. It  is enough to this end to suppose that the density $G$ may be split as $G(\tau, X )=\sigma (\tau, X )+F(\tau, X )$, where $\sigma (\tau, X )=n(\tau )\delta _0(X)$, $n(t)$ is a function, $\delta _0$ is the Dirac measure at the origin  and $F$ is a function, to plug this expression in the collision integral $C _{ 2, 2 }$ and perform the Dirac's delta integration with some care. Let us drop the $\tau $ variable that does not play any role here, and denote $F_k=F(X_k)$, $\alpha _k=n\delta _0(X_k)$,
\begin{align*}
p(G)=p(F) + p_1(F,n) + p_2(F,n)
\end{align*}
where, {deduction}
\begin{align}
p_1(F,n)&= \alpha _1 [  F_2 \, F_3 - F (F_2+F_3) ]+ \alpha _2 [ (F +F_1) \, F_3 - F \, F_1]+ \nonumber \\
&+ \alpha _3 [ (F +F_1) \, F_2 - F \, F_1 ]\nonumber \\
p_2(F,n)&= \alpha   \left[F_2 F_3 - F_1 (F_2+F_3)\right].\label{S10.1Ep2}
\end{align}
One may plug  these expressions in the collision integral $C_p(G)$ and, using the  symmetry of the resulting integrals with respect to the indices $2$ and $3$, deduce the system,
\begin{align*}
&\frac {\partial F} {\partial \tau }(\tau , X)=C_p(F(\tau ))(X)+n(\tau )\widetilde Q(F(\tau ), F(\tau ))(X)\\
&n'(\tau )=-n(\tau )\int _0^\infty \widetilde Q(F(\tau ), F(\tau ))(X)\, \sqrt XdX.
\end{align*}
System  (\ref{S1E00})--(\ref{S1E01}) then follows after dropping  $C_p(F)$,  neglecting in that way the interactions involving only particles in the normal gas. Following the logic of our notation, one could have denoted $\tilde Q$ as $\tilde  Q_p$. The collision integral that would be denoted $\tilde Q_q$  is the one considered for example in \cite{CE}.
\begin{align}
&\tilde Q_q(F)(X)=\frac {1} {\sqrt X}\int_0^X \Big(f(X-Y)f(Y)-f(X)\big[1+f(X-Y)+f(Y)\big]\Big)dY +\nonumber\\
&\hskip 1cm +\frac {2} {\sqrt X}\int_X^\infty \Big(f(Y)\big[1+f(Y-X)+f(x)\big]-f(Y-X) f(x)\Big) dY.\label{PR2}
\end{align}

\subsection{Deduction of equation (\ref{S1E7}).}

It is suitable  to write the operator $\widetilde Q$ in   equation (\ref{S1E00}) or (\ref{lzm75})  in a more suitable way. 
In the first integral at the right hand side of (\ref{S1E01}), the integrand is invariant through the change of variables $Z=X-Y$ and may be written as
\begin{align}
\label{S1E02}
\int _0^X\Big(F(X-Y)F(Y)&-F(X)(F(X-Y)+F(Y))\Big)dY=\nonumber\\
&2\int  _{ X/2 }^X\Big(F(X-Y)F(Y)-F(X)(F(X-Y)+F(Y))\Big)dY
\end{align}
and  the operator $\widetilde Q$ may be written,
\begin{align}
\widetilde Q(F, F)&=\frac {2} {X^{1/2}}\int _{X/2}^\infty\Big(F(|X-Y|)(F(Y)-F(X))+\text{sign}(Y-X)F(X)F(Y))\Big)dY.\label{S1E03}
\end{align}

The  solution $f$  to (\ref{lzm75})  is looked for under the form,
\begin{align}
\label{S1E04}
f(t, X)=\frac {V (t, X)} {X}
\end{align}
for which,
\begin{align}
\widetilde Q(f, f)=2X^{-1/2}\int _{X/2}^\infty\Big(\frac {V(|X-Y|)} {|X-Y|}\left(\frac {V(Y)} {Y}-\frac {V(X)} {X} \right)+\nonumber \\
+\text{sign}(Y-X)\frac {V(X)V(Y)} {XY}\Big)dY. \label{S1EinnD}
\end{align}
The function $V$ is now split as
\begin{align}
\label{S1EinnD2b}
V(t, X)=j (t, X)+g(t, X)
\end{align}
where $j$ and $g$ play very different roles in what follows. Then,
\begin{align*}
\widetilde Q(f, f)&=\widetilde Q(j,j)+\widetilde Q_L(j, g)+\widetilde Q(g, g)\\
\widetilde Q_L(j, g)&=2X^{-1/2}\int _{X/2}^\infty\Bigg[\frac {j (t, |X-Y|)} {|X-Y|}
\left(\frac {g(t, Y)} {Y}-\frac { g(t, X)} {X} \right)+\\
&+\frac { g(t, |X-Y|)} {|X-Y|} \left(\frac {j  (t, Y)} {Y}-\frac { j  (t, X)} {X} \right)+\\
&+\text{sign}(Y-X)\frac {j  (t, X)g(t, Y)+j  (t, Y)g(t, X)} {XY}\Bigg]dY
\end{align*}

The term $\widetilde Q_L$ is now rewritten as follows. Use first,
\begin{align*}
\frac {g(t, Y)} {Y}-\frac { g(t, X)} {X} =\frac {g(t, Y)-g(t, X)} {X}+g(t, Y) \frac {X-Y} {XY}
\end{align*}
to obtain
\begin{align*}
{\frac {j (t, |X-Y|)} {|X-Y|}\left(\frac {g(t, Y)} {Y}-\frac { g(t, X)} {X} \right)}=\frac {j (t, |X-Y|)(g(t, Y)-g(t, X))} {|X-Y|X}+\\
+\text{sign}(X-Y)\frac {j (t, |X-Y|)g(t, Y)} {XY}
\end{align*}
and then
\begin{align*}
\widetilde Q_L(j, g)=2X^{-1/2}\int _{X/2}^\infty\Bigg( \frac {j (t, |X-Y|)(g(t, Y)-g(t, X))} {|X-Y|X}+\\
+\textcolor{blue}{\text{sign}(X-Y)\frac {j (t, |X-Y|)g(t, Y)} {XY}}+\\
+\frac { g(t, |X-Y|)} {|X-Y|} \left(\frac {j  (t, Y)} {Y}-\frac { j  (t, X)} {X} \right)+\\
+\text{sign}(Y-X)\frac {\textcolor{blue}{j  (t, X)g(t, Y)}+j  (t, Y)g(t, X)} {XY}\Bigg)dY
\end{align*}
Piece then together the blue terms to obtain,
\begin{align*}
\widetilde Q_L(j, g)=2X^{-1/2}\int _{X/2}^\infty\Bigg( \frac {j (t, |X-Y|)(g(t, Y)-g(t, X))} {|X-Y|X}+\\
+\textcolor{blue}{\frac {\text{sign}(X-Y)g(t, Y)} {XY}\left(j (t, |X-Y|)-j (t, X) \right)
}+\\
+\textcolor{magenta}{\frac { g(t, |X-Y|)} {|X-Y|} \left(\frac {j  (t, Y)} {Y}-\frac { j  (t, X)} {X} \right)}+\\
+\text{sign}(Y-X)\frac {j  (t, Y)g(t, X)} {XY}\Bigg)dY.
\end{align*}
Similarly, write now in the magenta term,
\begin{align*}
& \frac {j  (t, Y)} {Y}-\frac { j  (t, X)} {X} =j (t, Y)\left(\frac {1} {Y}-\frac {1} {X} \right)+\frac {j (t, Y)-j (t, X)} {X}\Longrightarrow\\
&\frac {g(t, |X-Y|)} {|X-Y|}\left(\frac {j  (t, Y)} {Y}-\frac { j  (t, X)} {X} \right)=\frac {\text{sign}(X-Y)j (t, Y)g(t, |X-Y|)} {XY}+\\
&\hskip 6cm +\frac {(j (t, Y)-j (t, X))g(t, |X-Y|)} {X|X-Y|}
\end{align*}
and then
\begin{align}
\widetilde Q_L(j, g)=2X^{-1/2}\int _{X/2}^\infty\Bigg( \frac {j (t, |X-Y|)(g(t, Y)-g(t, X))} {|X-Y|X}+\nonumber\\
+ \frac {\text{sign}(X-Y)g(t, Y)} {XY}\left(j (t, |X-Y|)-j (t, X) \right)+\nonumber\\
+\frac {\text{sign}(X-Y)j (t, Y)(g(t, |X-Y|)-g(t, X))} {XY}+\nonumber\\
+\frac {(j (t, Y)-j (t, X))g(t, |X-Y|)} {X|X-Y|}\Bigg)dY.\label{mnst1}
\end{align}
The third term at the right hand side of (\ref{mnst1}) may be written as follows

\begin{align*}
&\int _{X/2}^\infty\frac {\text{sign}(X-Y)j (t, Y)(g(t, |X-Y|)-g(t, X))} {XY}dY=\\
&\int _{X/2}^X \frac{ j (t, Y)(g(t, |X-Y|)-g(t, X))} {XY}dY
-\int _{X}^\infty\frac { j (t, Y)(g(t, |X-Y|)-g(t, X))} {XY}dY\\
&\hskip 4.5cm =\int _{0}^{X/2} \frac{ j (t, X-Y)(g(t, X-Y)-g(t, X))} {X(X-Y)}dY\\
&\hskip 5.8cm-\int _{0}^\infty\frac { j (t, X+Y)(g(t, X+Y)-g(t, X))} {X(X+Y)}dY
\end{align*}
from where,
\begin{align*}
\widetilde Q_L(j, g)
=2 X^{-3/2}\int _{0}^\infty\Bigg( (g(t, Y)-g(t, X))\left(\frac {j (t, |X-Y|)} {|X-Y|}-\frac {j (t, |X+Y|)} {|X+Y|}\right)+\\
+ \frac {\text{sign}(X-Y)g(t, Y)} { Y}\left(j (t, |X-Y|)-j (t, X) \right)+\\
+\frac {(j (t, Y)-j (t, X))g(t, |X-Y|)} { |X-Y|}\Bigg)dY.
\end{align*}
As the last step of this long transformation of the linear term  $\widetilde Q$, denote for $j\equiv 1$,
\begin{align*}
\widetilde Q_L(1, g)=2 X^{-1}\mathscr L(g(t))(X)
\end{align*}
where $\mathscr L$ is defined in (\ref{S1InnD8}).
The linear term $\widetilde Q_L(j, g)$ may then finally  be written,
\begin{align*}
&\widetilde Q_L(j, g)=2 X^{-1}\mathscr L(g(t))(X)+\\
&2 X^{-3/2}\int _{0}^\infty\Bigg[ (g(t, Y)-g(t, X))\left(\frac {j (t, |X-Y|)-1} {|X-Y|}-\frac {j (t, |X+Y|)-1} {|X+Y|}\right)+\\
&+ \frac {\text{sign}(X-Y)g(t, Y)} { Y}\left(j (t, |X-Y|)-j (t, X) \right)+\\
&\hskip 6cm +\frac {(j (t, Y)-j (t, X))g(t, |X-Y|)} { |X-Y|}\Bigg]dY.
\end{align*}

On the other hand, in the nonlinear term  $\widetilde Q_(g, g)$,
\begin{align*}
\widetilde Q(g, g)=2X^{-1/2}\int _{X/2}^\infty\Bigg(\frac { g(t, |X-Y|)} {|X-Y|} \left(\frac { g(t, Y)} {Y}-\frac { g(t, X)} {X} \right)+\\
+\text{sign}(Y-X)\frac { g(t, X)  g(t, Y)} {XY}\Bigg)dY
\end{align*}
write,

\begin{align*}
\frac { g(t, Y)} {Y}-\frac { g(t, X)} {X} =\frac {g(t, Y)-g(t, X)} {Y}&+g(t, X)\left(\frac {1} {Y}-\frac {1} {X} \right)\Longrightarrow\\
\frac { g(t, |X-Y|)} {|X-Y|} \left(\frac { g(t, Y)} {Y}-\frac { g(t, X)} {X} \right)=&
\frac { g(t, |X-Y|)(g(t, Y)-g(t, X))} {|X-Y|Y}+&\\
&+ \text{sign}(X-Y)\frac { g(t, |X-Y|)g(t, X)} { XY}
\end{align*}
and then,
\begin{align*}
\widetilde Q(g, g)&=2X^{-1/2}\int _{X/2}^\infty\Bigg(\frac { g(t, |X-Y|)(g(t, Y)-g(t, X))} {|X-Y|Y}+\\
&+\text{sign}(X-Y)\frac { g(t, X)} {XY}\left( g(t, |X-Y|)-g(t, Y)\right)\Bigg)dY.
\end{align*}
In terms of the new functions $j $ and $g$ the equation (\ref{S1E00}) reads,
\begin{align}
\label{S1E6}
\frac {\partial f} {\partial t}&=X^{-1}\left(\frac {\partial j } {\partial t}+\frac {\partial g} {\partial t}\right)=\nonumber\\
&=\widetilde Q(f, f)=\left(\widetilde Q(j, j)+\widetilde Q_L(j, g)+\widetilde Q(g, g)\right)
\end{align}
If the function $j$ is now  
\begin{align}
\label{SAInnD9}
j (t, X)=\lambda (t)\varphi (X)
\end{align}
with  $\varphi \in C_c^1(0,\infty)$ as in (\ref{defvfiR})--(\ref{deffi2}), the function $f$  given by (\ref{S1E04}) and (\ref{S1EinnD2}) becomes (\ref{S1EinnD2}), with initial data $f(0, X)=\frac {\varphi (X)+g(0, X)} {X}$. If the constant $R>0$ in (\ref{defvfiR}) is large enough and $g(0)\equiv 0$,
then $f(0, X)$ is  a small perturbation of the function $X^{-1}$.\\

For   a function $j $ as in (\ref{SAInnD9}), the  terms $\widetilde Q(j)$ and  $\widetilde Q_L(j, g)$  in (\ref{S1E6}) become:
\begin{align}
\widetilde Q(j, j)=\lambda ^2(t)X^{-1}Q_N(\varphi ),\,\,\,\text{with}\,\,\,Q_N(\varphi )=X\widetilde Q(\varphi, \varphi  ) \label{S1E6Y1} \\
\widetilde Q_L(j, g)=2\lambda (t)X^{-1}\left(\mathscr L(g)+T_1(g, \varphi )+T_2(g, \varphi ) \right) \label{S1E6Y2}
\end{align}
where $Q_N$ and $\mathscr L_\varphi $ are defined in, (\ref{S1InnD6}) and (\ref{S1InnD7})

In these new variables, the equation   (\ref{S1E6}) gives equation (\ref{S1E7}). 

\subsection{Some integrals}
\begin{prop}
\label{P52}
$\Omega  _{ 2\theta }\in L^1(\tilde D_1\cup D_2), \forall \theta\in (0, 1/2)$, where $D_2$ is defined in (\ref{S9D2}) and
\begin{align}
\tilde D_1=\left\{(X, Y)\in \RR^2;\, 0<X<d,\,0<Y<2X \right\}\label{S9DT1}
\end{align}
(Of course, by (\ref{S9D1}),  $D_1\subset \tilde D_1$).
\end{prop}
\begin{proof} 
Consider first the integral on $\tilde D_1$. 
For $0<a<1$, $ b>0$ and all $\xi \in (0, 2)$
\begin{align}
&\int _0^{2\xi}  |\left(\zeta ^{-b}-\xi ^{-b}\right)\zeta ^{b-a}(\xi -\zeta )^{-1}|d\zeta \le C(a, b)\xi ^{-a} \label{P52E1}\\
&C(a, b)= B\left[2, 1-a, 0 \right]+B\left[2, 1-a+b, 0 \right]+2(|H(b-a)|+|H(-a)|),\nonumber
\end{align}
and
\begin{align}
&\int _0^\xi  \left(\xi ^{b-a}-\zeta ^{b-a}\right)\xi ^{-b}(\xi -\zeta )^{-1}d\zeta \le C'(b-a)\xi ^{-a} \label{P52E2}\\
&C'[z]=|H(z)|+|B(2, 1+z, 0)|+|H(z)|
\end{align}
where $B$ denotes here  the incomplete Beta function and $H(\cdot)$ denotes the Harmonic Number function. Integrability of  $\Omega  _{ 2\theta }$  on $\tilde D_1$ follows since the right hand sides of (\ref{P52E1}) and  (\ref{P52E2}) are integrable on $(0, d)$. 

On the other hand, by Fubini's Theorem,
\begin{align*}
\int _1^2d\xi \int  _{ \xi -1 }^1d\zeta =\int _0^1d\zeta \int_1^{1+\zeta }d\xi.
\end{align*}
Now, for $0<a<1$, $ b>0$ and all $\zeta  \in (0, 1)$
\begin{align*}
&\int _1^{1+\zeta }  \left(\zeta ^{-b}-\xi ^{-b}\right)\zeta ^{b-a}(\xi -\zeta )^{-1}d\xi  =C_1(\zeta ; a, b)\zeta  ^{-a}\\
&C_1(\zeta ; a, b)=B\left[\frac {\zeta } {1+\zeta }, b, 0\right]-B[\zeta , b, 0]-Log(1-\zeta ),\\
&\int _1^{1+\zeta }  \left(\xi  ^{b-a}-\zeta  ^{b-a}\right)\xi  ^{-b}(\xi -\zeta )^{-1}d\xi  =C_2(\zeta ; a, b)\zeta  ^{-a}\\
&C_2(\zeta ; a, b)=-B[\zeta , a, 0]+B[\zeta , b, 0]-B\left[\frac {\zeta } {1+\zeta }, a, 0\right]+B\left[\frac {\zeta } {1+\zeta }, b, 0\right]
\end{align*}
where $B$ is the incomplete Beta functions. There exists then a constant $C>0$ depending on $a$ and $b$ such that
\begin{align}
&|C_1(\zeta ; a, b)|\le C\left(\zeta+\zeta ^{1+b}\right) ,\,\,\forall \zeta \in [0,1], \label{P52E3}\\
&|C_2(\zeta ; a, b)|\le C\left(\zeta ^{1+a}+\zeta ^{1+b}\right),\,\,\forall \zeta \in [0,1], \label{P52E4}
\end{align}
and  the right hand sides of (\ref{P52E3}) and (\ref{P52E4}) are integrable on $(0, 1)$.
\vskip -0.5cm
\end{proof}

\subsection{The function $B$}
\label{FB}
For any $\beta \in (0, 1)$ and $s\in \CC$ such that $\Re e(s)\in (\beta, \beta +1)$ define
\begin{align}
\label{S3PBE0}
B(s)&=\exp\left(\int  _{ {\mathscr Re} (\rho) =\beta  }  \log (-W(\rho ))\left( \frac {1} {1-e^{2i\pi (s-\rho) } }-\frac {1} {1+e^{-2i\pi \rho }}\right)d\rho\right)\\
W(\rho )&=-2\gamma_e -2\psi \left(\frac {\rho } {2} \right)-\pi \cot \left(\frac {\pi \rho } {4} \right),\,\,\,\forall \rho \in \CC,\, \Re e(\rho )\in -2, 4 ) \label{S4E7}
\end{align}
where $\gamma _e$ is the Euler constant and $\psi (z)=\Gamma '(z)/\Gamma (z)$ is the Digamma function.  The following Proposition was proved in \cite{m} (Proposition 2.4).
\begin{prop}
\label{S3PBP}
The function $B$ is  analytic on the domain $s\in \CC,\,\Re e (s)\in \mathcal S _{\, 0, 2 }$ where it is given by the integral in (\ref{S3PBE0}) for some $\beta \in (0, 1)$ such that $\beta <\mathscr Re s<\beta +1$. It is extended to a meromorphic on the complex plane by the following relation, 
\begin{equation}
B(s)=-W(s-1)B(s-1),\forall s\in \CC. \label{S5EBX5}
\end{equation}
It has a sequence of poles and a sequence of zeros, determined by the zeros and poles of the function $W$ as follows.\\
1.-Poles. The poles of the function $B$  are located at $s=0$, $s=-1$,  at $\{4n+1, n=1, 2, 3, \cdots\}$ and at  $\{\sigma _n^*,\,\,n=1, 2, 3, \cdots \}$.\\
2.-Zeros. The zeros of the function $B$ are at $s=3$, $s=4$ at $\{-n, \,\,\,n=6, 7, 8, \cdots\}$  and at $\{\sigma _n+1,\,\,\,n=1, 2, \cdots\}$.
\end{prop}

\noindent
\textbf{Acknowledgments.}
The research of the author is supported by grant PID2020-112617GB-C21 of MCIN, grant  IT1247-19 of the Basque Government and grant RED2022-134784-T funded by MCIN/AEI/10.13039/501100011033.\\
The author is  grateful to the anonymous referees for their accurate comments and indications.\\
\textbf{Conflicts of interests.} The author declares that there were no conflicts of interest related to this work. \\
\textbf{Data Availability.} No datasets were generated or analysed during the current study.

\end{document}